\documentclass[a4paper,10pt]{amsart}
\usepackage[T1]{fontenc}
\usepackage{pxfonts}
\usepackage{shadethm,framed}

\usepackage[margin=3cm]{geometry}
\usepackage{enumerate}  
\usepackage{enumitem}
\usepackage{listings}
\usepackage{amsthm, amsfonts, amssymb, amsmath,mathbbol}
\usepackage{graphicx}
\usepackage{xcolor}
\usepackage{color}
\usepackage{framed}

\usepackage{tikz,tikz-cd}
\usepackage{pgf,tikz,pgfplots}
\usepackage{tcolorbox}
\usepackage{scalefnt}
\usepackage{accents}

\usepackage{caption}
\usepackage{makecell}
\usepackage{float}
\usepackage{thmtools}
\usepackage{thm-restate}

\usepackage{makecell}

\pgfplotsset{compat=1.18}

\usepgfplotslibrary{patchplots}
\usetikzlibrary{patterns, positioning, arrows,decorations.markings}

\definecolor{rose}{rgb}{0.93, 0.23, 0.51}
\definecolor{ref}{rgb}{0.29, 0.59, 0.82}
\definecolor{bovert}{RGB}{200,255,200}
\definecolor{lightblue}{RGB}{137, 207, 240}
\usepackage[hidelinks,urlcolor=blue]{hyperref}
\hypersetup{colorlinks,linkcolor=black,citecolor=rose,urlcolor={blue!80!black}}

\newtheorem{lem}{Lemma}[section]
\newtheorem{cor}[lem]{Corollary}
\newtheorem{thm}[lem]{Theorem}
\newtheorem{prop}[lem]{Proposition}

\newtheorem{fact}[lem]{Fact}

\usepackage{cleveref}

\declaretheorem[name=Theorem,numberwithin=section]{theorem}
\declaretheorem[name=Corollary,numberwithin=section]{corollary}

\theoremstyle{definition}
\newtheorem{definition}[lem]{Definition}
\newtheorem*{definition*}{Definition}
\newtheorem{ex}[lem]{Example}
\newtheorem{examples}[lem]{Examples}

\newtheorem{rmk}[lem]{Remark}
\newtheorem{question}[lem]{Question}
\newtheorem*{question*}{Question}
\newtheorem*{warning*}{Warning to the reader}

\numberwithin{equation}{section}

\newcommand{\N}{\mathbf N}
\newcommand{\Z}{\mathbf Z}

\newcommand{\R}{\mathbf R}

\newcommand{\g}{\mathfrak{g}}
\newcommand{\p}{\mathfrak{p}}
\newcommand{\RP}{\mathbf{RP}}
\newcommand{\Ein}{\mathbf{Ein}}
\newcommand{\Eintilde}{\smash{\widetilde{\mathbf{Ein}}}^{1,n}}

\newcommand{\dS}{\mathbf{dS}}
\newcommand{\D}{\mathbf D}

\newcommand{\B}{\mathbf B}

\newcommand{\C}{\mathbf C}
\renewcommand{\S}{\mathbf S}
\renewcommand{\H}{\mathbf H}
\renewcommand{\P}{\mathbf P}

\newcommand{\hol}{\mathsf{hol}}

\newcommand{\dev}{\mathsf{dev}}

\newcommand{\dc}{\mathsf{cd}}

\newcommand{\ghyp}{\mathsf{g}_{\operatorname{hyp}}}

\newcommand{\Conv}{\operatorname{Conv}}

\newcommand{\kob}{\mathsf{Kob}}

\newcommand{\Hcal}{\mathcal{H}}
\newcommand{\Conf}{\mathrm{Conf}}
\newcommand{\Sim}{\mathrm{Sim}}
\newcommand{\Span}{\mathrm{Span}}
\newcommand{\Iso}{\mathrm{Isom}}  

\newcommand{\PPL}{\mathsf{PPL}}
\newcommand{\Ric}{\mathsf{Ric}}

\renewcommand{\qed}{$\hfill\blacklozenge$}
\renewcommand{\b}{\mathbf{b}}

\newcommand{\PO}{\operatorname{PO}}

\newcommand{\PSL}{\operatorname{PSL}}

\renewcommand{\O}{\operatorname{O}}

\newcommand{\W}{\mathcal{W}}
\newcommand{\dvol}{d\text{Vol}}
\newcommand{\Ccal}{\mathcal{C}}
\newcommand{\Bcal}{\mathcal{B}}
\newcommand{\Kcal}{\mathcal{K}}
\newcommand{\Mtilde}{\smash{\widetilde{M}}}
\newcommand{\RPunTilde}{\smash{\widetilde{\RP}}^1}
\newcommand{\Mhat}{\smash{\hat{M}}}
\newcommand{\st}{\text{st}}
\title[On Markowitz's pseudodistance]{On Markowitz's pseudodistance for conformal manifolds}

\author{Adam Chalumeau}
\date{}

\begin{document}

\begin{abstract} 
    In the 1980s, M. J. Markowitz introduced a conformally invariant pseudodistance on pseudo-Riemannian manifolds, inspired by the Kobayashi metric in projective geometry. This construction relies on a distinguished class of parametrized lightlike geodesics, called projectively parametrized. We begin by reviewing the fundamental properties of this pseudodistance and provide several families of examples where it is non-degenerate and, in some cases, complete. In particular, we investigate three classes of manifolds: closed manifolds, conformally convex domains of the Einstein universe, and globally hyperbolic, conformally flat, $C$-maximal spacetimes. For the first two classes, we obtain results analogous to those of Brody and Barth concerning the complex Kobayashi metric. Finally, we apply Markowitz’s pseudodistance to classify all quasi-homogeneous domains of the Einstein-de Sitter space, that is, a half-space of the Minkowski space bounded by a spacelike hyperplane. Up to conformal transformations, only finitely many such domains exist, and all of them turn out to be homogeneous.
\end{abstract}

\maketitle


\section{Introduction}
\sloppy

\renewcommand{\thetheorem}{\Alph{theorem}}
\renewcommand{\thecorollary}{\Alph{corollary}}
\setcounter{theorem}{0}
\setcounter{corollary}{5}

In complex geometry, the Kobayashi metric, defined via chains of holomorphic discs \cite{Kob_complexe}, is a standard tool for studying complex manifolds. This metric admits several analogs in other geometric settings, usually referred to as Kobayashi--Royden metrics. For instance, in real projective geometry, Kobayashi introduced a projectively invariant distance based on chains of projective geodesics \cite{kob_projectif}. This construction is most prominent in the case of proper convex domains of real projective space, where it coincides with the classical Hilbert metric. Further extensions include Kobayashi--Royden-like constructions in the flat conformal Riemannian setting \cite{KP,Apanasov}, and more recently in both flat and non-flat Riemannian geometries \cite{Drinovec_Drnov_ek_2023, gaussier2023kobayashimetricsriemannianmanifolds}, and on some locally homogeneous manifolds \cite{vanLimbeek_Zimmer,Blandine,these_blandine}.

The present paper investigates a related pseudodistance in the framework of conformal pseudo-Riemannian manifolds $(M,[g])$, where $M$ is a smooth manifold of dimension $n\geq 3$, and $[g]$ is a conformal class of pseudo-Riemannian metrics on $M$. Such manifolds naturally carry a family of parametrized lightlike (i.e. null) geodesics, invariant under projective reparametrizations, called \emph{projectively parametrized lightlike geodesics}.
Building on this, Markowitz \cite{markowitz_1981} introduced on every conformal manifold $(M,[g])$ a Kobayashi--Royden-like pseudodistance $\delta_M$, meaning that $\delta_M$ is symmetric and satisfies the triangle inequality, though it may fail to be definite. The construction proceeds as follows. Let $I=(-1,1)$ denote the unit interval in $\R$, equipped with the Poincaré distance $d_I$ associated to the metric $\frac{4dx^2}{(1-x^2)^2}$. Then $\delta_M$ is defined to be the largest pseudodistance on $M$ with the property that every projectively parametrized lightlike geodesic $(I,d_I)\to (M,\delta_M)$ is distance-decreasing, see Section~\ref{section_delta_M} for an explicit definition.

\begin{definition*}
    A conformal manifold $(M,[g])$ is \emph{Markowitz hyperbolic} if the Markowitz pseudodistance $\delta_M$ is a distance.
\end{definition*}

Part of the interest in Markowitz hyperbolicity lies in the fact that conformal maps are 1-Lipschitz for the Markowitz pseudodistance. In particular, the conformal group $\Conf(M)$ of any Markowitz hyperbolic conformal manifold $M$ acts properly on $M$, a property that does not hold for general conformal manifolds. More generally, for any Markowitz hyperbolic conformal manifolds $M$ and $N$, one can explicitly exhibit compact subsets of $\Conf(M,N)$, see Proposition~\ref{prop_compacité_app_conformes}. Hence, Markowitz hyperbolicity provides a powerful tool for understanding the dynamics of conformal transformations.

Markowitz has given sufficient conditions, in terms of the Ricci curvature, for a spacetime to be Markowitz hyperbolic \cite{markowitz_1981}, and he also computed the pseudodistance $\delta_M$ for the Einstein-de Sitter space, see \cite{Markowitz_warped_product}. Since then, Markowitz hyperbolicity has not received a lot of attention from the community until recently, when B. Galiay and the author \cite{Cha_Gal} studied the Markowitz pseudodistance for bounded domains of the Minkowski space $\R^{p,q}$. 

The first step in our study is to provide a general characterization of Markowitz hyperbolicity in terms of the infinitesimal functional $F_M$ -an infinitesimal version of $\delta_M$ defined on the bundle $C(TM)$ of lightlike tangent vectors, see Section~\ref{section_FM}- and in terms of the equicontinuity of the family of projectively parametrized lightlike geodesics, in close analogy with Royden's characterization in the complex setting \cite{Royden} (see Theorem~\ref{thm_hyp} and Theorem~\ref{thm_hyp_conf_plat}). 

\begin{theorem}
    \label{thm_A}
    Let $(M,[g])$ be a conformal manifold. Then the following properties are equivalent:
    \begin{enumerate}
        \item $M$ is Markowitz hyperbolic, that is, $\delta_M$ is a distance.
        \item $\delta_M$ induces the standard topology on $M$.
        \item $\delta_M$ is locally equivalent to a (hence any) Riemannian distance.
        \item $F_M$ is locally bounded below by a Riemannian norm restricted to $C(TM)$.
        \item The family $\PPL(I,M)$ of projectively parametrized lightlike geodesics $\gamma:I\to M$ is equicontinuous, with respect to some distance on $M$ that is locally equivalent to a Riemannian distance. 
        \item (If $M$ is conformally flat) The family $\Conf(\D^{p,q},M)$ of conformal maps $\D^{p,q}\to M$ is equicontinuous, with respect to some distance on $M$ that is locally equivalent to a Riemannian distance.
    \end{enumerate}
    In that case, the Markowitz distance $\delta_M$ is a length metric on $M$.
\end{theorem}

In the last point, $\D^{p,q}$ refers to the pseudo-Riemannian \emph{diamond}, which is defined as $\D^{p,q}=\H^p\times\H^q$ endowed with the conformal class of the product metric $g=-g_{\H^p}+g_{\H^q}$. In addition to Theorem~\ref{thm_A}, we provide necessary conditions for the completeness of the Markowitz pseudodistance in terms of the infinitesimal functional, see Proposition~\ref{prop_cara_complete_hyperbolic}.

\subsection{Geometric characterization of Markowitz hyperbolicity}
Even though characterizations of Markowitz hyperbolicity have been given in Theorem~\ref{thm_A}, they are usually difficult to verify in practice. We now provide a simpler geometric property that is related to Markowitz hyperbolicity.
\begin{definition*}
    A conformal manifold $(M,[g])$ is \emph{totally conformally lightlike incomplete} if every projectively parametrized lightlike geodesic $\R\to M$ is constant.
\end{definition*}
For a conformal manifold $(M,[g])$, any two points on the image of some $\gamma\in\PPL(\R,M)$ are at zero Markowitz pseudodistance from one another. Hence, when $(M,[g])$ is Markowitz hyperbolic, such $\gamma$ must be constant. In other words, one has the implication
\begin{center}
    Markowitz hyperbolic $\Rightarrow$ totally conformally lightlike incomplete,
\end{center}
but the converse implication need not hold in general, see Figure~\ref{Omega_slip_non_Mhyp} in Section \ref{Section_total_conf_lightlike_incompletness}. The next theorems give general frameworks where the converse implication actually holds. 
\subsubsection{Compact case}
In complex geometry, Brody \cite{Brody} first proved that closed complex manifolds are Kobayashi-hyperbolic if and only if their entire curves are constant. Kobayashi and Wu gave an analogous statement in the projective case, in the projectively flat and non-projectively flat case, respectively \cite{kob_projectif,Wu}.
For the Markowitz pseudodistance, we obtain the following characterization (see Theorem~\ref{thm_compacité_fort}).

\begin{theorem}
    \label{thm_B1}
    Let $(M,[g])$ be a closed conformal manifold. If $M$ is totally conformally lightlike incomplete, then there exists a neighborhood $\W$ of $[g]$ for the $C^2$ topology, such that $(M,[g^\prime])$ is Markowitz hyperbolic for every $[g^\prime]\in \W$. In particular, $(M,[g])$ is Markowitz hyperbolic.
\end{theorem}

In particular, the property ``$(M,[g])$ is Markowitz hyperbolic'' is open with respect to the $C^2$ topology on $[g]$. For real projective manifolds, Wu \cite{Wu} shows the equivalence between total incompleteness and Kobayashi hyperbolicity, but requires the completeness of an affine connection in the projective class. Here we need no additional assumption. The proof consists in using the standard reparametrization lemma of Brody and to show that a uniform limit of projectively parametrized lightlike geodesics is always projectively parametrized, and that the convergence is smooth. This requires a detailed study of the uniform and smooth topologies on $\PPL(I,M)$, see Section~\ref{Section_topologies_PPL}. 

\subsubsection{Conformally convex case}
The second class of conformal manifolds we study is the analog of convex domains of $\C^n$ or $\RP^n$. In complex geometry, Barth showed that a convex domain of $\C^n$ containing no complex line carries a complete Kobayashi metric \cite{Barth}. Similarly, in the projective setting, a properly convex domains of $\RP^n$ containing no line has a complete Kobayashi metric, and projective segments are geodesic of that metric \cite{kob_projectif}.

For the Markowitz pseudodistance, we study domains of the model space of conformal pseudo-
Riemannian geometry, called the \emph{Einstein universe}, and denoted $\Ein^{p,q}$. This space is the pseudo-Riemannian analog of the conformal sphere $\S^n$.
We say that a domain $\Omega \subset \Ein^{p,q}$ is \emph{conformally convex} if every $x\in\partial\Omega$ admits a supporting Möbius hyperplane, see Definition~\ref{definition_conf_conv}. For instance, any convex domain in one of the standard pseudo-Riemannian models of constant curvature -namely $\R^{p,q}$ for zero curvature, $\H^{p,q}$ for negative curvature, or $\dS^{p,q}$ for positive curvature- is conformally equivalent to a conformally convex domain in $\Ein^{p,q}$, see Example~\ref{exemple_de_conf_conv}.
Moreover, the class of conformally convex domains also contains dually convex domains of $\Ein^{p,q}$, which were previously introduced and studied by Zimmer in the broader setting of flag manifolds \cite{Zimpropqh}.
We obtain the following characterization for conformally convex domains (see Theorem~\ref{thm_conformement_convex}).

\begin{theorem}
\label{thm_B2}
    Let $\Omega\subset\Ein^{p,q}$ be a conformally convex domain. Then $\Omega$ is Markowitz hyperbolic if and only if $\Omega$ is totally conformally lightlike incomplete. In that case, the metric space $(\Omega,\delta_\Omega)$ is complete and every maximal $\gamma\in \PPL(I,\Omega)$ is an isometry $\gamma:(I,d_I)\to (\Omega,\delta_\Omega)$.
\end{theorem}

In particular, there exist many domains that carry a complete Markowitz distance, and these domains need not be convex in $\R^{p,q}$, nor simply connected, see for instance Figure~\ref{figure_conf_conv_domain}. This phenomenon contrasts with the projective Kobayashi metric, as domains of real projective space have a complete Kobayashi metric if and only if they are properly convex, see \cite[Prop. 12.2.3]{Goldman2022book}. It seems to be a hard problem to characterize the completeness of Markowitz distance of domains of $\Ein^{p,q}$, as there exists non conformally convex domains that have a complete Markowitz distance, see Proposition~\ref{prop_domain_non_conf_conv_but_complete}. 
Note that Theorem~\ref{thm_B2} was already known for dually convex domains that are moreover bounded in a suitable stereographic projection (and in particular totally conformally lightlike incomplete), see \cite[Lem. 4.7]{Cha_Gal}.

\subsubsection{Globally hyperbolic case} The last family of conformal manifolds that we consider is specific to Lorentzian signature $(1,n)$. We focus on \emph{globally hyperbolic spacetimes}, that is, Lorentzian manifolds admitting a topological hypersurface $\Sigma$ that intersects every inextensible causal curve exactly once. Such a hypersurface is called a \emph{Cauchy hypersurface}.  

Given two globally hyperbolic spacetimes $M$ and $N$, a \emph{Cauchy embedding} from $M$ to $N$ is a conformal map $f: M \to N$ that sends a Cauchy hypersurface of $M$ to a Cauchy hypersurface of $N$. A globally hyperbolic spacetime $M$ is said to be $C$-\emph{maximal} if every Cauchy embedding $f: M \to N$ is surjective.  
This notion was introduced by Rossi \cite{Clara} as a conformal analog of maximal spacetimes in Lorentzian geometry \cite{Choquet-Bruhat_Geroch}, and was later studied by Smaï, see \cite{SmaiThese, smai2023enveloping, smai2025futures}. Standard examples of $C$-maximal spacetimes include the universal cover of the Einstein universe and Minkowski spacetime.  

These spacetimes are of particular interest since Rossi proved that any globally hyperbolic conformally flat spacetime extends uniquely to a $C$-maximal one \cite{Clara, smai2023enveloping}. Therefore, in a certain sense, the study of conformally flat globally hyperbolic spacetimes reduces to the study of $C$-maximal ones. For these latter, we obtain the following characterization (see Theorem~\ref{thm_GHM}).

\begin{theorem}
\label{thm_B3}
    Let $M$ be a conformally flat, globally hyperbolic, $C$-maximal spacetime.
    Then $M$ is Markowitz hyperbolic if and only if $M$ is totally conformally lightlike incomplete.
\end{theorem}

As a consequence, if $M$ is a conformally flat, globally hyperbolic, $C$-maximal (abbreviated GHMC) spacetime, then $\Conf(M)$ acts properly on $M$ as soon as $M$ is totally conformally lightlike incomplete, see Corollary~\ref{corollary_properness_GHMC}. This contrasts with known examples of GHMC spacetimes that admit a non-constant lightlike geodesic $\R \to M$, for which the group of conformal transformations can act non-properly. We mention two important such examples:
\begin{itemize}
    \item Any simply connected GHMC spacetime for which there exists a non-constant $\gamma \in \PPL(J,M)$, with $J \subset \RPunTilde$ containing a copy of $\R$ as a proper subset, is conformally equivalent to $\Eintilde$ \cite{Clara}, and in particular admits non-proper conformal actions. 
    \item There exist other GHMC spacetimes, not totally conformally lightlike incomplete, that also admit non-proper conformal actions. For instance, the spacetime $M = \H^k \times \dS^{1,n-k}$ is GHMC with non-constant lightlike geodesics $\R \to M$, see \cite{smai2023enveloping}. This spacetime admits non-proper conformal actions: by the Calabi--Markus phenomenon \cite{Calabi_Markus}, any non-compact isometric action on the $\dS^{1,n-k}$ factor is non-proper.
\end{itemize}

\subsection{An application to conformal Lorentzian geometry}
Kobayashi--Royden metrics are famously used as tools for the study of the dynamics of automorphism groups.  
A classical application is the Wong--Rosay theorem \cite{Wong,Rosay}, which characterizes the unit ball of $\C^n$ in terms of its group of biholomorphisms. For domains of $\Ein^{p,q}$, a similar characterization was proven in \cite{Cha_Gal}, where the authors study domains of the pseudo-Riemannian Einstein universe that are \emph{proper}, i.e., bounded in $\R^{p,q}$ under a suitable stereographic projection, and \emph{quasi-homogeneous}, meaning that the conformal group acts with compact quotient. In the Lorentzian case, they show that causal diamonds are the only such domains of $\Ein^{1,n}$.

In this paper, we study quasi-homogeneous domains in the Lorentzian Einstein universe $\Ein^{1,n}$, without assuming properness. We focus on domains $\Omega$ contained in an \emph{Einstein-de Sitter half-space}, that is, a half-space of $\R^{1,n}$ bounded by a spacelike hyperplane. This setting arises naturally from the broader study of quasi-homogeneous Markowitz hyperbolic domains of $\Ein^{1,n}$. Indeed, the Einstein-de Sitter half-space is itself Markowitz hyperbolic, so its quasi-homogeneous domains form a distinguished subclass of quasi-homogeneous Markowitz hyperbolic domains of $\Ein^{1,n}$. We are thus led to ask:

\begin{question}
    \label{question_HBD}
    What are all quasi-homogeneous domains of the Einstein-de Sitter half-space ?
\end{question}

Clearly, since diamonds of $\R^{1,n}$ are bounded, they lie inside some Einstein-de Sitter half-space and hence satisfy the above assumption. However, besides diamonds, there exist other homogeneous domains contained in an Einstein-de Sitter half-space, which we call \emph{homogeneous Bonsante domains}. For $0 \leq \ell \leq n$, a \emph{future} (resp. \emph{past}) \emph{homogeneous Bonsante domain} (abbreviated HB-domain) of index $\ell$ is the future (resp. past) of a spacelike subspace $F_\ell$ of dimension $\ell$, see Figure~\ref{Figure_Misner} in Section \ref{section_misner}. For instance, when $\ell=0$, the associated HB-domain is conformally equivalent to a diamond. We name these domains after F. Bonsante, who studied convex future-complete domains of $\R^{1,n}$ disjoint from a spacelike hyperplane \cite{Bonsante}. 
We show that, up to a conformal transformation, they are the only solutions to Question~\ref{question_HBD}.

\begin{theorem}
    \label{thm_application}
    Let $\Omega \subset \R^{1,n}$ be contained in an Einstein-de Sitter half-space. If $\Omega$ is quasi-homogeneous, then $\Omega$ is either a homogeneous Bonsante domain or a diamond (possibly with endpoints in $\partial \R^{1,n}$).
\end{theorem}

In particular, any quasi-homogeneous domain of $\R^{1,n}$ disjoint from at least one spacelike hyperplane is homogeneous. In the conclusion of Theorem~\ref{thm_application}, the domain $\Omega$ is either an HB-domain or conformally equivalent to a diamond of $\R^{1,n}$. In the latter case, there are three possible configurations for $\Omega$, see Figure~\ref{Figure_Modèles_affines_des_diamants}.

We deduce a classification of closed conformally flat Lorentzian manifolds whose developing map takes values in an Einstein-de Sitter half-space.

\begin{corollary}
    \label{corollaire_de_lapplication}
    Let $M$ be a closed conformally flat Lorentzian manifold. If the developing image of $M$ is contained in an Einstein-de Sitter half-space, then $\smash{\widetilde M}$ is conformally equivalent to a diamond of $\R^{1,n}$. In particular, the manifold $M$ is finitely covered by a product 
    $(\S^1\times N, [-dt^2\oplus\ghyp]),$
    where $N$ is a closed Riemannian manifold of constant negative curvature. 
\end{corollary}

Corollary F extends, in the Lorentzian setting, a result of \cite{Cha_Gal}, who classified closed conformally flat manifolds with proper development in the Einstein universe of arbitrary signature. That earlier work relied on Zimmer’s techniques \cite{Zimpropqh}, specific to proper domains, and in particular on the fact that proper quasi-homogeneous domains are dually convex. By contrast, the present situation is fundamentally different: quasi-homogeneous domains in the Einstein-de Sitter half-space may not be dually convex, since the half-space itself is not. Consequently, the proof of Theorem~\ref{thm_application} requires a different strategy from that of \cite{Cha_Gal}.

\subsection{Further directions}
\label{section_questions}

Corollary~\ref{corollaire_de_lapplication} suggests that there should be only a few closed conformally flat Lorentzian manifolds that are Markowitz hyperbolic. In fact, the only known examples are those finitely covered by the diamond.
This naturally leads to the following question:

\begin{question}
Let $M$ be a closed conformally flat Lorentzian manifold. If $M$ is totally conformally lightlike incomplete, is $\smash{\widetilde M}$ necessarily conformally equivalent to a diamond $\D^{1,n}=\R\times\H^n$?
\end{question}

Corollary~\ref{corollaire_de_lapplication} suggests that the answer should be positive. Indeed, it seems unlikely that a closed manifold could be both totally conformally lightlike incomplete and have a developing image intersecting any domain of $\Ein^{1,n}$ that projects to an Einstein-de Sitter half-space under some stereographic projection. Moreover, there is a close analogy with the projective setting: any compact, totally incomplete projectively flat manifold is covered by a proper convex domain of $\RP^n$, see \cite[Prop.~12.2.3]{Goldman2022book}.

\subsection{Organization of the paper} 
After a preliminary Section~\ref{section_preliminary}, we introduce the projective parameters on lightlike geodesics and study the topology of $\PPL(I,M)$ in Section~\ref{section_para_proj}. Markowitz hyperbolicity is discussed in Section~\ref{section_hyp_Mark}, where the first part of Theorem~\ref{thm_A} is proved. The second part of that theorem is established in Section~\ref{section_conf_plat}. We then give compactness criteria for the space of conformal maps in Section~\ref{section_compacité}. The proofs of Theorems~\ref{thm_B1},~\ref{thm_B2}, and~\ref{thm_B3} are given in Sections~\ref{section_Brody},~\ref{section_conf_plat}, and~\ref{section_GHM}, respectively. Finally, the proofs of Theorem~\ref{thm_application} and Corollary~\ref{corollaire_de_lapplication} are given in Section~\ref{section_application}.

\subsection{Acknowledgment} This paper is based in part on the author’s doctoral thesis. The author would like to express his sincere gratitude to his PhD advisor, Charles Frances, for his guidance and continued support.

\section{Preliminaries}
\label{section_preliminary}

\subsection{Conformal manifolds}

We begin by reviewing the context of conformal pseudo-Riemannian manifolds. A \emph{conformal manifold} is a pair $(M,[g])$, where $M$ is a smooth manifold of dimension\footnote{Throughout this paper, all results can be extended to dimension $n=2$, if one restricts to Lorentzian surfaces that are moreover endowed with a $(\Ein^{1,1},\PO(2,2))$-structure, see Remark \ref{rmk_dim_2}.} $n\geq 3$, and $[g]$ is the \emph{conformal class} of a pseudo-Riemannian metric $g$, namely 
$$[g]=\left\{e^f\cdot g\,\vert\,f\in C^\infty(M)\right\}.$$ 
The \emph{signature} of a conformal manifold is a pair of integers $(p,q)$ referring to the number of negative and positive signs for the metric $g$, respectively. 
A tangent vector $v\in TM$ is called \emph{lightlike} (resp. \emph{timelike}, \emph{spacelike}) if $g(v,v)$ is null (resp. negative, positive). These notions are independent of the choice of a metric in the conformal class. A curve is said to be \emph{lightlike} (resp. \emph{timelike}, \emph{spacelike}) if its tangent vector is lightlike (resp. timelike, spacelike) at every point. More generally, it makes sens to say that a smooth submanifold $\Sigma\subset (M,[g])$ has a certain signature. 

Concerning pseudo-Riemannian invariants, such as the Levi-Civita connection or the curvature tensor, most of them fail to be invariant under conformal changes. Hence, pseudo-Riemannian geodesics do not admit a conformal meaning. However, the notion of \emph{lightlike pregeodesic}, namely a lightlike curve that can be reparametrized into a geodesic, has a conformal meaning.

\begin{thm}
    Two conformally equivalent metrics have the same lightlike pregeodesics.
\end{thm}

A map $\varphi:(M,[g])\to (N,[g^\prime])$ between two conformal manifolds of the same signature is said to be \emph{conformal} if $\varphi^*g^\prime\in [g]$. Equivalently, a conformal map is a smooth map that preserves the type of vectors. Note that in particular, a conformal map is always a local diffeomorphism. We denote by $\Conf(M,N)$ the set of conformal maps from $M$ to $N$, and by $\Conf(M)$ the \emph{conformal group} of $M$, that is the group of conformal diffeomorphisms of $M$. 

\subsection{The Minkowski space} We call \emph{Minkowski space}, and denote by $\R^{p,q}$, the conformal manifold
$\R^{p,q}=(\R^{p+q},[\b])$, where $\b$ is the standard flat metric  of signature $(p,q)$ on $\R^{p+q}$, that is:
$$\b=-dx_1^2-\dots-dx_p^2+dx_{p+1}^2+\dots+dx_{p+q}^2.$$
Lightlike geodesics of the Minkowski space are straight lines with lightlike directions. The Minkowski space is acted on conformally by the \emph{similarity group}, which is the group of affine motions that preserve $[\b]$. This group is generated by translations, linear isometries and dilations, and it splits as the semi-direct product
$$\Sim(\R^{p,q})=\R_{>0}\O(p,q)\ltimes\R^{p,q},$$
where $\O(p,q)$ is the group of linear isometries of $(\R^{p,q},\b)$.  

\subsection{The Einstein universe} A central example of conformal manifold is the so-called \emph{Einstein universe}, which is the pseudo-Riemannian analog of the conformal sphere. The \emph{Einstein universe}, denoted $\Ein^{p,q}$, may be defined as the set of isotropic lines of $\R^{p+1,q+1}$, that is:
$$\Ein^{p,q}=\{[v]\in\P(\R^{p+1,q+1})\,\vert\,\b(v,v)=0\},$$
where $[v]$ is the line generated by $v\in \R^{p+1,q+1}$, and $\b$ denotes the standard metric on $\R^{p+1,q+1}$. 
In that model, the Einstein universe is a compact hypersurface of real projective space. One can see that the conformal class $[\b]$ defines a conformal class of pseudo-Riemannian metrics of signature $(p,q)$ on $\Ein^{p,q}$.

\begin{prop}
    \label{prop_rev_double}
    The Einstein universe $\Ein^{p,q}$ is conformally equivalent to $(\S^p\times \S^q/_{\pm \operatorname{Id}},[-g_{\S^p}+g_{\S^q}])$.
\end{prop}

\begin{proof}
    Let $\R^{p+1,q+1}=H_{p+1}\oplus H_{q+1}$ be an orthogonal splitting into a negative definite and a positive definite subspace, respectively. Let $S_p$ and $S_q$ denote the unit spheres of $H_{p+1}$ and $H_{q+1}$, respectively. Then, a line $x\in\Ein^{p,q}$ intersects $S_p+S_q$ at exactly two antipodal points $v_x$ and $-v_x$. The map $f: \Ein^{p,q}\to S_p+S_q/_{\pm \operatorname{Id}}$ given by $f(x)=\pm v_x$ gives the desired identification.
\end{proof}

\subsubsection{Photons and lightcones} The image of a maximal lightlike geodesic curve in the Einstein universe is called a \emph{photon}. In the projective model, a photon is exactly the projectivization $\P(\Pi)$ of a totally degenerate two-plane $\Pi\subset\R^{p+1,q+1}$. Equivalently, photons are the only projective lines of $\P(\R^{p+1,q+1})$ that are entirely contained in $\Ein^{p,q}$. A photon is therefore diffeomorphic to a circle, and carries naturally a projective structure induced from that of real projective space. 

Given a point $x\in \Ein^{p,q}$, the \emph{lightcone} of $x$, denoted $C(x)$, is the union of all photons containing $x$. It is precisely the intersection of the projective hyperplane $x^\perp$ with $\Ein^{p,q}$. It is a closed singular hypersurface of $\Ein^{p,q}$, that is smooth outside $x$. The lightcone of a point gives the following analog of the stereographic projection from the punctured sphere to the Euclidean space.

\begin{prop}[Stereographic projection]  
    \label{prop_projection_stereo}
    Let $x\in \Ein^{p,q}$. Then the complement of $C(x)$ in $\Ein^{p,q}$ is conformally diffeomorphic to $\R^{p,q}$.
\end{prop}

\begin{proof}
    Let $v\in \R^{p+1,q+1}$ such that $x=[v]$ and let $w\in \R^{p+1,q+1}$ be an isotropic vector such that $\b(v,w)=1$. For every $h\in H=\Span(v,w)^\perp\simeq\R^{p,q}$, we define 
    $\varphi(h)=[w+h-\b(h,h)v]$.
    Then the map $\varphi:H\to \Ein^{p,q}$ is a conformal embedding whose image is 
    $$\{[u]\in\P(\R^{p+1,q+1})\,\vert\,\b(u,v)\neq 0\}=\Ein^{p,q}\setminus C(x).$$
\end{proof}

\subsubsection{Conformal group of the Einstein universe} The group $\O(p+1,q+1)$ is the group of linear isometries of $\R^{p+1,q+1}$. Its action on lightlike lines yields a faithful conformal action of $\PO(p+1,q+1)=\O(p+1,q+1)/\{\pm\operatorname{Id}\}$ on $\Ein^{p,q}$. This action is transitive and the Einstein universe can be identified with the homogeneous space $\PO(p+1,q+1)/P$, where $P$ is the stabilizer of an isotropic line in $\R^{p+1,q+1}$. It turns out that conformal diffeomorphisms of the Einstein universe are all elements of $\PO(p+1,q+1)$. More generally, any local conformal transformation of $\Ein^{p,q}$ extends to an element of that group, see \cite{france_liouville}.

\begin{thm}[Liouville]
    Let $U,V$ be nonempty connected open subsets of $\Ein^{p,q}$, with $p+q\geq 3$. Then any conformal map $\varphi:U\to V$ extends uniquely to a map $\hat{\varphi}\in \PO(p+1,q+1)$.
\end{thm}

\subsection{Conformally flat geometry}

A conformal manifold is said to be \emph{conformally flat} if it is locally conformally equivalent to the Minkowski space $\R^{p,q}$. Since $\Ein^{p,q}$ is homogeneous and contains a conformal copy of $\R^{p,q}$, it is a conformally flat manifold. In fact, the Einstein universe is a model for conformally flat geometry, meaning that conformally flat manifolds are in one-to-one correspondence with geometric manifolds modeled on the homogeneous space $(\PO(p+1,q+1),\Ein^{p,q})$ in the sense of Ehresmann--Thurston. For a review of $(G,X)$-structures, we refer to \cite{Thurston}. As a consequence, a conformally flat structure on a manifold $M$ is determined by the data of a pair of maps
$$\dev:\Mtilde\to \Ein^{p,q} \text{ and }\hol:\pi_1(M)\to \PO(p+1,q+1),$$
where $\dev$ is a conformal map called the \emph{developing map}, and $\hol$ is a group homomorphism called the \emph{holonomy morphism}, that makes the developing map equivariant with respect to the $\pi_1(M)$-action on $\Mtilde$.

\subsection{Causality on spacetimes} We now turn to the Lorentzian signature, that is when $p=1$. Lorentzian signature will be denoted $(1,n)$, where $n+1$ refers to the manifold's dimension. A tangent vector in a conformal Lorentzian manifold $(M,[g])$ is \emph{causal} if it is either timelike or lightlike. A \emph{time orientation} on $(M,[g])$ is given by a continuous timelike vector field. A conformal Lorentzian manifold with a time orientation is called a \emph{conformal spacetime}. For a given time orientation $X$, causal vectors $v$ satisfying $g(v,X)<0$ (resp. $g(v,X)>0$) are called \emph{future causal vectors} (resp. \emph{past causal vectors}). This infinitesimal language can be translated to curves: a smooth causal curve $\gamma:I\to M$ is called \emph{future oriented} if its tangent vector $\gamma^\prime(t)$ is future for all $t\in I$.

\subsection{Causal relations} Given two points $x,y$ of a conformal spacetime $(M,[g])$, we say that $y$ is in the causal future of $x$ (resp. chronal future of $x$), and we write $x\leq y$ (resp. $x\ll y$), if there exists a future causal curve (resp. future timelike curve) in $M$ starting at $x$ and ending at $y$. The causal/chronal future and causal/chronal past of a point $x\in M$ will be denoted as follows
$$J^+(x)=\{y\in M\,\vert\,x\leq y\},\,\,\,\,\,\,\,\,\,\,\,\,\,\,\,\,\,\,I^+(x)=\{y\in M\,\vert\,x\ll y\},$$
$$J^-(x)=\{y\in M\,\vert\,y\leq x\},\,\,\,\,\,\,\,\,\,\,\,\,\,\,\,\,\,\,I^-(x)=\{y\in M\,\vert\,y\ll x\}.$$
For $x\leq y$, the \emph{causal diamond} (resp. \emph{chronal diamond}) generated by $x$ and $y$ is $J(x,y)=J^+(x)\cap J^-(y)$ (resp. $I(x,y)=I^+(x)\cap I^-(y)$). Finally, a domain $\Omega$ of a conformal spacetime $(M,[g])$ is called 
\begin{itemize}
    \item \emph{causally convex} if $J(x,y)\subset \Omega$ for every $x\leq y\in \Omega$.
    \item \emph{future complete} (resp. \emph{past complete}) if $J^+(x)\subset\Omega$ (resp.  $J^-(x)\subset\Omega$) for every $x\in \Omega$.
\end{itemize}

\section{Projective parameters of lightlike geodesics}
\label{section_para_proj}

Let $(M,[g])$ be an $n$-dimensional pseudo-Riemannian conformal manifold, with $n\geq 3$. Let $g$ be a metric in the conformal class of $M$, and let $\xi(t)$ be a lightlike geodesic of $(M,g)$, that is, a solution to $\nabla_{\xi^\prime}\xi^\prime=0$, where $\nabla$ is the Levi-Civita connection of $g$. A \emph{projective parameter} along $\xi(t)$ is a solution $u(t)$ of the differential equation 
\begin{equation}
\label{equation_parametrage_projectif}
    Su(t)=\frac{2}{n-2}\Ric_g(\xi^\prime(t),\xi^\prime(t)),\footnote{Note that our sign convention is different from \cite{markowitz_1981}. The sign choice in Markowitz's paper is due to an error in the proof of Proposition 2.10 of that paper.}
\end{equation}
where $Sf=\left(\frac{f^\prime}{f^{\prime\prime}}\right)^\prime-\frac{1}{2}\left(\frac{f^\prime}{f^{\prime\prime}}\right)^2$ denotes the \emph{Schwarzian derivative} of the function $f$ and $\Ric_g$ is the Ricci tensor of the metric $g$. We refer to \cite{Ovsienko_Tabachnikov_2004} for basic properties of the Schwarzian derivative. Let $g_1,g_2\in [g]$, and let $\xi_1:J_1\to M$ and $\xi_2:J_2\to M$ be two maximal affine geodesics for $g_1$ and $g_2$, respectively, such that $\xi_1^\prime(0)=\xi_2^\prime(0)\neq0$. These geodesics admit projective parameters $u_1$ and $u_2$, defined on neighborhoods of $t=0$. Choose such parameters so that $u_1(0)=u_2(0)=0$ and the first and second derivatives of $\gamma_1=\xi_1\circ u_1^{-1}$ and $\gamma_2=\xi_2\circ u_2^{-1}$ coincide at $t=0$. Then, the projectively parametrized lightlike geodesics $\gamma_1$ and $\gamma_2$ coincide on a neighborhood of the origin, see \cite{markowitz_1981}.

\begin{definition}
    \label{def_PPL}
    Let $(M,[g])$ be a conformal manifold, and $I\subset\R$ an interval. A curve $\gamma:I\to M$ is called a \emph{projectively parametrized lightlike geodesic} if for some (hence any) metric $g$ in the conformal class of $M$, there exists an affine geodesic $\xi:J\to M$ for $g$, and a projective parameter $u:J\to I$ along $\xi$, such that $\gamma(t)=\xi(u^{-1}(t))$ for all $t\in I$. 
\end{definition}

We denote by $\PPL(I,M)$ the set of projectively parametrized lightlike geodesics from $I$ to $M$. Given $\gamma\in\PPL(I,M)$, then for any homography $h:J\to I$, the curve $\gamma\circ h$ is again projectively parametrized. Conversely, if $\gamma$ is non-constant and $\gamma\circ h$ is projectively parametrized for some smooth parameter $h$, then $h$ is a homography. This explains the terminology \guillemotleft~projective parameter~\guillemotright.

\begin{examples}
    \label{exemple_para_proj}
    $\bullet$ Let $(M,g)$ be a pseudo-Riemannian manifold such that $\Ric_g(v,v)=0$ for every lightlike vector $v\in TM$. 
    Then any affine lightlike geodesic $\gamma:I\to (M,g)$ is a projectively parametrized lightlike geodesic of $(M,[g])$.

    $\bullet$ Let $M=(\S^p\times \S^q,[-g_{\S^p}+g_{\S^q}]])$ be the double cover of $\Ein^{p,q}$. For $g=g_{\S^p}+g_{\S^q}$, and $\xi=(\xi^p,\xi^q)$ is a product of unit-speed geodesics, Equation~(\ref{equation_parametrage_projectif}) becomes $Su(t)=2$, which is solved by $u=\tan$. Thus $\gamma=\xi\circ \tan^{-1}:\R\to M$ is projectively parametrized, and its endpoints are antipodal.  
\end{examples}

\begin{rmk}
    \label{rmk_dim_2}
    In dimension $n=2$, that is, for Lorentzian surfaces, projectively parametrized lightlike geodesics are not well defined, since Equation~(\ref{equation_parametrage_projectif}) is not well defined in that case. For a $(\Ein^{1,1},\PO(2,2))$-surface $\Sigma$, one can make sense of projectively parametrized lightlike geodesics as follows. A lightlike curve $\gamma:I\to \Sigma$ is \emph{projectively parametrized} if for some suitable chart $\varphi:U\to\R^{1,1}$, the curve $\varphi(\gamma)$ is a geodesic for the flat metric of $\R^{1,1}$. With this definition, all properties about projectively parametrized lightlike geodesics -- and hence about Markowitz pseudodistance -- remain true, if one replaces conformal transformations by $(\Ein^{1,1},\PO(2,2))$-maps.
\end{rmk}

\subsection{Geometric structure on lightlike geodesics} Let us explain how to extend the notion of projectively parametrized lightlike geodesics to curves defined on an interval of $\RPunTilde$. Let $(M,[g])$ be a conformal manifold and let $J$ be an open interval of $\RPunTilde$. A curve $\gamma:J\to M$ is said to be projectively parametrized if for every interval $I\subset \R$, and every projective immersion $h:I\to J$, the curve $\gamma\circ h:I\to M$ is projectively parametrized in the previous sense. 

\begin{prop}\label{prop_str_proj_sur_les_geod_de_lum}
    Let $(M,[g])$ be a conformal manifold and let $v\in TM$ be lightlike. Then there exists an open interval $J\subset \RPunTilde$ and an inextendible projectively parametrized lightlike geodesic $\gamma:J\to M$ that is tangent to $v$. The interval $J$ and the curve $\gamma$ are unique up to a projective reparametrization. Also, the curve $\gamma$ is determined by its 2-jet at any of its point.
\end{prop}

\begin{proof}
    Let $I\subset \R$ and $\xi:I\to M$ be an inextendible lightlike geodesic tangent to $v$. One may construct a $(\RP^1,\PSL(2,\R))$-structure on $I$ as follows. For every $t\in I$, let $u_t:I_x\to \R$ be a smooth immersion defined on a subinterval $I_x\subset I$ such that $\xi\circ u_t^{-1}$ is projectively parametrized. Then $\{u_t\}$ is an atlas on $I$ whose transition maps are homographies, i.e. elements of $\PSL(2,\R)$. Let $u:I\to \RPunTilde$ be a developing map for this geometric structure, and let $J\subset\RPunTilde$ be its image. Then the curve $\xi\circ u^{-1}$ is projectively parametrized. The fact that $\gamma$ is determined by its 2-jet follows from the discussion above Definition \ref{def_PPL}.
\end{proof}

The \emph{projective structure} of an unparametrized lightlike geodesic $\gamma$ is the projective structure of an interval $J\subset\RPunTilde$ such that $\gamma$ can be projectively parametrized by a curve $J\to M$.

\subsection{Convergence of lightlike geodesics}\label{Section_topologies_PPL} There are two natural topologies on the space $\PPL(I,M)$: the uniform topology and the smooth topology, the latter being stronger than the former. The aim of this section is to show that, in fact, these two topologies coincide. To that end, we will prove the following.

\begin{prop}
    \label{prop_topologie_PPL}
    Let $(g_k)$ be a sequence of pseudo-Riemannian metrics on $M$ converging to a metric $g$ in the $C^m$ topology, for $m\geq 2$. Let $(\gamma_k:I\to M)$ be a sequence of projectively parametrized lightlike geodesics for $(g_k)$. If $(\gamma_k)$ converges uniformly on compact subsets to a curve $\gamma:I\to M$, then $\gamma$ is a projectively parametrized lightlike geodesic for $g$, and $(\gamma_k)$ converges to $\gamma$ in the $C^{m+1}$ topology.
\end{prop}

Applying the previous proposition to a constant sequence of metrics $(g_k=g)$ yields the equality of the previously mentioned topologies.

\begin{cor}
    \label{cor_topologie_PPL}
    Let $(M,[g])$ be a conformal manifold. Then the uniform topology on $\PPL(I,M)$ coincides with the smooth topology.\qed
\end{cor}

The rest of this section is devoted to the proof of Proposition~\ref{prop_topologie_PPL}. We will make use of the following standard ODE regularity proposition.

\begin{prop}[{\cite[Chap 2, Thm. 3.2]{Hartman}}]\label{prop_equa_diff} Let $(f_k:\Omega\to \R^n)$ be a sequence of functions defined on an open subset $\Omega\subset\R\times \R^n$, such that $(f_k)$ converges uniformly on compact subsets of $\Omega$ to a function $f:\Omega\to \R^n$. Let $(t_k^0,y_k^0)\to (t^0,y^0)\in \Omega$ be a sequence of initial data, and let $y_k:(t_k^-,t_k^+)\to \R^n$ (resp. $y:(t^-,t^+)\to \R^n$) be the maximal solution to the ordinary differential equation
$$y_k(t_k^0)=y_k^0\text{ and }y_k^\prime=f_k(t,y_k),$$
(resp. $y(t^0)=y^0\text{ and }y^\prime=f(t,y)$). Then $\limsup_k t_k^-\leq t^-<t^+\leq \liminf_k t_k^+$ and 
$y_k\to y$
uniformly on compact subsets of $(t^-,t^+)$.
\end{prop}

\begin{cor}
    \label{cor_conv_geod_avec_2jet}
    Let $(g_k)$ be a sequence of pseudo-Riemannian metrics on $M$ converging to a metric $g$ in the $C^m$ topology, for $m\geq 2$. Let $(\gamma_k:I_k\to M)$ be a sequence of projectively parametrized lightlike geodesics of $(M,[g_k])$, defined on intervals $I_k$ containing $0$. If the sequence of 2-jets $(J^2\gamma_k(0))$ converges and $\smash{\lim_k\gamma_k^\prime(0)\neq 0}$, then there exists an interval $I_\varepsilon=(-\varepsilon,\varepsilon)$ such that the curves $\gamma_k$ may be extended to projectively parametrized lightlike geodesics on $I_\varepsilon$, and $(\gamma_k)$ converges on compact subsets of $I_\varepsilon$, in the $C^{m+1}$ topology, to a curve $\gamma$ that is projectively parametrized for the metric $g$. 
\end{cor}

\begin{proof}
    For all $k\geq 0$, we let $\xi_k:J_k\to M$ denote the maximal (affine) geodesic of $(M,g_k)$ with $J^1\xi_k(0)=J^1\gamma_k(0)$, where $J_k\subset\R$ be its maximal domain of definition. Let $\xi:J\to \R$ be the unique affine geodesic of $(M,g)$, such that $J^1\xi(0)=\lim_kJ^1\xi_k(0)$. Since $(g_k)$ converges to $g$ in the $C^{m}$ topology, it follows from Proposition~\ref{prop_equa_diff} that the sequence $(\xi_k)$ converges on compact subsets of $J$ to $\xi$ in the $C^{m+1}$ topology. Since $(\gamma_k)$ is projectively parametrized, there exists a subinterval $L_k\subset J_k$ containing $0$ and a projective parameter $u_k:L_k\to I_k$ along $\xi_k$ for the metric $g_k$, such that  $\xi_k=\gamma_k\circ u_k$ holds on $L_k$. As $u_k$ is a solution of $Su_k(t)=\frac{2}{n-2}\Ric_{g_k}(\xi_k^\prime(t),\xi_k^\prime(t))$, the vector $y_k=(u_k,u_k^\prime,u_k^{\prime\prime})$ is a solution of the Cauchy problem
    $$y_k(0)=(0,1,u_k^{\prime\prime}(0))\text{ and }y_k^\prime=f_k(t,y_k),$$ where $f_k(t,y)=\left(y_2,y_3,\frac{2}{n-2}\Ric_{g_k}(\xi_k^\prime(t),\xi_k^\prime(t))y_2+\frac{3y_3^2}{2y_2}\right)$. Since $(g_k)$ converges to $g$ in the $C^m$ topology, the sequence $(f_k)$ converges on compact subsets in the $C^{m-2}$ topology to 
    $$f(t,y)=\left(y_2,y_3,\frac{2}{n-2}\Ric_{g}(\xi^\prime(t),\xi^\prime(t))y_2+\frac{3y_3^2}{2y_2}\right).$$
    Now, differentiating $\xi_k=\gamma_k\circ u_k$ twice yields the equation 
    $$\xi_k^{\prime\prime}(0)=u_k^{\prime\prime}(0)\gamma_k^\prime(0)+\gamma_k^{\prime\prime}(0).$$
    Since $J^2\gamma_k(0)$ and $J^2\xi_k(0)$ both converge, and since $\lim_k\gamma_k^\prime(0)\neq 0$, the previous equation implies that the sequence $(u_k^{\prime\prime}(0))$ converges to some constant $a\in \R$, so $y_k(0)\to (0,1,a)$. Applying Proposition~\ref{prop_equa_diff}, the sequence $(y_k)$ converges on compact subsets in the $C^{m-1}$ topology to the maximal solution $y$ of the Cauchy problem $y(0)=(0,1,a)$ and $y^\prime=f(t,y)$. In particular, the sequence of projective parameters $(u_k)$ converges on compact subsets of $L$ to a projective parameter $u$ in the $C^{m+1}$ topology. Also, we can always extend the domain of definition of $u$ and of each $u_k$ if necessary and assume that there exists $\varepsilon>0$ such that $u(J)\supset I_\varepsilon=(-\varepsilon,\varepsilon)$, and $u_k(J_k)\supset I_\varepsilon$ for all $k\geq 0$.  In particular $\gamma_k=\xi_k\circ u_k^{-1}\to \xi\circ u^{-1}=\gamma$ in the $C^{m+1}$ topology.
\end{proof}

\begin{proof}[Proof of Proposition~\ref{prop_topologie_PPL}]
    Assume for simplicity that $0\in I$. Let us first show that we can find a sequence $(h_k)$ of homographies fixing $0$, such that $J^2(\gamma_k\circ h_k^{-1})(0)$ converges with $\lim_k(\gamma_k\circ h_k^{-1})^\prime(0)\neq 0$. We can always find a sequence of linear maps $(a_k)$, such that $J^1(\gamma_k\circ a_k)^\prime(0)$ converges to a nonzero vector. Let $k\geq 0$ and let $\varphi:U\subset M\to \R^{p,q}$ be some exponential chart at $\gamma_k(0)$ for the metric $(g_k)$. Then in the chart $(U,\varphi)$, the vectors $\gamma_k^\prime(0)$ and $\gamma_k^{\prime\prime}(0)$ are collinear. In particular one can find a unique homography $u_k$ such that $u_k(0)=0$, $u_k^\prime(0)=1$ and $(\gamma_k\circ a_k\circ u_k)^{\prime\prime}(0)=0$ in the chart $(U,\varphi)$. Then $h_k=(a_k\circ u_k)^{-1}$ gives the required sequence of homographies. We will write $\alpha_k=\gamma_k\circ h_k^{-1}$. From Corollary~\ref{cor_conv_geod_avec_2jet}, one can find $\varepsilon>0$ such that $\alpha_k$ converges on $I_\varepsilon=(-\varepsilon,\varepsilon)$ for the $C^{m+1}$ topology to a curve $\alpha:I_\varepsilon\to M$ satisfying $\alpha^\prime(0)\neq 0$. 
    
    Assume first that the sequence $(h_k)$ converges in $\PSL(2,\R)$. Then $(h_k)$ converges smoothly on compact subsets, so $\gamma_k=\alpha_k\circ h_k$ converges in the $C^{m+1}$ topology. We now assume that $(h_k)$ is a diverging sequence of $\PSL(2,\R)$. Up to extracting a subsequence, the sequence $(h_k)$ admits a North/South dynamic: there exist two points $x^+,x^-\in\RP^1=\R\cup\{\infty\}$ such that $(h_k)$ converges to $x^+$ uniformly on compact subsets of $\RP^1\setminus\{x^-\}$. We claim that $x^-\neq 0$. Assume by contradiction that $x^-=0$. 
    
    \emph{$\bullet$ Case 1: the sequence $(h_k)$ has a parabolic action, i.e. $x^+=0$}. In that case, for $\varepsilon>0$, the sequence of intervals $h_k(I)$ contains infinitely many times the interval $(0,\varepsilon)$, or the interval $(-\varepsilon,0)$. Assume for instance that $(0,\varepsilon)\subset h_k(I)$ for all but finitely many $k\geq 0$. Then for all $s\in (0,\varepsilon)$, one can find a sequence $t_k>0$ converging to $0$ such that $h_k(t_k)\to s$. Hence $\alpha(s)=\lim_k\alpha_k(h_k(t_k))=\lim_k\gamma_k(t_k)=\gamma(0)$. This contradicts the fact that $\alpha^\prime(0)\neq 0$.   

    \emph{$\bullet$ Case 2: the sequence $(h_k)$ has a hyperbolic action, i.e. $x^+\neq0$}. Then there exists $r>0$ such that $x^+\not \in (-r,r)$. As before, for every $s\in (-r,r)$, there exists $t_k\to 0$ such that $h_k(t_k)= s$. Thus $\alpha(s)=\lim_k\gamma_k(t_k)=\gamma(0)$, contradicting again the fact that $\alpha^\prime(0)\neq 0$.
    
    This shows that $x^-\neq 0$. In particular, one can find an interval $J\subset I$ containing $0$ such that $(h_k)$ converges smoothly on $J$ to a homography or the constant zero map. Therefore, a subsequence of $\gamma_k=\alpha_k\circ h_k$ converges on a neighborhood of $0$ and for the $C^{m+1}$ topology to a projectively parametrized lightlike geodesic of $g$. It follows that $\gamma$ is smooth and that for any $t\in I$, there exists a subsequence of $(\gamma_k)$ converging to $\gamma$ for the $C^{m+1}$ topology on a neighborhood of $t$. From a standard compactness argument, $\gamma$ is projectively parametrized for $g$, and $(\gamma_k)$ converges to $\gamma$.
\end{proof}

\section{Markowitz hyperbolicity}
\label{section_hyp_Mark}

In the rest of this paper, we will denote by $I$ the interval $(-1,1)$ endowed with its standard projective structure and with the hyperbolic metric $\rho_I=\frac{4dx^2}{(1-x^2)^2}$ and associated distance $d_I$. 

\subsection{The Markowitz pseudodistance}\label{section_delta_M}

Let $(M,[g])$ be a connected pseudo-Riemannian manifold and $x,y\in M$. A \emph{lightlike chain} from $x$ to $y$ is the data of: an integer $m\geq 1$, a finite sequence $(\gamma_k)_{1\leq k\leq m}\in \PPL(I,M)$ and two finite sequences of points $(s_k)_{1\leq k\leq m},(t_k)_{1\leq k\leq m}\in I$ such that 
$$\gamma_1(s_1)=x,\,\gamma_m(t_m)=y\text{ and }\gamma_k(t_k)=\gamma_{k+1}(s_{k+1}) \text{ for all }1\leq k\leq m-1.$$
We will write $\Ccal=(m,(\gamma_k),(s_k),(t_k))$ for such a lightlike chain. Note that there always exists at least one lightlike chain between two points of $M$. Indeed, the equivalence relation $x\sim y\iff$ ``there exists a lightlike chain between $x$ and $y$'' is open:
\begin{fact}
    \label{fact_existence_chaine}
    Every point $x\in M$ admits a neighborhood $U$, such that for any $y\in U$, the point $x$ can be joined to $y$ by at most $p+q=\dim(M)$ lightlike geodesic segments.
\end{fact}
\begin{proof}
Let $g$ be a metric in the conformal class of $M$. Let $x\in M$ and let $(X_1,\dots,X_{p+q})$ be a lightlike frame defined on a neighborhood of $x$. Given a vector field  $X$, we write $\phi_X^t$ for the map $\phi_X^t(z)=\exp_z(tX(z))$.
Let $\theta$ be the map defined on a neighborhood of $0\in\R^{p+q}$ by 
$\smash{\theta(t_1,\dots,t_{p+q})=\phi_{X_1}^{t_1}\circ\dots\circ\phi_{X_{p+q}}^{t_{p+q}}(x)}$.
For all $i\in\{1,\dots,p+q\}$, one has 
$\smash{\frac{\partial\theta}{\partial t_i}(0)=X_i}$, hence $\theta$ is a local diffeomorphism at $0$. By construction, any point in the image of $\theta$ can be joined to $x$ by at most $p+q$ lightlike geodesics. 
\end{proof}

The \emph{length} of a chain $\Ccal=(m,(\gamma_k),(s_k),(t_k))$ is defined to be
$L(\Ccal)=d_I(s_1,t_1)+\dots+d_I(s_m,t_m).$
The \emph{Markowitz pseudodistance} between two points $x$ and $y$ is then defined as
\begin{equation}
    \label{equation_def_d_mark}
    \delta_M(x,y)=\inf_\Ccal L(\Ccal),
\end{equation}
where the infimum runs over all lightlike chains from $x$ to $y$. From the previous fact, the quantity $\delta_M(x,y)$ is finite for all $x,y\in M$. One can check from the definition that $\delta_M$ is a pseudodistance: it is symmetric, satisfies the triangular inequality but need not separate points. 
The main advantage of the Markowitz pseudodistance is that it turns conformal maps into 1-Lipschitz maps. 
\begin{prop}
\label{proposition_naturalité}
    Let $f:M\to N$ be a conformal map. Then, for all $x,y\in M$, one has 
$$\delta_N(f(x),f(y))\leq \delta_M(x,y).$$
\end{prop}
\begin{proof}
    This follows from the definition of $\delta$, since the image of a lightlike chain from $x$ to $y$ by the map $f$ is a lightlike chain from $f(x)$ to $f(y)$ of the same length.
\end{proof}

\subsection{The infinitesimal functional}
\label{section_FM}
The Markowitz pseudodistance also admits an infinitesimal form, defined on the bundle $C(TM)$ of lightlike tangent vectors of a conformal manifold $M$. Let $(M,[g])$ be a conformal manifold, and let $v\in C(TM)$ be lightlike. Following \cite{markowitz_1981}, the \emph{infinitesimal functional} of $M$ at $v$ is defined as 
\begin{equation}
    F_M(v)=\inf_\gamma \sqrt{\rho_I(u,u)},
\end{equation}
where the infimum runs over all $\gamma\in\PPL(I,M)$, and vector $u\in TI$ such that $\gamma_*u=v$. This defines a map $F_M:C(TM)\to \R$ 
that is homogeneous in the sense that $F_M(\lambda v)=\vert\lambda\vert F_M(v)$ for all $\lambda\in \R$ and $v\in C(TM)$. Furthermore, the map $F_M$ is upper-semicontinuous, hence measurable and bounded on compact subsets, see \cite[Lem. 4.1]{markowitz_1981}. In particular, it is possible to integrate $F_M$ along lightlike curves. The infinitesimal functional can then be linked with the Markowitz pseudodistance in the following natural way:

\begin{thm}[{\cite[Thm. 4.8]{markowitz_1981}}]
\label{thm_lien_dmark_Fmark}
    Let $(M,[g])$ be a conformal spacetime. Then, for all $x,y\in M$, one has 
    $$\delta_M(x,y)=\inf_\gamma\int F_M(\gamma^\prime(t)) dt,$$
    where the infimum runs over all piecewise lightlike geodesic curves $\gamma$ from $x$ to $y$.
\end{thm}

Similarly to Proposition~\ref{proposition_naturalité}, the infinitesimal functional turns conformal maps into Lipschitz maps.

\begin{prop}
    \label{prop_naturalité_F_M}
    Let $f:M\to N$ be a conformal map. Then, for all $v\in C(TM)$, one has
    $$F_N(f_*v)\leq F_M(v).$$
    If $f$ is a covering map, then equality holds in the previous inequality.
\end{prop}

\begin{proof}
    The inequality follows from the definition of $F$. When $f$ is a covering map, then any $\gamma\in\PPL(I,N)$ can be lifted to a $\tilde\gamma\in\PPL(I,M)$ by $f$, so equality holds.
\end{proof}

For later purposes, we give a formula of the infinitesimal functional for a domain of the double cover of $\Ein^{p,q}$.
\begin{prop}
        \label{Prop_exemple_calcul_F_M}
    Let $\Omega\subset \S^p\times \S^q$ be a connected domain. Let $x\in \Omega$ and $v\in T_x\Omega$ be lightlike. Let $\gamma$ be the maximal lightlike geodesic segment that is tangent to $v$ and contained in $\Omega$, and let $\sqrt{2}L$ denote its length with respect to the Riemannian metric $g=g_{\S^p}+g_{\S^q}$.
    If $L<\pi$, then
        $$F_\Omega(v)=\left(\frac{1+\tan(d_g(x,x_-)-L/2)^2}{\tan(L/2)+\tan(d_g(x,x_-)-L/2)}
        +\frac{1+\tan(d_g(x,x_+)-L/2)^2}{\tan(L/2)+\tan(d_g(x,x_+)-L/2)}\right)\frac{\Vert v\Vert_g}{\sqrt{2}},$$
    where $x_-$ and $x_+$ are the endpoints of $\gamma$. Otherwise, if $L\geq \pi$, then $F_\Omega(v)=0$.
\end{prop}
\begin{proof}
    Assume that $L<\pi$, and let $\xi:(-L/2,L/2)\to \Omega$ be a $g$-geodesic at speed $\sqrt{2}$ that parametrizes the segment  $(x_-,x_+)$. Then $\xi=(\xi^p,\xi^q)$, where $\xi^i$ is a unit-speed geodesic of $\S^i$, for $i=p,q$. Hence $\gamma=\xi\circ \tan^{-1}$ is a maximal projectively parametrized lightlike geodesic tangent to $v$ at $u=\tan(L/2-d(x,x_+))$, see Example~\ref{exemple_para_proj}.
    We deduce that 
    $$F_\Omega(\gamma^\prime(t))=\frac{2\tan(L/2)}{\tan(L/2)^2-u^2}=(1+u^2)^{-1}F_\Omega\left(\sqrt{2}\frac{v}{\Vert v\Vert_g}\right).$$
    The result follows in that case.
    When $L\geq\pi$, then using the same argument as before, there exist non-constant projectively parametrized lightlike geodesics $\gamma:\R\to \Omega$ tangent to $v$, so $F_\Omega(v)=0$, see also Proposition~\ref{prop_TCLI}.
\end{proof}

\begin{ex}
    \label{exemple_formula_F_M}
    Let $\Omega\subset \R^{p,q}$ and $v\in T_x\Omega$ be lightlike. Let $a,b\in \R^{p,q}$ denote the endpoints of the maximal lightlike segment tangent to $v$ and contained in $\Omega$. Then
    $$F_\Omega(v)=\frac{\Vert v \Vert}{\Vert x-a\Vert}+\frac{\Vert v \Vert}{\Vert x-b\Vert},$$
    for any norm $\Vert \cdot\Vert$ on $\R^{p,q}$.
\end{ex}

\subsection{Characterizations of Markowitz hyperbolicity}\label{Section_characterisation_of_Mark_hyp} We now turn to the main notion of this paper, namely, manifolds for which the Markowitz pseudodistance is definite. 

\begin{definition}
    A conformal manifold $(M,[g])$ is called \emph{Markowitz hyperbolic} if the pseudodistance $\delta_M$ separates points, that is $\delta_M(x,y)> 0$ whenever $x\neq y$.
\end{definition}

This notion of hyperbolicity was first introduced by Markowitz in \cite{markowitz_1981} as a conformal analog of Kobayashi's hyperbolicity for projective manifolds \cite{kob_projectif}. In general, it is impossible to compute $\delta_M$, therefore it can be a hard task to determine whether a conformal manifold is Markowitz hyperbolic or not. 

\begin{thm}
    \label{thm_hyp}
    Let $(M,[g])$ be a conformal manifold. Then the following properties are equivalent:
    \begin{enumerate}
        \item $M$ is Markowitz hyperbolic, that is $\delta_M$ is a distance.
        \item $\delta_M$ induces the standard topology on $M$.
        \item $\delta_M$ is locally equivalent to a (hence any) Riemannian distance.
        \item $F_M$ is locally bounded below by a Riemannian norm restricted to $C(TM)$.
        \item The family $\PPL(I,M)$ is equicontinuous, with respect to some distance on $M$ that is locally equivalent to a Riemannian distance. 
    \end{enumerate}
    In that case, the Markowitz distance $\delta_M$ is a length metric on $M$.    
\end{thm}

\begin{rmk}
    Note that in Theorem \ref{thm_hyp}, condition (4) is not equivalent to the positivity of $F_M$, since the infinitesimal functional need not be continuous, see Section \ref{Section_total_conf_lightlike_incompletness}.
\end{rmk} 

\begin{proof}
    Let us first show that $\delta_M$ is always locally bounded above by a Riemannian distance. Since $F_M$ is bounded on compact subsets of $M$, we can find a Riemannian metric $h$ on $M$ such that $F_M\leq \Vert\cdot\Vert_h$ on $C(TM)$. Let $d_h$ be the Riemannian distance associated to $h$ and, for $x,y\in M$, let 
$$\tilde{\delta}(x,y)=\inf_{\gamma:x\rightsquigarrow y}\int \Vert\gamma^\prime(t)\Vert_h\text{d}t,$$
where the infimum runs over all piecewise lightlike geodesic curves from $x$ to $y$.     
From Theorem~\ref{thm_lien_dmark_Fmark}, one has $\delta_M\leq \tilde{\delta}$, so it remains to show that for all $x\in M$, there exists a neighborhood of $x$ and a constant $\alpha>0$ such that $\tilde{\delta}\leq \alpha d_h$ on that neighborhood. Let $x\in M$, and let $(X_1,\dots,X_n)$ be a lightlike frame in a neighborhood of $x$, and fix a metric $g$ in the conformal class of $M$. For $y$ close to $x$, we denote by $\theta_y$ the map defined on an neighborhood of $0\in \R^n$ by $$\smash{\theta_y(t_1,\dots,t_n)=\phi_{X_1}^{t_1}\circ\dots\circ\phi_{X_n}^{t_n}(y)},$$ 
where $\phi_X^t(z)=\exp_z(tX(z))$ and $\exp$ is the exponential map of $g$.
One can find $\varepsilon>0$ and a compact neighborhood $V$ of $x$ such that for all $y\in V$, the map $\theta_y$ is defined on $\overline{B}=[-\varepsilon,\varepsilon]^{p+q}$ and it is a diffeomorphism onto its image, with $V\subset \theta_y(\overline{B})$. By compactness of $\overline{B}\times V$, there exist constants $\alpha_1,\alpha_2>0$ such that
\begin{equation}
\label{inegalite_g_g_0}
    \alpha_1^{-1} h\leq (\theta_x)_*h_{eucl}\leq \alpha_2 h,\text{ for all $x\in V$.}
\end{equation}
Let $0<\varepsilon^\prime<\varepsilon$ and let $B^\prime=(-\varepsilon^\prime,\varepsilon^\prime)^n$. We may choose $\varepsilon^\prime$ sufficiently small and take $V$ smaller if necessary, so that there exists a neighborhood $W$ of $x$, convex for the metric $h$, such that $\theta_y(B^\prime)\subset W\subset \theta_y(B)$ for all $y\in V$. We can still reduce $V$ and assume that $V\subset\theta_y(B^\prime)$ for all $y\in V$. Let $y,z\in V$ and let $u\in B^\prime$ such that $z=\theta_y(u)$. We let $c_1$ be a curve parameterizing the segment $[0,u]\subset B^\prime$, and let $c_2$ be the piecewise linear curve from $0$ to $u$ composed of at most $n$ segments, whose directions are parallel to the axes of $B^\prime$, and whose Riemannian length is
\begin{equation*}
\label{c_1_et_c_2}
    L(c_2,h_{eucl})=\vert u_1\vert+\dots+\vert u_n\vert \leq \sqrt{n}\,L(c_1,h_{eucl}).
\end{equation*} 
By construction, the curve $\gamma_2=\theta_y\circ c_2$ is a lightlike chain from $y$ to $z$, hence Equation~(\ref{inegalite_g_g_0}) implies:
\begin{equation*}
    \tilde{\delta}(y,z)\leq L(\gamma_2,h)\leq \alpha_1  L(\gamma_2,(\theta_y)_*h_{eucl})=\alpha_1L(c_2,h_{eucl})\leq \sqrt{n}\,\alpha_1 L(c_1,h_{eucl}).
\end{equation*}
Since $W$ is convex for $h$, Equation~(\ref{inegalite_g_g_0}) also implies: 
\begin{align*}
    d_h(y,z)&= \inf_{\gamma\subset W}L(\gamma,h)\\
    &\geq \alpha_2^{-1}\inf_{\gamma\subset \theta_y(\overline{B})}L(\gamma,(\theta_y)_*h_{eucl})\\
    &=\alpha_2^{-1}L(\theta_y(c_1),(\theta_y)_*h_{eucl})\\
    &=\alpha_2^{-1}L(c_1,h_{eucl}),
\end{align*}
where the infimum are taken over smooth curves $\gamma$ from $y$ to $z$. Hence we obtain
\begin{equation*}
    \tilde{\delta}(y,z)\leq\sqrt{n}\,\alpha_1\alpha_2d_h(y,z).
\end{equation*}
Hence $\tilde{\delta}\leq \alpha d_h$ on a neighborhood of $x$, where $\alpha=\sqrt{n}\,\alpha_1\alpha_2$. Hence $\delta_M$ is always locally bounded above by a Riemannian distance. In particular, it follows that $\delta_M$ is continuous.

    We can now turn to the proof of the several equivalences between the above propositions. Implications $(3)\Rightarrow(2)\Rightarrow (1)$ are immediate.     
    To prove implication $(1)\Rightarrow(4)$, assume that $F_M$ is not locally bounded below by a Riemannian norm. Then there exists a sequence $(v_k)\in C(TM)$ converging to a nonzero $v\in C(TM)$, such that $F_M(v_k)\to 0$. Up to extraction, we can find a sequence $(\gamma_k)$ such that $\gamma_k\in\PPL((-2^k,2^k),M)$ for all $k\geq 0$, such that $\gamma_k^\prime(0)=v_k$. Let $(h_k)$ be a sequence of homographies such that $\alpha_k=\gamma_k\circ h_k$ has the same 1-jet as $\gamma_k$ at $t=0$, and such that the 2-jet of $\alpha_k$ at $t=0$ converges as $k\to \infty$. Explicitly, we can write $h_k(s)=\frac{s}{1+st_k}$, for some sequence $(t_k)\in \R$. Up to extraction, we can assume that $t_k\geq 0$ for simplicity. From Corollary \ref{cor_conv_geod_avec_2jet}, there exists $\varepsilon>0$ such that $\alpha_k$ converges smoothly on the interval $[-\varepsilon,\varepsilon]$ to a curve $\alpha$.
    Let $s\in[0,\varepsilon]$ and let $k\geq 0$ be large enough so that $\varepsilon\leq 2^{k-1}$. Then, since $\vert h_k^\prime(s)\vert\leq\varepsilon$, one has 
    $$F_M(\alpha_k^\prime(s))=\vert h_k^\prime(s)\vert F_M(\gamma_k^\prime(h_k(s)))\leq \frac{\varepsilon2^{1-k}}{1-(h_k(s)/2^k)^2}\leq \frac{\varepsilon2^{1-k}}{1-(\varepsilon/2^k)^2}\leq  \varepsilon2^{2-k}.$$
    From Theorem~\ref{thm_lien_dmark_Fmark}, one has $\delta_M(\alpha_k(s),\alpha_k(t))\leq \varepsilon^2 2^{2-k}$ for all $s,t\in [0,\varepsilon]$ and for all $k\geq 0$ large enough. Since $\delta_M$ is continuous, one has $\delta_M(\alpha(0),\alpha(s))=0$ for all $s\in [0,\varepsilon]$. Now $\alpha^\prime(0)=v\neq 0$, so the restriction of $\alpha$ to $[0,\varepsilon]$ is non-constant. Hence $\delta_M$ is not a distance on $M$. 
    
    We now show that $(4)\Rightarrow(3)$. Assume that $F_M$ is locally bounded below by a Riemannian norm restricted to $C(TM)$. Let $h$ be any Riemannian metric on $M$ and let $\Vert \cdot \Vert_h$ denote its norm. For $x\in M$, let 
    $$f(x)=\inf_{v\in C(T_xM)\setminus\{0\}}F_M(v/\Vert v\Vert_h).$$ 
    Then, on every compact subset, the function $f$ is bounded below by a positive constant. If $\hat{f}:M\to \R_{>0}$ is any smooth function such that $f\geq \hat{f}$, then the Riemannian metric $\hat h= \hat{f}\cdot h$ satisfies $F_M\geq \Vert\cdot\Vert_{\hat h}$. From Theorem~\ref{thm_lien_dmark_Fmark}, we obtain $\delta_M\geq d_{\hat h}$. Hence $\delta_M$ is locally equivalent to $d_{\hat h}$ by the above argument. 

    Implication $(3)\Rightarrow(5)$ is clear since every $\gamma\in \PPL(I,M)$ is a 1-Lipschitz map $\gamma:(I,d_I)\to (M,\delta_M)$.
    We now show that $(5)\Rightarrow(4)$. Let $(v_k)\in C(TM)$ be a sequence of lightlike vectors converging to some $v\neq 0$. One can find a sequence $(\gamma_k)\in\PPL(I,M)$ with $\gamma_k^\prime(0)=\frac{1}{2\alpha_k} v_k$, where $\alpha_k\leq F_M(v_k)+1/k$. Since $\gamma_k$ is equicontinuous, there exists $\varepsilon>0$, such that a subsequence of $\gamma_k\vert_{(-\varepsilon,\varepsilon)}$ converges uniformly on compact subsets to some $\gamma:(-\varepsilon,\varepsilon)\to M$. From Proposition~\ref{prop_topologie_PPL}, the convergence is smooth and therefore $\gamma_k^\prime(0)\to \gamma^\prime(0)\in TM$. This implies that $F_M(v_k)\not \to 0$ as $k\to\infty$, so $F_M$ is locally bounded below by a Riemannian norm. 

    \indent Finally, we assume that $(M,\delta_M)$ is a metric space, and we show that $(M,\delta_M)$ is a length metric space. Denote by $L(\gamma,\delta_M)$ the metric space length of a curve $\gamma$ with respect to the distance $\delta_M$. Let $x,y\in M$ and let $\varepsilon>0$. There exists a chain $\Ccal=\{N,(\gamma_k),(s_k),(t_k)\}$ from $x$ to $y$ such that $\delta(x,y)\leq L(\Ccal)\leq \delta(x,y)+\varepsilon$. Let $\gamma$ be the curve from $x$ to $y$ obtained as the concatenation of the curves $\gamma_k\vert_{[s_k,t_k]}$.  For all $k\in \{1,\dots,N\}$, since $\gamma_k:(I,d_I)\to(M,\delta_M)$ is 1-Lipschitz by definition of $\delta_M$, we get 
$$L(\gamma_k\vert_{[s_k,t_k]},\delta_M)\leq d_I(s_k,t_k).$$
Summing over all $k\in \{1,\dots,N\}$, one obtains 
$$L(\gamma,\delta_M)\leq \sum_{k=1}^N d_I(s_k,t_k)= L(\Ccal)\leq \delta(x,y)+\varepsilon,$$
so $\delta_M$ is a length metric.
\end{proof}

We state the following corollary of the preceding proof, which can be used to give bi-Lipschitz estimates on the Markowitz distance. 

\begin{cor}\label{cor_equivalence_delta_et_d_g}
    Let $(M,[g])$ be a Markowitz hyperbolic conformal manifold. Assume that there exists a Riemannian metric $h$ and a subgroup $\Gamma\subset\Conf(M,[g])$, such that $h$ is $\Gamma$-invariant and $M/\Gamma$ is compact. Then $\delta_M$ is equivalent to the Riemannian distance $d_h$.
\end{cor}

\begin{proof}
      Let $\Vert\cdot\Vert_h$ denote the Riemannian norm of $h$, and let $\Kcal\subset M$ be a compact subset that intersects every $\Gamma$-orbit. Since $M$ is Markowitz hyperbolic, there exists $\alpha>1$, such that 
      \begin{equation}
        \label{double_ineg}
          \alpha^{-1} \Vert\cdot\Vert_h\leq F_M\leq \alpha \Vert\cdot\Vert_h
      \end{equation} 
      holds on $\Kcal$, see Theorem~\ref{thm_hyp}. Since both $h$ and $F_M$ are $\Gamma$-invariant, Equation~(\ref{double_ineg}) actually holds on $M$. For $x,y\in M$, let $\tilde\delta(x,y)=\inf \int \Vert\gamma^\prime(t)\Vert_h dt$, where the infimum runs over all piecewise lightlike geodesics from $x$ to $y$. Note that $\tilde\delta$ was already defined in the proof of Theorem~\ref{thm_hyp}. From Theorem~\ref{thm_lien_dmark_Fmark}, one has 
      $\alpha^{-1}\tilde \delta \leq \delta_M\leq \alpha\tilde \delta$. Therefore, it remains to show that $\tilde\delta$ is equivalent to $d_h$. We already have $d_h\leq \tilde\delta$ by construction. From the proof of Theorem~\ref{thm_hyp} and the compactness of $\Kcal$, we can find $\beta>1$ such that $\tilde\delta\leq\beta d_h$ holds on a neighborhood $U$ of $\Kcal$. Let $\alpha:[0,1]\to M$ be a continuous curve. One can divide $\alpha$ in curves $\alpha_1,\dots,\alpha_m$ all of which are contained in at least one element of the open covering $\{\gamma\cdot U\}_{\gamma\in \Gamma}$. Let $1\leq i\leq m$ and let $\gamma_i\in \Gamma$ such that $\alpha_i\subset \gamma_i^{-1}\cdot U$. Since $d_h$ and $\tilde\delta$ are $\Gamma$-invariant, one has $$L(\alpha_i,\tilde\delta)=L(\gamma_i(\alpha_i),\tilde\delta)\leq \beta L(\gamma_i(\alpha_i),d_h)=\beta L(\alpha_i,d_h),$$
      where $L$ refers to the metric length. By additivity of the length, we deduce that
      $L(\alpha,\tilde\delta)\leq \beta L(\alpha, d_h).$
      Since $\tilde\delta$ is a length metric, it follows that the inequality $\tilde\delta\leq\beta d_h$ holds globally on $M$. Hence $d_h$ and $\tilde\delta$ are equivalent.
\end{proof}

\subsection{Total conformal lightlike incompleteness}\label{Section_total_conf_lightlike_incompletness}
We now discuss a notion that is related to Markowitz hyperbolicity.

\begin{definition}\label{def_tot_conf_lightlike_incompletness}
A conformal manifold $(M,[g])$ is said to be \emph{totally conformally lightlike incomplete} if every projectively parametrized lightlike geodesic $\R\to M$ is constant.
\end{definition}

Total conformal lightlike incompleteness is the pseudo-Riemannian analog of Brody hyperbolicity in complex geometry, see \cite{Brody}.

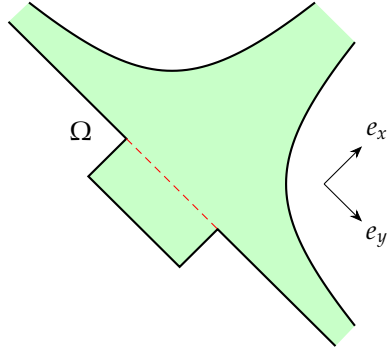
\begin{figure}
    \centering

\begin{tikzpicture}[scale = 1] 
        \def\a{0.6}
        \def\c{1.05}
        \def\rayon{1.5}
        \def\eps{0.5}

        \fill[bovert] ({\rayon*sinh(\c)},{\rayon*cosh(\c)}) -- ({\rayon*sinh(-\c)},{\rayon*cosh(-\c)}) -- (-2.15,2.15)--(-\a,\a)--(-\a-\eps,\a-\eps)--(\a-\eps,-\a-\eps)--(\a,-\a)--(2.15,-2.15)-- ({\rayon*cosh(\c)},{\rayon*sinh(-\c)}) -- ({\rayon*cosh(\c)},{\rayon*sinh(\c)});

        \fill[white,domain=-\c-0.1:\c+0.1,smooth,variable=\u,thin] 
      plot ({\rayon*sinh(\u)},{\rayon*cosh(\u)});

      \fill[white,domain=-\c-0.1:\c+0.1,smooth,variable=\u,thin] 
      plot ({\rayon*cosh(\u)},{\rayon*sinh(\u)});

        \draw[thick] (-2.15,2.15)--(-\a,\a)--(-\a-\eps,\a-\eps)--(\a-\eps,-\a-\eps)--(\a,-\a)--(2.15,-2.15);
        \draw[densely dashed,red] (-\a,\a) -- (\a,-\a);
         \draw[domain=-\c:\c,smooth,variable=\u,thick] 
      plot ({\rayon*sinh(\u)},{\rayon*cosh(\u)});

      \draw[domain=-\c:\c,smooth,variable=\u,thick] 
      plot ({\rayon*cosh(\u)},{\rayon*sinh(\u)});
        
          \draw[->,>=Stealth] (2,0) -- (2.5,0.5);
          \draw[->,>=Stealth] (2,0) -- (2.5,-0.5);
          \draw node at (2.7,0.7) {$e_x$};
          \draw node at (2.7,-0.7) {$e_y$};
          \draw node at (-1.2,0.7) {$\Omega$};

        \end{tikzpicture}
        
        \caption{\sloppy The domain $\Omega=(\max\{\vert x\vert,\vert y\vert\}<1)\cup (\vert xy\vert<1\text{ and }x>0)$ contained in Minkowski space $(\R^{1,1},[dxdy])$ that is totally conformally lightlike incomplete but not Markowitz hyperbolic. The Markowitz pseudodistance of $\Omega$ is identically zero on the red dashed segment $(x=0)\cap \Omega$.}
        \label{Omega_slip_non_Mhyp}
\end{figure}

\begin{prop}\label{prop_TCLI}
    Let $(M,[g])$ be a conformal manifold. Then, one has:
    \begin{enumerate}
        \item $M$ is Markowitz hyperbolic $\Rightarrow$ $M$ is totally conformally lightlike incomplete.
        \item $M$ is totally conformally lightlike incomplete $\iff$ $F_M$ is positive (that is $v\neq 0\Rightarrow F_M(v)>0$).
    \end{enumerate}
\end{prop}

\begin{proof}
    If $M$ is Markowitz hyperbolic, then $F_M$ is locally bounded below by a Riemannian norm (Theorem \ref{thm_hyp}), so $F_M$ is positive. Therefore assertion $(1)$ follows from assertion $(2)$. It remains to prove assertion $(2)$.
    Assume that $(M,[g])$ contains a non-constant $\gamma\in\PPL(\R,M)$. Then, for $k\geq 1$, the curve $\gamma_k:I\to M$ defined by $\gamma_k(s)=\gamma(ks)$ for $s\in I$ is a projectively parametrized lightlike geodesic. Hence, if we let $v=\gamma^\prime(0)\neq 0$, one has
    $$F_M(v)=\frac{1}{k}F_M(\gamma_k^\prime(0))\leq \frac{1}{k}\sqrt{\rho_I(\partial_t,\partial_t)}=\frac{2}{k},$$
    hence $F_M(v)=0$. Therefore $F_M$ is not positive. Conversely, assume that $M$ is totally conformally lightlike incomplete, and let $v$ be a nonzero lightlike vector. Let $\gamma:J\to M$ be a maximal projectively parametrized lightlike geodesic tangent to $v$, where $J\subset\RPunTilde$ is a connected interval (such a geodesic exists by Proposition~\ref{prop_str_proj_sur_les_geod_de_lum}). Since $M$ is totally conformally lightlike incomplete, the domain $J$ is projectively equivalent to $I$, so we can always assume that $J=I$, up to a reparametrization. If $u\in TI$ is such that $\gamma_*u=v$, then 
    $$F_M(v)=\sqrt{\rho_I(u,u)}>0.$$
    Hence $F_M$ is positive.
\end{proof}

\begin{rmk}
    A totally conformally lightlike incomplete manifold may fail to be Markowitz hyperbolic, see Figure~\ref{Omega_slip_non_Mhyp}.
\end{rmk}

\subsection{Complete hyperbolicity} We now briefly discuss the completeness of $\delta_M$, similarly to \cite{Royden}.

\begin{prop}
    \label{prop_cara_complete_hyperbolic}
    Let $(M,[g])$ be a conformal manifold. Then one has:
    \begin{enumerate}        
        \item $(M,\delta_M)$ is complete $\iff$ closed balls are compact.
        \item $(M,\delta_M)$ is complete $\Rightarrow$ $F_M$ is positive and continuous $\Rightarrow$ $M$ is Markowitz hyperbolic. 
    \end{enumerate}
\end{prop}

\begin{proof}
    Assertion $(1)$ follows from the fact that $\delta_M$ is a length metric and Hopf-Rinow theorem \cite[Prop. 3.7]{bridson2013metric}. Assume now that $(M,\delta_M)$ is complete. Then $F_M$ is positive by Proposition \ref{prop_TCLI}.  Let $(v_k)$ be a sequence of lightlike vectors converging to some $v\in C(TM)$, so that $F_M(v_k)=2$ for all $k\geq 0$. Then we can find a sequence $\gamma_k\in\PPL(I,M)$ such that $\gamma_k^\prime(0)=v_k$ for all $k\geq 0$. Now $\gamma_k:(I,d_I)\to (M,\delta_M)$ is 1-Lipschitz for all $k\geq 0$, hence there exists a subsequence of $(\gamma_k)$ that converges uniformly on compact subsets to some curve $\gamma:I\to M$, by the Arzelà-Ascoli theorem. Proposition \ref{prop_topologie_PPL} implies that $\gamma\in \PPL(I,M)$ and that the convergence is smooth. In particular $v=\gamma^\prime(0)$, so 
    $F_M(v)\leq 2$ by definition of $F_M$. This shows that $F_M$ is lower-semicontinuous, hence $F_M$ is continuous. Thus, the first implication of assertion $(2)$ is proved. Assume now that $F_M$ is positive and continuous. Then $F_M$ is locally bounded below by a Riemannian norm, so $M$ is Markowitz hyperbolic by Theorem \ref{thm_hyp}.
\end{proof}

\subsection{Example: the diamond}\label{Section_ex_diamant} The \emph{diamond} of signature $(p,q)$ is the conformal manifold defined as $\D^{p,q}=\H^p\times\H^q$ endowed with the conformal class of the product metric $-g_{\H^p}+g_{\H^q}$. This conformal spacetime is conformally flat, and it can be identified with a proper domain of $\Ein^{p,q}$, see \cite[Sect. 3]{Cha_Gal}. In Lorentzian signature $(p,q)=(1,n)$, the diamond $\D^{1,n}=\R\times \H^n$ is conformally equivalent to a chronal diamond $I(x,y)$ of the Minkowski space $\R^{1,n}$. The diamond is Markowitz hyperbolic, and its Markowitz pseudodistance can be fully computed. 

\begin{prop}[{\cite[Ex. 4.2]{Cha_Gal}}]\label{prop_distance_diamant}
The Markowitz distance of $\D^{p,q}=\H^p\times \H^q$ between two points $x=(x_p,x_q)$ and $y=(y_p,y_q)$ is given by 
$$\delta_{\D^{p,q}}(x,y)=\max\{d(x_p,y_p),d(x_q,y_q)\},$$
where $d$ denotes the hyperbolic distance.
\end{prop}

\subsection{Examples: HB-domains}\label{section_misner} Let $n\geq 1$ and equip $\R^{1,n}$ with the standard flat metric $\b$ and a time orientation. 
Recall from the introduction that for $0\leq \ell\leq n$, a \emph{future HB-domain} is the future $\Omega_\ell=I^+(F_\ell)$ of a $\ell$-dimensional spacelike subspace $F_\ell\subset \R^{1,n}$, see Figure~\ref{Figure_Misner}.

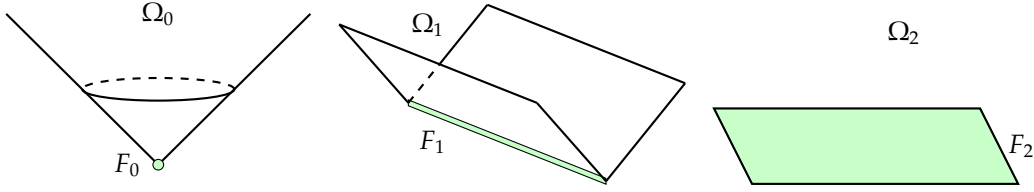
\begin{figure}
    \centering
    \begin{tabular}{ccc}

        \begin{tikzpicture}[scale=1]
            \draw[line width = 0.8pt] (-2,2) -- (0,0) -- (2,2); 
         \fill[red,line width = 0.8pt] (0,0) circle (1pt);
              \draw[dashed,line width = 0.8pt] (1,1) arc (0:180:1 and 0.15);
          \draw[line width = 0.8pt] (-1,1) arc (180:360:1 and 0.15);
          \draw node at (-0.4,0) {$F_0$};
          \filldraw[fill=bovert] (0,0) circle (2pt);
          \node at (0,2) {$\Omega_0$};
        \end{tikzpicture}
        
        & 
        
        \begin{tikzpicture}[scale=1.3]
        \draw node at (-0.75,0) {$F_1$};
        \def\a{1}
        \def\B{0.4}
        \def\c{0.8}
        \def\d{1}
        \filldraw[fill=bovert] (\a,-\B+0.03) -- (-\a,\B+0.03) -- (-\a,\B-0.03) -- (\a,-\B-0.03) -- (\a,-\B+0.03);
        \draw[line width = 0.8pt]  (\a+\c,-\B+\d) -- (-\a+\c,\B+\d);
        \draw[line width = 0.8pt]  (\a,-\B) -- (\a+\c,-\B+\d);
        \draw[dashed, line width = 0.8pt] (-\a,\B) -- (-\a+0.4*\c,\B+0.4*\d);
        \draw[line width = 0.8pt]  (-\a+0.4*\c,\B+0.4*\d) -- (-\a+\c,\B+\d);
        \def\e{-0.7}
        \def\f{0.8}
        \draw[line width = 0.8pt]  (\a+\e,-\B+\f) -- (-\a+\e,\B+\f);
        \draw[line width = 0.8pt]  (\a,-\B) -- (\a+\e,-\B+\f);
        \draw[line width = 0.8pt]  (-\a,\B) -- (-\a+\e,\B+\f);
        \node at (-0.8,1.2) {$\Omega_1$};
         \end{tikzpicture}

         &
        
    \begin{tikzpicture}[scale=0.5]
    \filldraw[fill=bovert, opacity=0.7] 
        (-5,2) -- (-4,0) -- (3,0) -- (2,2) -- cycle;
    \draw[thick] 
        (-5,2) -- (-4,0) -- (3,0) -- (2,2) -- cycle;
        \node at (3.1,1) {$F_2$};
        \node at (0,4) {$\Omega_2$};
    \end{tikzpicture}
            
    \end{tabular}
    \caption{The three future HB-domains of $\R^{1,2}$.}
    \label{Figure_Misner}
\end{figure}

\begin{prop}
    \label{prop_equivalence_d_Omega_ell}
    For $0\leq \ell\leq n$, the conformal spacetime $\Omega_\ell$ is Markowitz hyperbolic, and $(\Omega_\ell,\delta_{\Omega_\ell})$ is bi-Lipschitz equivalent to $\H^{n-\ell}\times \H^{\ell+1}$.
\end{prop}

\begin{proof}
    For $x\in \Omega_\ell$, we let $x_\ell^\perp$ denote the orthogonal projection of $x$ onto $F_\ell^\perp$. Define a Lorentzian metric $g_\ell$ on $\Omega_\ell$ by 
    $$g_\ell=\frac{1}{\Vert x_\ell^\perp\Vert^2}\b,$$
    where $\Vert\cdot\Vert=\sqrt{\vert \b(\cdot,\cdot)\vert}$ is the Lorentzian norm on $\R^{1,n}$. Let $\mathcal{D}\subset T\Omega_\ell$ be the timelike foliation defined for all $x\in\Omega_\ell$ by 
    $\mathcal{D}_x=\R x_\ell^\perp.$
    Let $\Gamma\subset\Conf(\R^{1,n})$ be the group of similarities that preserve the orthogonal splitting $\R^{1,n}=F_\ell\oplus F_\ell^\perp$ and the time orientation of $\R^{1,n}$. Then, one can check directly that $\Gamma$ preserves $\Omega_\ell$ and acts transitively on it, so $\Omega_\ell/\Gamma$ is compact (when $\ell\geq1$, one can deduce from Liouville's theorem that $\Gamma=\Conf(\Omega_\ell)$, but we will not need this here). Also, the group $\Gamma$ preserves both $g_\ell$ and $\mathcal{D}$.
 
    Let $g_\ell^+$ be the Wick rotation of $g_\ell$ by $\mathcal{D}$, i.e. $g_\ell^+$ is the unique Riemannian metric on $\Omega_\ell$ that coincides with $g_\ell$ and $-g_\ell$ on $\mathcal{D}^\perp$ and $\mathcal{D}$, respectively, and such that $\mathcal{D}^\perp$ and $\mathcal{D}$ are orthogonal with respect to $g_\ell^+$. Since $g_\ell$ and $\mathcal{D}$ are $\Gamma$-invariant, the metric $g_\ell^+$ is also $\Gamma$-invariant.  From Corollary~\ref{cor_equivalence_delta_et_d_g}, the Markowitz distance $\delta_{\Omega_\ell}$ is equivalent to the Riemannian distance of $g_\ell^+$. Now, the spacetime $(\Omega_\ell,g_\ell)$ is isometric to 
    $$((\R_{>0}\times\H^{n-\ell})\times_{t^{-1}}\R^\ell, g_\ell=-t^{-2}dt^2+g_{\H^{n-\ell}}+t^{-2}g_{\R^\ell}),$$
    where an explicit isometry is given by the map $x\mapsto (\Vert x_\ell^\perp\Vert^2,x/\Vert x_\ell^\perp\Vert^2,x-x_\ell^\perp)$. Under this identification, the metric $g_\ell^+$ is given by 
    $$g_\ell^+=t^{-2}dt^2+g_{\H^{n-\ell}}+t^{-2}g_{\R^\ell}=g_{\H^{n-\ell}}+t^{-2}(dt^2+g_{\R^\ell}).$$
    Hence $(\Omega_\ell,g_\ell^+)$ is isometric to $\H^{n-\ell}\times \H^{\ell+1}$.
\end{proof}

\begin{rmk}
    \label{remarque_hyp_de_Omega_l}
    Markowitz hyperbolicity of $\Omega_\ell$ can also be deduced directly from Theorem~\ref{thm_B2}: the domain $\Omega_\ell$ is convex in $\R^{1,n}$, hence it identifies with a conformally convex domain of $\Ein^{1,n}$ (we anticipate some terminology here, see also Section~\ref{section_conf_convex}). Since $\Omega_\ell$ contains no lightlike line in $\R^{1,n}$, it is totally conformally lightlike incomplete. Then Theorem~\ref{thm_B2} implies that $\Omega_\ell$ is Markowitz hyperbolic.
\end{rmk}

\begin{rmk}
    When $\ell=n$ the HB-domain $\Omega_n\subset \R^{1,n}$ is an Einstein-de Sitter half-space. Its Markowitz distance was explicitly computed by Markowitz, see \cite[Thm. 5]{Markowitz_warped_product}. 
    Proposition~\ref{prop_equivalence_d_Omega_ell} implies in particular that $(\Omega_n,\delta_{\Omega_n})$ is bi-Lipschitz isometric to the real $n$-dimensional hyperbolic space $\H^n$. In particular, $(\Omega_n,\delta_{\Omega_n})$ is a Gromov hyperbolic metric space. Other domains of Minkowski space that share this property can be found in \cite{article_hyp_de_Gromov}.
\end{rmk}

\section{Compactness for conformal maps}
\label{section_compacité}

Given a conformal map $f:M\to N$ and given two metrics $g$ and $\tilde{g}$ in the conformal class of $M$ and $N$, respectively, the \emph{conformal distortion} of $f$ at $x$ is defined as the unique positive number $\dc(f,x)_{g,g^\prime}>0$, such that 
$$f^*\tilde{g}\vert_x=(\dc(f,x)_{g,\tilde{g}})^2g\vert_x.$$
In other words, the conformal distortion is the obstruction for a map $f$ to be isometric at $x$ with respect to the metrics $g$ and $\tilde{g}$.

\begin{prop}
\label{prop_compacité_app_conformes}
    Let $M$ and $N$ be two connected Markowitz hyperbolic conformal manifolds. Let $(x_k)\in M$ be a sequence converging to a point $x\in M$ and let $(f_k)\in \Conf(M,N)$ be such that $y_k=f_k(x_k)$ converges to some $y\in N$. Then, the following properties are equivalent:
    \begin{enumerate}
        \item the sequence $(f_k)$ is contained in a compact subset of $\Conf(M,N)$,
        \item for some (hence any) metrics $g$ and $\tilde{g}$ on $M$ and $N$ respectively, one has  
            $\inf_k\dc(f_k,x_k)_{g,\tilde{g}}>0$.
    \end{enumerate}
    If we assume in addition that the maps $f_k$ are injective for all $k\geq 0$, then these conditions are equivalent to the following one:
    \begin{enumerate}
        \setcounter{enumi}{2}
        \item there exists a neighborhood $U$ of $y$, such that $U\subset f_k(M)$ for large enough $k\geq 0$.
    \end{enumerate}
\end{prop}

\begin{proof} $(1)\Rightarrow (2)$ and $(1)\Rightarrow(3)$ are clear. Let us show the converse implications. Since $N$ is locally compact, we can find a radius $\varepsilon>0$ so that $B_{\delta_N}(y,3\varepsilon)$ has compact closure in $N$. Let $V=B_{\delta_M}(x,\varepsilon)$. Let $k\geq 0$ be large enough so that $\delta_M(x_k,x)<\varepsilon$ and $\delta_N(y_k,y)<\varepsilon$. Then, for $a\in V$, one has
    \begin{align*}
        \delta_N(f_k(a),y)&\leq \delta_N(f_k(a),f_k(x_k))+ \delta_N(f_k(x_k),y)\\
        &\leq \delta_M(a,x_k)+ \delta_N(y_k,y)\\
        &\leq \delta_M(a,x)+\delta_M(x,x_k)+ \delta_N(y_k,y)\leq 3\varepsilon.        
    \end{align*}
    Therefore $f_k(V)\subset B_{\delta_N}(y,3\varepsilon)$ for large enough $k\geq 0$. Since $(f_k)$ is equicontinuous by Proposition~\ref{proposition_naturalité}, we can find a subsequence $(f_{m(k)})$ converging uniformly on compact subsets of $B_{\delta_M}(x,r)$ to a map $f:B_{\delta_M}(x,r)\to N$. From \cite[Thm. 1.2]{Charles_transformations_conformes}, the limit map $f$ is smooth, and the convergence is smooth on compact subsets of $B_{\delta_M}(x,r)$. Also, the smooth map $f$ is either conformal or a locally trivial bundle onto a totally degenerate submanifold. We will now show that the second alternative cannot occur under hypothesis $(2)$, or under hypothesis $(3)$, if the maps $f_k$ are assumed injective for all $k\geq 0$. 

    Assume that $(2)$ holds. Let $v\in T_{x}M$ be a timelike vector. One has
    \begin{equation*}
        \dc(f_{m(k)}, g, \tilde{g}, x)^2=\tilde{g}(T_xf_{m(k)}\cdot v,T_xf_{m(k)}\cdot v)/g(v,v)\longrightarrow \tilde{g}(T_xf\cdot v,Tf\cdot v)/g(v,v)>0.
    \end{equation*}
    Hence $\tilde{g}(T_xf\cdot v,Tf\cdot v)\neq 0$, so $f$ cannot take values in a totally degenerate submanifold. Hence $f$ is conformal. 

     Assume that $(3)$ holds and that the maps $f_k$ are injective. We can always decrease $\varepsilon>0$ if necessary and assume that $\overline{B}_{\delta_N}(y,3\varepsilon)\subset U$. Let $a,b\in V$. Then, since the maps $f_k$ are injective, the restriction $f_k:M\to f_k(M)$ is a diffeomorphism, hence 
     $$\delta_M(a,b)= \delta_{f_k(M)}(f_k(a),f_k(b)).$$
     Also, since $U\subset f_k(M)$ for large enough $k\geq 0$, one has $\delta_{f_k(M)}\leq \delta_U$, hence 
    $$\delta_M(a,b)\leq \delta_U(f_k(a),f_k(b)).$$
    In particular, one obtains
    $$\delta_M(a,b)\leq \delta_U(f(a),f(b)),$$
    so $f$ is injective on $V$. Therefore, the map $f$ cannot be a locally trivial bundle over a degenerate submanifold. Hence $f$ is a conformal map. 

    From Lemma~\ref{Lemme_convergence_locale_app_conf} below, we deduce that $f$ extends to a smooth conformal map and that $(f_{m(k)})$ converges smoothly to $f$ on compact subsets of $M$.
\end{proof}
It remains to prove the following standard lemma that we have used to conclude the proof of Proposition~\ref{prop_compacité_app_conformes}.

\begin{lem}
\label{Lemme_convergence_locale_app_conf}
    Let $f_k:M\to N$ be a sequence of conformal maps, where $M$ and $N$ are connected conformal manifolds. Assume that there exists a nonempty open subset $U$ of $M$, such that $f_k\vert_U$ converges smoothly on compact subsets of $U$ to a conformal map $f:U\to N$. Then $f$ extends to a conformal map $f:M\to N$, and the sequence $(f_k)$ converges smoothly to $f$ on compact subsets of $M$. 
\end{lem}

This lemma is better proved using the viewpoint of Cartan connections.  Recall from \cite{Sharpe} that a Cartan geometry modeled on the homogeneous space $\Ein^{p,q}=\PO(p+1,q+1)/P$, where $P$ is the stabilizer of an isotropic line inside $G=\PO(p+1,q+1)$, is the data of 
\begin{itemize}
    \item a principal $P$-bundle $\pi:\Mhat\to M$.
    \item a parallelism $\omega:T\Mhat\to \g$ (i.e. $\omega_{\hat{x}}:T_{\hat{x}}\Mhat\to \g$ is a linear isomorphism for all $\hat{x}\in \Mhat$), called the \emph{Cartan connection}, such that the following two properties hold: 
    \begin{enumerate}
        \item $\omega(\frac{d}{dt}\vline_{t=0} \hat{x}\cdot e^{tv})=v$ for every vector $v\in \mathfrak{p}$ and point $\hat{x}\in \Mhat$;
        \item $(R_h)^*\omega=\operatorname{Ad}(h^{-1})\cdot \omega$ for all $h\in P$.
    \end{enumerate} 
\end{itemize}

A Cartan geometry $\pi:\Mhat\to M$ naturally defines a conformal class of pseudo-Riemannian metrics of $M$ in the following way. Let $x\in M$ and let $\hat{x}\in\pi^{-1}(\{x\})$. Since $T_{\hat{x}}(\hat{x}\cdot P)=\ker T_{\hat{x}}\pi$, the tangent map $T_{\hat{x}}\pi$ gives an identification between $T_{\hat{x}}\Mhat/T_{\hat{x}}(\hat{x}\cdot P)$ and $T_xM$. Also the Cartan connection factors through an isomorphism 
$$\omega_{\hat{x}}:T_{\hat{x}}\Mhat/T_{\hat{x}}(x\cdot P)\to \g/\p.$$
By combining these identifications, one obtains an isomorphism 
$j_{\hat{x}}:T_xM\to \g/\p$ which only depends on the choice of  ${\hat{x}}\in\pi^{-1}(\{x\})$. Given another lift ${\hat{x}}\cdot h$ of $x$, where $h\in P$, the two identifications are related by the relation
$$j_{\hat{x}\cdot h}=\operatorname{Ad}(h^{-1})\circ j_{\hat{x}}.$$
Hence $T_xM$ is identified to $\g/\p$ modulo the adjoint action of $P$. Now, using a stereographic projection, the quotient $T_P\Ein^{p,q}=\g/\p$ identifies with $\R^{p,q}$ and the adjoint action of $P$ is conjugated to the conformal action of the similarity group of  $(\R^{p,q},[\b])$. Hence there exists a $\operatorname{Ad}(P)$-invariant conformal class $[g_{\g/\p}]$ of pseudo-Riemannian metric on $\g/\p$, so the class $j_{\hat{x}}^*[g_{\g/\p}]$ is a well defined conformal class of pseudo-Riemannian metrics on $M$.

Given a conformal manifold $(M,[g])$, a Cartan geometry $\pi:\Mhat\to M$ modeled on $\Ein^{p,q}$ is said to be \emph{compatible to the conformal structure} if $[g]$ is obtained with the same procedure of the preceding paragraph. In general, there  exist infinitely many non-equivalent compatible Cartan geometries on a conformal pseudo-Riemannian manifold. However, for every conformal manifold, there exists a unique \emph{normal Cartan connection} that is compatible to the conformal class, see \cite{Sharpe}.
We will write $(\Mhat,\omega_M)$ for the normal Cartan connection associated to $M$. Since the Cartan connection is uniquely determined by the conformal structure of $(M,[g])$, a conformal map $f:M\to N$ will lift to a unique parallelism preserving map 
$\hat{f}:(\Mhat,\omega_M)\to (\hat{N},\omega_N)$, i.e. $\hat{f}^*\omega_N=\omega_M$. 

\begin{proof}[Proof of Lemma~\ref{Lemme_convergence_locale_app_conf}]
     Let us fix some notations first. For a vector $v\in \g$, let $\hat{v}$ denote the vector field on $\hat{M}$ and $\hat{N}$, such that $\omega(\hat{v})=v$. Let $\langle\cdot\,,\cdot\rangle$ be an inner product on $\g$ and let $g_M=\omega_M^*\langle\cdot\,,\cdot\rangle$ and $g_N=\omega_M^*\langle\cdot\,,\cdot\rangle$ be the induced Riemannian metrics on $\hat{M}$ and $\hat{N}$, respectively. For $v\in \g$ and $\hat{x}\in \hat{M}$, we denote $\exp(\hat{x},v):=\exp(\hat{x},\hat{v}_{\hat{x}})$, where $\exp(\hat{x},\hat{v}_{\hat{x}})$ is the exponential map of $g_M$.  
    Let $x\in U$ and let $\hat{x}\in \hat{M}$ be a lift of $x$, and let $(\hat{f_k})$ be a sequence of lifts of $(f_k)$. Let $y\in N$ denote the limit of $y_k=f_k(x)$. Since $(f_k)$ converges to a conformal map on $U$, the sequence $(\hat y_k)=(\hat f_k(\hat x))$ converges in $\hat N$. This follows from the fact that $(f_k)$ and $(f_k\vert_U^{-1})$ are \emph{stable} at $x$ and $y$, respectively, see \cite[Lem. 4.3]{Charles_transformations_conformes}. Let
    $$\Omega=\{a\in \hat{M}\,\vert\,(\hat{f_k}(a))\text{ converges in $N$}\}.$$
    The set $\Omega$ is nonempty since it contains $\hat x$.  Let $a\in \Omega$ and let $b\in N$ be the limit of the sequence $(b_k:=\hat{f_k}(a))$. Let $\varepsilon>0$ be such that $\exp(a,\cdot)$ is defined on $B(0,\varepsilon)\subset \mathfrak{g}$ and maps diffeomorphically onto its image. Since $\hat{f_k}$ is isometric for $\hat{M}$, one has
    \begin{equation}
    \label{formule_exponentielle}
        \hat{f_k}(\exp(a,v))=\exp(b_k,v),
    \end{equation}
    for all $v\in B(0,\varepsilon)$. In particular,  the exponential map $\exp(b_k,\cdot)$ is defined on $B(0,\varepsilon)$ for all $k\geq 0$. 
    \emph{Claim: $\exp(b,\cdot)$ is also defined on $B(0,\varepsilon)$}. Let $0<\varepsilon^\prime<\varepsilon$ and $v\in \g$ be a unit vector so that $t\mapsto\exp(y,tv)$ is defined on the interval $[0,\varepsilon^\prime)$.  Let $k\geq 0$ be such that $d(b_k,b)\leq (\varepsilon-\varepsilon^\prime)/2$, and let $\gamma$ be a $C^1$ curve of length $L(\gamma,g_N)<(\varepsilon-\varepsilon^\prime)/2$ from $b_k$ to $b$. The concatenation of $\gamma$ and the geodesic $t\mapsto\exp(b,tv)$ yields a curve $\alpha$ of length
    $L(\alpha,g_N)<\varepsilon.$
    Since $B(a,\frac{\varepsilon+\varepsilon^\prime}{2})$ has compact closure in $\hat M$, the curve $\alpha$ lifts to a unique curve $\tilde{\alpha}$ along $\hat{f_k}$. Since $\tilde{\alpha}$ is extensible at time $t=\varepsilon^\prime$, so is $\alpha$. This proves the previous claim. From Equation~(\ref{formule_exponentielle}), we deduce that $(\hat{f_k})$ converges smoothly to 
    $\exp(b,\cdot)\circ \exp(a,\cdot)^{-1}$
    on the geodesic ball $B(a,\varepsilon)$. In particular, one has $B(a,\varepsilon)\subset\Omega$, so $\Omega$ is open. 
    Let us show that $\Omega$ is also closed in $\hat{M}$. Given a sequence $(a_k)\in \Omega^\N$ converging to $a\in \hat{M}$, we can find $\varepsilon>0$ such that $B(a_l,\varepsilon)$ is a geodesic ball for large enough $l\geq 0$ and $a\in B(a_l,\varepsilon)$. Since $\hat{f_k}\vert_{B(a_l,\varepsilon)}$ converges, the sequence $(\hat{f_k}(a))$ converges in $N$. Hence $\Omega$ is closed and $\Omega=\hat M$ by connectedness. Hence $(\hat{f_k})$ converges smoothly on compact subsets of $\hat M$ to a map $\hat{f}:\hat{M}\to \hat{N}$. Now 
    $\hat{f}^*\omega_N=\lim_k\hat{f_k}^*\omega_N=\omega_M,$
    hence $\hat f$ is parallelism preserving. Therefore $\hat{f}$ covers a conformal map $f:M\to N$ and $(f_k)$ converges smoothly to $f$ on compact subsets of $M$. 
\end{proof}

\begin{rmk}
    We briefly mention a consequence of Proposition~\ref{prop_compacité_app_conformes}.
    Given a conformally flat Markowitz hyperbolic manifold $M$ and a point $x\in M$, the set of pointed maps $\mathcal{F}_x=\Conf((\D^{p,q},a),(M,x))$ is nonempty. For a fixed $g$ and for $a\in \D^{p,q}$, the map 
    $$\dc({-},a)_{g_{\D^{p,q}},g}:\mathcal{F}_x\to \R$$ is bounded above and attains its maximum for at least one conformal map $f_x:\D^{p,q}\to M$ by Proposition~\ref{prop_compacité_app_conformes}. 
    The metric $\hat g_x=(f_x)_*g_{\D^{p,q}}$ is a well defined, conformally invariant metric on $M$. See \cite{these_Chalumeau} for a detailed study of that metric. 
\end{rmk}

\section{Closed conformal manifolds}
\label{section_Brody}

Let $M$ be a closed manifold and denote by $\Ccal(M)$ the set conformal class of pseudo-Riemannian metrics on $M$. We say that a sequence $[g_k]\in \Ccal(M)$ \emph{converges to $[g]$ in the $C^2$ topology} if for some (hence any) metric $g^*\in [g]$, there exists a sequence $(g_k^*)$ such that $g_k^*\in [g_k]$ for all $k\geq 0$ and $g_k^*\to g^*$ as $k\to \infty$ in the (standard) $C^2$ topology. This endows $\Ccal(M)$ with a topology.

\begin{thm}
    \label{thm_compacité_fort}
    Let $(M,[g])$ be a closed conformal manifold. If $M$ is totally conformally lightlike incomplete, then there exists a neighborhood $\W$ of $[g]$ for the $C^2$ topology, such that $(M,[g^\prime])$ is Markowitz hyperbolic for every $[g^\prime]\in \W$. In particular, $(M,[g])$ is Markowitz hyperbolic.
\end{thm}

In particular, the property ``$(M,[g])$ is Markowitz hyperbolic'' is open with respect to the $C^2$ topology on $[g]$.

\begin{lem}[Brody's lemma]
    Let $\Vert\cdot\Vert$ be a Riemannian norm on $M$ and let $\gamma\in\PPL(I,M)$ such that $\Vert \gamma^\prime(0)\Vert>c$ for some constant $c>0$. Then, there exists a homography $h:I\to I$, such that $\alpha=\gamma\circ h$ satisfies $\Vert\alpha^\prime(0)\Vert=c$ and $\Vert\alpha^\prime(t)\Vert\leq c/(1-t^2)$ for all $t\in I$.\qed
\end{lem}

The proof is identical to that in \cite{Brody}.

\begin{proof}[Proof of Theorem~\ref{thm_compacité_fort}]
    We assume that there exists a sequence $(g_k)$ of metrics converging to $g$ for the $C^2$ topology such that $(M,[g_k])$ is not Markowitz hyperbolic and show that $(M,[g])$ is not totally conformally lightlike incomplete. Write $F_k$ for the infinitesimal functional of $(M,[g_k])$. From Theorem~\ref{thm_hyp}, for every $k\geq 0$, there exists a sequence $(\gamma_k^\ell)\in \PPL(I,(M,[g_k]))$ such that $(\gamma_k^\ell(0))$ converges in $M$ and $\Vert(\gamma_k^\ell)^\prime(0)\Vert\to\infty$ as $\ell\to\infty$.
    By diagonal extraction, we can find a sequence $(\gamma_k)\in\PPL(I,(M,[g_k]))$ such that $(\gamma_k(0))$ converges in $M$ and $\Vert\gamma_k^\prime(0)\Vert\to\infty$ as $k\to\infty$.
    Up to extracting a subsequence, we can assume that $\Vert\gamma_k^\prime(0)\Vert\geq 2^k$ for all $k\geq 0$. From Brody's lemma, there exists a sequence $(h_k)$ of homographies such that for all $k\geq 0$, the curve $\alpha_k=\gamma_k\circ h_k$ satisfies 
    $$\Vert\alpha_k^\prime(0)\Vert=2^k\text{ and }\Vert\alpha_k^\prime(t)\Vert\leq 2^k/(1-t^2),$$
    for all $t\in I$. Since $M$ is compact, we can always extract a subsequence and assume that $(\alpha_k(0))$ converges in $M$. For $k\geq 0$, define $I_k=(-2^k,2^k)$ and $\beta_k:I_k\to M$ by $\beta_k(t)=\alpha_k(t/2^k)$ for all $t\in I_k$. For all $t\in I_k$, one has 
    $$\Vert\beta_k^\prime(t)\Vert\leq 1/(1-(t/2^k)^2),$$
    so for every compact $\Kcal\subset \R$, the family $(\beta_k\vert_\Kcal)$ is equicontinuous, hence converges uniformly to a curve $\Kcal\to M$ up to extracting a subsequence. Therefore, we may extract a subsequence of $(\beta_k)$ that converges uniformly on compact subsets of $\R$ to a curve $\beta:\R\to M$. From Proposition~\ref{prop_topologie_PPL}, the curve $\beta$ is a projectively parametrized lightlike geodesic for $[g]$ and the convergence is smooth. Since $\Vert\beta_k^\prime(0)\Vert=1$, the derivative $\beta^\prime(0)$ is nonzero, so $\beta$ is non-constant and $(M,[g])$ is not totally conformally lightlike incomplete.
\end{proof}

\begin{cor}
    Let $(M,[g])$ be a closed conformal manifold. If $\Conf(M,[g])$ is non-compact, then there exists a non-constant projectively parametrized lightlike geodesic $\R\to M$.
\end{cor}

\begin{proof}
    We assume that $M$ is totally conformally lightlike incomplete. Then Theorem~\ref{thm_compacité_fort} implies that $M$ is Markowitz hyperbolic, so the isometry group $\Iso(M,\delta_M)$ is compact in the compact-open topology. Now $\Conf(M)$ is closed in $\Iso(M,\delta_M)$, so $\Conf(M)$ is compact in the compact-open topology. It follows that $\Conf(M)$ is compact for the smooth topology, see for instance \cite{Charles_transformations_conformes}.
\end{proof}

\begin{ex}
    Let $\Sigma^p$ and $\Sigma^q$ be compact Riemannian manifolds of constant negative curvature and of dimension $p$ and $q$, respectively. The conformal manifold $M=(\Sigma^p\times \Sigma^q,[-g_{\H^p}+g_{\H^q}])$ is conformally flat and Markowitz hyperbolic, as its universal cover $\H^p\times \H^q$ is Markowitz hyperbolic, see Section~\ref{Section_ex_diamant}. Therefore a pseudo-Riemannian metric $g^\prime$ sufficiently close to $g=-g_{\H^p}+g_{\H^q}$ induces a Markowitz hyperbolic conformal structure $(M,[g^\prime])$ by Theorem \ref{thm_compacité_fort}. 
\end{ex} 

\section{Conformally flat manifolds}
\label{section_conf_plat}

The rest of the paper is devoted to conformally flat manifolds. 

\subsection{Markowitz hyperbolicity for conformally flat manifolds.}

Our aim is to prove the following characterization of Markowitz hyperbolicity for conformally flat manifolds. 

\begin{thm}
    \label{thm_hyp_conf_plat}
    Let $(M,[g])$ be a conformally flat manifold. The following properties are equivalent:
    \begin{enumerate}
        \item $M$ is Markowitz hyperbolic.
        \item The family $\Conf(\D^{p,q},M)$ is equicontinuous, with respect to some distance on $M$ that is locally equivalent to a Riemannian distance.
    \end{enumerate}     
\end{thm}

We will need to identify $\D^{p,q}$ conformally with a bounded domain $\R^{p,q}$, see also \cite{Cha_Gal}. Let $\R^{p,q}=H_p\oplus H_q$ be an orthogonal splitting into negative and positive definite subspaces $H_p$ and $H_q$, respectively, and write vectors $v=v_p+v_q$ according to this decomposition. We define a norm $\Vert\cdot\Vert_{p,q}$ on $\R^{p+q}$ as 
$$\Vert v\Vert_{p,q}=\sqrt{-\b(v_p,v_p)}+\sqrt{\b(v_q,v_q)},$$
where $\b$ is the flat metric on $\R^{p,q}$. Then any open ball $B=B_{p,q}(v,r)$ for the norm $\Vert\cdot\Vert_{p,q}$ can be conformally identified with $\D^{p,q}=(\H^p\times \H^q,[-g_{\H^p}+g_{\H^q}])$, see \cite[Sect. 3.3.2]{Cha_Gal}. Thus we obtain the following:

\begin{lem}
    \label{lemme_base_voisi_photon}
    Let $\alpha\in \PPL(I,\Ein^{p,q})$. Then the segment $\Ccal=\overline{\alpha(I)}$ admits a basis of neighborhoods $\Bcal=\{D_i\}$ consisting of open subsets each of which is conformally equivalent to $\D^{p,q}$.
\end{lem}

\begin{proof}
    Let $\Delta$ be the photon of $\Ein^{p,q}$ containing $\Ccal$, and choose a point $x\in \Delta\setminus \Ccal$ (such a point exists since $\alpha$ is projectively parametrized on $I$). Let $y\in C(x)\setminus\Delta$ and fix a stereographic projection $\st:\R^{p,q}\to \Ein^{p,q}\setminus C(y)$ at $y$. Since $C(y)\cap \Delta=\{x\}$, the segment $ \Ccal$ is disjoint from $C(y)$, hence it appears as a bounded lightlike segment in the stereographic projection.     
    Let $\R^{p,q}=H_p\oplus H_q$ be an orthogonal decomposition as above, and let $u=u_p+u_q\in\R^{p,q}$ be a lightlike vector such that $\mathcal{C}=\{tu\,\vert\,t\in[-1,1]\}$ with $\b(u_p,u_p)=-\b(u_q,u_q)=-1/2$ (we can always assume that $\Ccal$ takes this form, up to composition by a similarity). For $\varepsilon>0$, we let $\D_\varepsilon$ denote the ball  
    $$\D_\varepsilon=B_{p,q}(0,\sqrt{2}+\varepsilon)=\left\{v_p+v_q\in \R^{p,q}\,\vert\,\sqrt{-\b(v_p,v_p)}+\sqrt{\b(v_q,v_q)}< \sqrt{2}+\varepsilon\right\}.$$
    By construction, one has $\Ccal\subset\D_\varepsilon$ for all $\varepsilon>0$. Let $v=-u_p+u_q$ and for $\lambda\geq1$, define $\varphi_\lambda$ to be the similarity of $\R^{p,q}$ given by 
    $$\varphi_\lambda(u)=u,\,\varphi_\lambda(v)=v/\lambda^2\text{ and }\varphi_\lambda=\lambda^{-1}Id \text{ on }u^\perp\cap v^\perp. $$
    Let us show that the family $\{\varphi_\lambda(\D_\varepsilon)\}_{\varepsilon,\lambda}$ is a base of neighborhoods of $\mathcal{C}$. Since $\varphi_\lambda$ acts trivially on the line generated by  $u$, every $\varphi_\lambda(\D_\varepsilon)$ contains $\mathcal{C}$. Let $V$ be a neighborhood of $\Ccal$ and let $\varepsilon>0$ be such that $\{ tu\,\vert\,\vert t\vert\leq 1+\varepsilon/\sqrt2\}\subset V$. Let $\pi:\R^{p,q}\to \R u$ be the linear projection on $u$ in the direction of $v^\perp$. 
    Now $\varphi_\lambda\to \pi$ uniformly on compact subsets, so $$\varphi_\lambda(\overline{\D_\varepsilon})\to \pi(\overline{\D_\varepsilon})=\{ tu\,\vert\,\vert t\vert\leq 1+\varepsilon/\sqrt2\}\text{ as }\lambda\to +\infty,$$  in the Hausdorff topology. Hence, for $\lambda>1$ sufficiently large, one has $\varphi_\lambda(\overline{\D_\varepsilon})\subset V$.
\end{proof}

\begin{lem}
    \label{lemme_inj_geod}
    Let $M$ be a simply connected conformally flat manifold. Then the developing map is injective when restricted to the image of any $\gamma\in \PPL(I,M)$.
\end{lem}

\begin{proof}
    The statement is immediate when $\gamma\in\PPL(I,M)$ is constant. If $\gamma$ is non-constant,  the curve $\alpha=\dev\circ\gamma\in\PPL(I,\Ein^{p,q})$ is a non-constant projectively parametrized lightlike geodesic. The curve $\alpha$ lifts to a $\widetilde\alpha\in\PPL(I,\S^p\times\S^q)$. The lift $\widetilde\alpha$ does not contain antipodal points, see Example \ref{exemple_para_proj}, so $\alpha$ is injective. Therefore $\dev$ is injective when restricted to $\gamma(I)$.
\end{proof}

From the preceding two lemmas, we deduce the following:

\begin{cor}
    \label{cor_voisinage_diamant}
    Let $M$ be a conformally flat manifold, and let $\gamma\in\PPL(I,M)$ such that $\smash{\overline{\gamma(I)}}$ is compact. Then there exist $\alpha\in \PPL(I,\D^{p,q})$ and $\varphi\in \Conf(\D^{p,q},M)$ such that $\gamma=\varphi(\alpha)$.
\end{cor}

\begin{proof}
    We can always replace $M$ with its universal cover and assume that $M$ is simply connected. Let $\dev:M\to \Ein^{p,q}$ be a developing map. Since $\smash{\overline{\gamma(I)}}$ is compact, one can always find $\hat\gamma\in\PPL(I,M)$ such that $\smash{\overline{\gamma(I)}}\subset \hat{\gamma}(I)$. In particular, we deduce from  Lemma~\ref{lemme_inj_geod} that the developing map is injective in restriction to $\smash{\overline{\gamma(I)}}$. Therefore we can find a neighborhood $U$ of $\smash{\overline{\gamma(I)}}$ in $M$ such that the developing map is still injective when restricted to $U$. Let $\beta=\dev(\gamma)\in\PPL(I,\Ein^{p,q})$ and let $V=\dev(U)$. The open set $V\subset \Ein^{p,q}$ is a neighborhood of $\Ccal=\smash{\overline{\alpha(I)}}$ in $\Ein^{p,q}$. From Lemma~\ref{lemme_base_voisi_photon}, we can find a conformal embedding $\psi:\D^{p,q}\to V$ containing $\Ccal$ in its image. Since $\psi$ is injective, there exists a unique $\alpha\in \PPL(I,\D^{p,q})$ such that $\psi(\alpha)=\beta$.  Since $\dev\vert_{U}$ is injective, there is a unique conformal map $\varphi:\D^{p,q}\to M$ such that $\dev\circ\varphi=\psi$. Now $\dev(\varphi(\alpha))=\beta=\dev(\gamma)$, hence $\varphi(\alpha)=\gamma$ by injectivity on $\dev\vert_U$.
\end{proof}

\begin{proof}[Proof of Theorem~\ref{thm_hyp_conf_plat}]
    Implication $(1)\Rightarrow (2)$ follows from Proposition~\ref{proposition_naturalité}. Indeed, every conformal map $\varphi\in \Conf(\D^{p,q},M)$ induces a 1-Lipschitz mapping
    $\varphi:(\D^{p,q},\delta_{\D^{p,q}})\to (M,\delta_M)$, and $\delta_M$ is locally equivalent to a Riemannian distance by Theorem \ref{thm_hyp}.

    For the converse implication, assume that $M$ is not Markowitz hyperbolic. From Theorem~\ref{thm_hyp}, there exists a sequence $(\gamma_k)\in\PPL(I,M)$ such that $\gamma_k(0)$ converges in $M$ and $\gamma_k^\prime(0)\to \infty$. 
    We can always replace $\gamma_k$ by $t\mapsto\gamma_k(t/2)$ for $t\in I$ and assume that $\gamma_k(I)$ has compact closure in $M$, for all $k\geq 0$. From Corollary~\ref{cor_voisinage_diamant},  we can find a sequence of conformal maps $(\varphi_k)\in\Conf(\D^{p,q},M)$ and a sequence $\alpha_k\in\PPL(I,\D^{p,q})$ such that $\gamma_k=\varphi_k(\alpha_k)$. Also, since $\D^{p,q}$ is homogeneous, we can always assume that after a suitable reparameterization, one has $\alpha_k(0)=x$ for a fixed basepoint $x\in \D^{p,q}$. We claim that $(\varphi_k)$ is not equicontinuous at $x$. Indeed, since $\alpha_k\in \PPL(I,\D^{p,q})$, one has 
    $$F_{\D^{p,q}}(\alpha_k^\prime(0))\leq \sqrt{\rho_I(\partial_t,\partial_t)}=2.$$ 
    Since $\D^{p,q}$ is Markowitz hyperbolic by Proposition~\ref{prop_distance_diamant}, the sequence $(\alpha_k^\prime(0))$ is bounded by Theorem~\ref{thm_hyp}.
    Now 
    $T_x\varphi_k(\alpha_k^\prime(0))=\gamma_k^\prime(0)\to \infty$ as $k\to \infty$, so the operator norm of $T_x\varphi_k$ is unbounded. If $(\varphi_k)$ were equicontinuous at $x$, then a subsequence of $(\varphi_k)$ would converge uniformly in a neighborhood of $x$. But then the convergence would be smooth by \cite[Thm. 1.2]{Charles_transformations_conformes}, contradicting the fact that $T_x\varphi_k$ diverges. Hence $(\varphi_k)$ is non-equicontinuous.    
\end{proof}

\subsection{Conformally convex domains}\label{section_conf_convex}

A \emph{Möbius hyperplane} is a subset $\Hcal\subset \Ein^{p,q}$ that is the intersection of $\Ein^{p,q}$ with a projective hyperplane, that is $\Hcal=\P(H)\cap\Ein^{p,q}$ for some hyperplane $H\subset \R^{p+1,q+1}$. Up to conformal equivalence, there are three types of Möbius hyperplanes depending on the signature of $H$. A Möbius hyperplane is either a lightcone or it is conformally equivalent to $\Ein^{p-1,q}$ or $\Ein^{p,q-1}$. 

\begin{figure}
    \centering
    \begin{tikzpicture}[scale=2]

  \def\a{1}
  \def\h{3}


    \fill[bovert] ({cosh(\a)}, {sinh(\a)}) -- ({-cosh(\a)}, {sinh(\a)}) -- (-1,0) -- (1,0) -- ({cosh(\a)}, {sinh(\a)});

    \fill[bovert] ({cosh(\a)}, {sinh(\a)}) arc (0:180:{cosh(\a)} and 0.2);

    \fill[bovert] (-1,0) arc (180:360:1 and 0.15);

    \fill[white, opacity=0.3] ({sinh(\a)}, {sinh(\a)}) -- (0,0) -- ({-sinh(\a)}, {sinh(\a)}) ;

     \fill[white] ({sinh(\a)}, {sinh(\a)}) arc (0:360:{sinh(\a)} and 0.15);

    \fill[bovert] ({sinh(\a)}, {sinh(\a)}) arc (0:360:{sinh(\a)} and 0.15);

    \fill[white, opacity=0.5] ({sinh(\a)}, {sinh(\a)}) arc (0:360:{sinh(\a)} and 0.15);

    \fill[domain=0:\a,smooth,variable=\u,white] 
      plot ({cosh(\u)}, {sinh(\u)});

    \fill[domain=0:1,smooth,variable=\u,white] 
      plot ({-cosh(\u)}, {sinh(\u)});

  \draw[domain=0:\a,smooth,variable=\u,thin] 
      plot ({cosh(\u)}, {sinh(\u)});

    \draw[domain=0:1,smooth,variable=\u,thin] 
      plot ({-cosh(\u)}, {sinh(\u)});

    \draw (-1,0) arc (180:360:1 and 0.15);
    \draw[densely dashed] (1,0) arc (0:180:1 and 0.15);


    \draw ({-cosh(\a)}, {sinh(\a)}) arc (180:360:{cosh(\a)} and 0.2);
    \draw ({cosh(\a)}, {sinh(\a)}) arc (0:180:{cosh(\a)} and 0.2);

    \draw[densely dashed] ({sinh(\a)}, {sinh(\a)}) -- (0,0) -- ({-sinh(\a)}, {sinh(\a)});

    \draw ({sinh(\a)}, {sinh(\a)}) arc (0:360:{sinh(\a)} and 0.15);

\end{tikzpicture}
    \caption{A conformally convex domain of $\R^{1,2}$.}
    \label{figure_conf_conv_domain}
\end{figure}
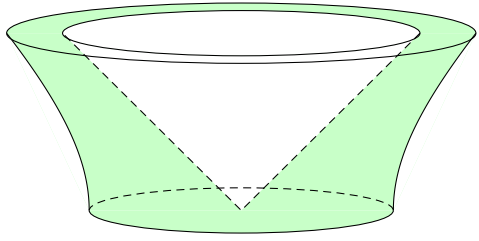

\begin{definition}
    \label{definition_conf_conv}
    A subset $\Omega\subset\Ein^{p,q}$ is said to be \emph{conformally convex} if for every point $x\in \partial\Omega$, there exists a Möbius hyperplane $\Hcal$, such that $x\in \Hcal$ and $\Omega\cap\Hcal=\emptyset$. In that case, we say that $\Hcal$ is a \emph{supporting Möbius hyperplane} of $\Omega$ at $x$.
\end{definition}

Conformal convexity is the conformal analog of convexity in real projective space, where projective hyperplanes are replaced by Möbius hyperplanes.

\begin{examples}
    \label{exemple_de_conf_conv}
    Let us mention the following examples of conformally convex domains.
    \begin{itemize}
        \item Any convex domain of $\R^{p,q}$ can be identified with a conformally convex domain under a stereographic projection. Indeed, given a stereographic projection $\st:\R^{p,q}\to \Ein^{p,q}$ and a supporting affine hyperplane $H$ of a convex domain $\Omega\subset\R^{p,q}$, the closure $\smash{\overline{\st(H)}}$ is a supporting Möbius hyperplane of $\st(\Omega)$, see Lemma~\ref{lemme_geometry_mobius_support_hyperplane}.
        \item Any convex domain of $\H^{p,q}$ can be identified with a conformally convex domain of $\Ein^{p,q}$. Recall that $\H^{p,q}$ refers to the hyperboloid $\H^{p,q}=\{v\in\R^{p+1,q}\,\vert\,\b(v,v)=-1\}$ endowed with the metric $\b$, which restricts to a metric of signature $(p,q)$ with constant negative curvature. A domain $\Omega\subset \H^{p,q}$ is \emph{convex} if any $x\in \partial\Omega$ admits a totally geodesic supporting hypersurface. Then the image of a convex $\Omega\subset\H^{p,q}$ under the natural conformal injection $j:\H^{p,q}\hookrightarrow\Ein^{p,q}$ given by $j(v)= [v+e_{p+q+2}]$, is conformally convex in $\Ein^{p,q}$.
        \item Similarly, any convex domain of $\dS^{p,q}=\{v\in\R^{p,q+1}\,\vert\,\b(v,v)=+1\}$ can be identified with a conformally convex domain of $\Ein^{p,q}$. 
        \item Any dually convex domain of $\Ein^{p,q}$ is conformally convex. Recall that a domain $\Omega$ is dually convex if every point in its boundary admits a supporting lightcone, see also \cite{Zimpropqh} for a general definition for flag manifolds. 
        \item Any open intersection of conformally convex domains is again conformally convex.
        \item The domain $\Omega=\{(t,x,y)\in \R^{1,2}\,\vert\,0<t<1 \text{ and }0<-t^2+x^2+y^2<1\}$ depicted in Figure~\ref{figure_conf_conv_domain} is conformally convex, but neither convex nor dually convex.
    \end{itemize}
\end{examples}

Our main result for conformally convex domains is the following Barth-type theorem.

\begin{thm}
    \label{thm_conformement_convex}
    Let $\Omega\subset\Ein^{p,q}$ be a conformally convex domain. Then $\Omega$ is Markowitz hyperbolic if and only if $\Omega$ is totally conformally lightlike incomplete. In that case, the metric space $(\Omega,\delta_\Omega)$ is complete and every maximal $\gamma\in \PPL(I,\Omega)$ is an isometry $\gamma:(I,d_I)\to (\Omega,\delta_\Omega)$.
\end{thm}

The proof is given below in Section~\ref{section_preuve_thm_conf_conv}.

\begin{ex}[Markowitz distance of a ball]
    Let $\R^{1,1}$ be equipped with the flat metric $g=dxdy$ and let $D=\{(x,y)\in\R^{1,1}\,\vert\,x^2+y^2<1\}$ denote the unit ball of the Minkowski space. Since $D$ is convex, it identifies with a conformally convex domain of the Einstein universe, see Example~\ref{exemple_de_conf_conv}, so the conclusion of Theorem~\ref{thm_conformement_convex} holds for $D$. Note that, in this particular case, the completeness of $\delta_D$ also follows from the inequality $H_D\leq \delta_D$, where $H_D$ is the Hilbert metric on $D$. The fact that every maximal $\gamma\in \PPL(I,D)$ is an isometry $\gamma:(I,d_I)\to (\Omega,\delta_D)$ reduces to the equality $H_D = \delta_D$ along lightlike geodesics.
    
    Let us explain why, surprisingly, the distance $\delta_D$ and $H_D$ are not equivalent, and in fact, why there are no $\alpha,\beta >0$, such that $\delta_D\leq \alpha H_D+\beta$. Assume by contradiction that such constants do exist, so that the identity map $\text{id}:D\to D$ induces a quasi-isometry between 
    $(D,\delta_D)$ and $(D,H_D)\simeq\H^2$, which is Gromov hyperbolic. Hence, there exists a constant $r>0$ such that any $\delta_D$ geodesic between points $x,y\in D$ is contained in the $r$-neighborhood of a hyperbolic geodesic $[x,y]$ of $D$, see \cite{bridson2013metric}. Therefore, for any $x,y\in D$, there is a piecewise lightlike geodesic curve $\gamma$ that is contained in the $(r+1)$-neighborhood of $[x,y]$ and whose length is $\leq \delta_D(x,y)+1$. Choosing the segment $[x,y]$ sufficiently close to the boundary yields a contradiction, see Figure~\ref{figure_mark_dist_ball}. See also \cite{article_hyp_de_Gromov}.
\end{ex} 

    \begin{figure}
        \centering
        \begin{tabular}{ccccc}
        \begin{tikzpicture}[scale=2,rotate = 45]
            \draw[fill=none](0,0) circle (1);
            \foreach\t in {0.2,0.4,0.6,0.8}{
            
            \draw [red,samples=200,thick,variable=\x,domain={-sqrt(1-\t*\t)}:{sqrt(1-\t*\t)}] plot ({\x}, {\t});

            \draw [red,samples=200,thick,variable=\x,domain={-sqrt(1-\t*\t)}:{sqrt(1-\t*\t)}] plot ({\x}, {-\t});

            \draw [red,samples=200,thick,variable=\x,domain={-sqrt(1-\t*\t)}:{sqrt(1-\t*\t)}] plot ({\t}, {\x});

            \draw [red,samples=200,thick,variable=\x,domain={-sqrt(1-\t*\t)}:{sqrt(1-\t*\t)}] plot ({-\t}, {\x});
        
            }

            \draw[thick, red] (0,1) -- (0,-1);
            \draw[thick, red] (1,0) -- (-1,0);
            
        \end{tikzpicture}
        &&&&
        \begin{tikzpicture}[scale=2,rotate = 45]
        \def\eps{0.15}
            \draw[fill=none](0,0) circle (1);
            \foreach\t in {0.01,0.04,0.1,0.2,0.5,1,3}{
            
            \draw [red,samples=200,thick,variable=\x,domain=180-asin(1/(1+\t)):180+asin(1/(1+\t))] plot ({sqrt(2*\t+\t*\t)*cos(\x)+1+\t}, {sqrt(2*\t+\t*\t)*sin(\x)});
            
            \draw [red,samples=200,thick,variable=\y,domain=180-asin(1/(1+\t)):180+asin(1/(1+\t))] plot ({-sqrt(2*\t+\t*\t)*cos(\y)-1-\t}, {sqrt(2*\t+\t*\t)*sin(\y)});
            
            \draw [red,samples=200,thick,variable=\x,domain=180-asin(1/(1+\t)):180+asin(1/(1+\t))] plot ({sqrt(2*\t+\t*\t)*sin(\x)}, {sqrt(2*\t+\t*\t)*cos(\x)+1+\t});
            
            \draw [red,samples=200,thick,variable=\y,domain=180-asin(1/(1+\t)):180+asin(1/(1+\t))] plot ({sqrt(2*\t+\t*\t)*sin(\y)}, {-sqrt(2*\t+\t*\t)*cos(\y)-1-\t});
            }

            \draw[very thick, red] (0,1) -- (0,-1);
            \draw[very thick, red] (1,0) -- (-1,0);

            \draw[fill=none] (0.55-\eps,0.55) -- (0.55,0.55+\eps) -- (0.55+\eps,0.55) -- (0.55,0.55-\eps) -- (0.55-\eps,0.55);
            
            \draw[->, >=Stealth] (0.65,0.65) to [out=90, in=-30] (0.1,1.3);

            \begin{scope}[yshift=43,xshift=-12,rotate=45,scale=0.35]
                \draw (1,1) -- (1,-1) -- (-1,-1) -- (-1,1) -- (1,1);
                \foreach\t in {-0.8,-0.4,0,0.4,0.8}{
                \draw[red,thick] (-1,-0.1+\t) -- (1,0.1+\t);
                \draw[red,thick] (-1,0.1+0.1*\t*\t+\t) -- (1,-0.1+0.1*\t*\t+\t);
                }
            \end{scope}
            
        \end{tikzpicture}\\
        Affine model in $\R^{1,1}$ 
        &&&&
        Conformal model in the Poincaré disk
        \end{tabular}
        \caption{Markowitz distance of a ball $B\subset\R^{1,1}$.}
        \label{figure_mark_dist_ball}
    \end{figure}
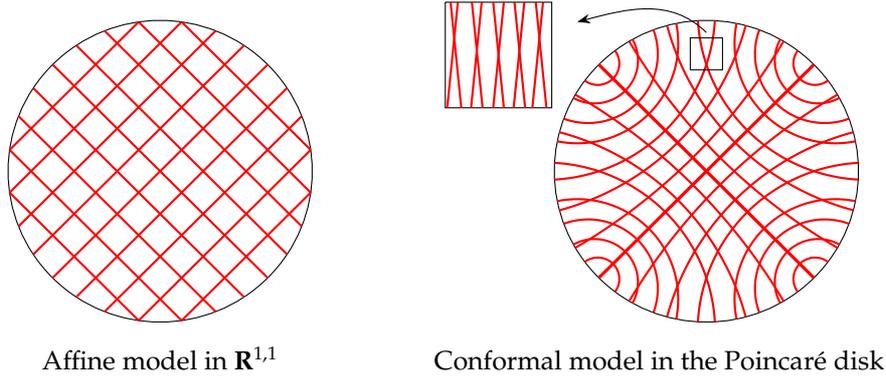

\subsection{Projective maps}

In what follows, a \emph{projective manifold} is a manifold $P$ endowed with a projective class $[\nabla]$ of affine connections. These manifolds are equipped with a Kobayashi--Royden-like pseudodistance, denoted $\delta^\kob$, defined as in Equation~(\ref{equation_def_d_mark}), but where lightlike geodesics are replaced by \emph{projective geodesics}, see \cite{kob_projectif}. Recall that in a projective manifold $P$, a curve $\alpha:I\to P$ is a projective geodesic if for some affine connection $\nabla$ in the projective class of $P$, there exists a $\nabla$-geodesic $\xi(t)$ and a parameter $u(t)$, called \emph{projective}, such that $\alpha=\xi\circ u^{-1}$, see for instance \cite[Sect. 6]{kob_projectif}. 

\begin{prop}[{\cite[Ex. 3.17]{kob_projectif}}]
    \label{Proposition_complletude_delta_kob}
    Let $\Omega\subset\RP^n$ be a connected domain. Then $\delta_\Omega^\kob$ is complete if and only if $\Omega$ is properly convex. In that case, $\delta_\Omega^\kob$ is the Hilbert metric on $\Omega$.
\end{prop}

We now define an essential tool for the proof of Theorem~\ref{thm_conformement_convex}. 

\begin{definition}
    \label{def_app_proj}
    Let $(M,[g])$ be a conformal manifold and $P$ be a projective manifold. A map $f:M\to P$ is \emph{projective} if for every $\gamma\in \PPL(I,M)$, the curve $f\circ\gamma:I\to P$ is a projective geodesic.
\end{definition}

\begin{ex}
    \label{exemple_app_projective}
    Let $\Omega\subset\Ein^{p,q}$ be contained in some open subset $U\subset \P(\R^{p+1,q+1})$. Then for every $V\subset \RP^m$, a map $f:U\to V$ that is projective in the classical sense restricts to a projective map $f\vert_\Omega:\Omega\to V$ in the sense of Definition \ref{def_app_proj}.
\end{ex}

\begin{prop}
    \label{Proposition_applications_projectives}
    Let $f:M\to P$ be a projective map. Then, for all $x,y\in M$, one has
$$\delta_P^\kob(f(x),f(y))\leq \delta_M(x,y).$$
\end{prop}

\begin{proof}
    This follows from the definition, since the map $f$ sends a lightlike chain of $M$ to a chain of projective geodesics of $P$ of the same length.
\end{proof}

\subsection{Proof of Theorem~\ref{thm_conformement_convex}} \label{section_preuve_thm_conf_conv}
We will first give geometric lemmas on Möbius hyperplanes. 
\begin{lem}
\label{lemme_geometry_mobius_support_hyperplane}
    Let $\Hcal$ be a Möbius hyperplane and let $x\in \Ein^{p,q}$ such that $\Hcal\neq C(x)$. Then, the trace of $\Hcal$ in a stereographic projection of $x$ is an affine hyperplane of $\R^{p,q}$ when $x\in \Hcal$, and a hyperboloid or a lightcone of $\R^{p,q}$ when $x\not\in \Hcal$.
\end{lem}

\begin{proof}
    Let $H\subset\R^{p+1,q+1}$ be a linear hyperplane such that $\Hcal=\P(H)\cap \Ein^{p,q}$, and let $s\in\{0,1,-1\}$ be the sign of $H^\perp$. Let $v\in \R^{p+1,q+1}$ such that $x=[v]$. If $x\not\in \Hcal$, we can let $w\in \R^{p+1,q+1}$ be an isotropic vector so that $H=F+\R(w-sv)$ and $\b(v,w)> 0$, where $F=v^\perp\cap H$. 
    Let $\varphi:F\simeq\R^{p,q}\to \Ein^{p,q}\setminus C(x)$ be the stereographic projection given by $\varphi(h)=[w+h-\b(h,h)v]$, see Proposition~\ref{prop_projection_stereo}.
    For $h\in F$, one has 
    \begin{align*}
        \varphi(h)\in \Hcal\iff \b(h,h)= s.
    \end{align*}
    Hence $\Hcal$ is a hyperboloid or a lightcone in the stereographic projection $\varphi$. Since any other stereographic projection $\R^{p,q}\simeq \Ein^{p,q}\setminus C(x)$ is obtained from $\varphi$ by composition with a similarity, the result follows in this case. The case $x\in \Hcal$ is similar.
\end{proof}

\begin{lem}
\label{lemme_intersection_photon_hyperplane}
    Let $\Hcal$ be a Möbius hyperplane and let $\Delta$ be a photon. Then $\Delta$ is either contained in $\Hcal$ or it intersects $\Hcal$ transversely at a single point. 
\end{lem}

\begin{proof}
    Write $\Delta=\P(\Pi)$ and $\Hcal=\P(H)\cap\Ein^{p,q}$, where $\Pi\subset\R^{p+1,q+1}$ is a totally isotropic plane and $H\subset\R^{p+1,q+1}$ is a linear hyperplane. If $\Delta\not\subset\Hcal$, then $\Pi+H=\R^{p+1,q+1}$, and $\Delta\cap\Hcal=\{x\}$ for $x=\P(\Pi\cap H)$. Let $y\in \Hcal\setminus \Delta$. Then in the stereographic projection of $y$, the Möbius hyperplane $\Hcal$ is an affine hyperplane of $\R^{p,q}$ and $\Delta$ is an affine line of $\R^{p,q}$, see Lemma~\ref{lemme_geometry_mobius_support_hyperplane}. In particular, the intersection $\Delta\cap\Hcal$ is transverse.
\end{proof}

\begin{lem}
    \label{lemme_de_caracterisation}
    Let $\Omega\subset\Ein^{p,q}$ be a non Markowitz hyperbolic domain. Then there exist a compact $\Kcal\subset \Omega$ and a sequence of lightlike segments $(c_k)$ intersecting $\Kcal$ and contained in $\Omega$, such that $(c_k)$ converges to a photon in the Hausdorff topology.
\end{lem}

\begin{proof}
    Let $r:\S^p\times\S^q\to \Ein^{p,q}$ be the standard double cover given by Proposition ~\ref{prop_rev_double}, and write $\Omega^\prime=r^{-1}(\Omega)$ and $g=g_{\S^p}+g_{\S^q}$. Since the Markowitz functional is invariant under covering by Proposition \ref{prop_naturalité_F_M}, we can find a sequence of vectors $(v_k)$ converging to some nonzero $v\in C(T\Omega^\prime)$, such that $F_{\Omega^\prime}(v_k)\to 0$. For $k\geq 0$, let $x_k^+$ and $x_k^-$ be the future and past endpoints of the lightlike geodesic generated by $v_k$ in $\Omega^\prime$, respectively, and let $\sqrt{2}L_k$ be the $g$-length of that geodesic. From Proposition~\ref{Prop_exemple_calcul_F_M}, one has 
    $$F_{\Omega^\prime}(v_k)=\left(\frac{1+\tan(d_g(x_k,x_k^-)-L_k/2)^2}{\tan(L_k/2)+\tan(d_g(x_k,x_k^-)-L_k/2)}
    +\frac{1+\tan(d_g(x_k,x_k^+)-L_k/2)^2}{\tan(L_k/2)+\tan(d_g(x_k,x_k^+)-L_k/2)}\right)\frac{\Vert v_k\Vert_g}{\sqrt{2}}\to 0,$$
    where $x_k$ is the basepoint of $v_k$, for all $k\geq 0$.
    Since the sequence $(x_k)$ is contained in a compact subset of $\Omega^\prime$, there exists some $\varepsilon>0$ such that $d_g(x_k,x_k^\pm)>\varepsilon$ for all $k\geq 0$. 
    We deduce that $\tan(L_k/2)\to +\infty$, so $L_k\to \pi$. Let $(\hat c_k)$ be a sequence of lightlike segments so that $c_k\subset(x_k^-,x_k^+)$ for all $k\geq 0$, and $L_g(\hat c_k)\to \sqrt 2\pi$ as $k\to \infty$. Then $(\hat c_k)$ converges to a geodesic segment $\hat\Delta$ of $\S^p\times \S^q$  of length $\sqrt 2\pi$, so $(c_k)=(r(\hat c_k))$ converges to the photon $\Delta=c(\hat\Delta)$ in the Hausdorff topology.
\end{proof}

\begin{proof}[Proof of Theorem~\ref{thm_conformement_convex}]
If $\Omega$ is Markowitz hyperbolic, then it is totally conformally lightlike incomplete, see Proposition~\ref{prop_TCLI}. Assume that $\Omega$ is not Markowitz hyperbolic and let us show that there exists a non-constant projectively parametrized lightlike geodesic $\gamma:\R\to \Omega$. By Lemma~\ref{lemme_de_caracterisation}, there exist a compact $\Kcal\subset \Omega$ and a sequence $([a_k,b_k])$ of lightlike segments contained in $\Omega$ that intersect $\Kcal$, such that $[a_k,b_k]$ converges to a photon in the Hausdorff topology. For $k\geq 0$, let $\Delta_k$ be the unique photon containing $[a_k,b_k]$. Then $\Delta_k\to \Delta$ in the Hausdorff topology. Let us fix a sequence $(x_k)$ so that $x_k\in \Delta_k\setminus [a_k,b_k]$, and $x_k\to x\in \Delta$. We claim that $\Delta\setminus\{x\}$ is contained in $\Omega$. Assume by contradiction that $\Delta\setminus\{x\}\not\subset\Omega$. Then there exists a point $y\in \partial\Omega\cap \Delta$ distinct from $x$. Since $\Omega$ is conformally convex, we can find a Möbius hyperplane $\Hcal$ disjoint from $\Omega$ that contains $y$. Since $\Kcal\cap \Delta\neq \emptyset$, the photon $\Delta$ cannot be contained in $\Hcal$, so $\Delta$ intersects $\Hcal$ transversely at $y$, see Lemma~\ref{lemme_intersection_photon_hyperplane}. Let $c=[y-\varepsilon,y+\varepsilon]\subset\Delta$ be a small enough lightlike segment around $y$. Since $[a_k,b_k]\to \Delta$, we can find a sequence of lightlike segments $c_k\subset [a_k,b_k]$ that converge to $c$ in the Hausdorff topology. In particular $c_k$ intersects $\Hcal$ for large enough $k$. This contradicts the fact that $\Hcal\cap[a_k,b_k]=\emptyset$ for all $k\geq 0$. Therefore $\Delta\setminus\{x\}\subset \Omega$. Since $\Delta\setminus\{x\}$ can be projectively parametrized on $\R$, we deduce that $\Omega$ is not totally conformally lightlike incomplete.

We now assume that $\Omega$ is Markowitz hyperbolic. Let us show that $(\Omega,\delta_\Omega)$ is a complete metric space. Let $(x_k)$ be a Cauchy sequence, and assume by contradiction that $(x_k)$ does not converge in $\Omega$. Up to extraction, we can assume that $(x_k)$ converges to a point $x\in \partial\Omega$. We distinguish two cases.

\smallskip

\emph{$\bullet$ Case 1: There exists a Möbius hyperplane $\Hcal_1$ disjoint from $\Omega$ that does not contain $x$.} Let $\Hcal_2$ be a supporting Möbius hyperplane at $x$ and let $f:\Omega\to \P(V)$ be the quotient projective map, where $V=(H_1\oplus H_2)/H_1\cap H_2$ and $H_i\subset\R^{p+1,q+1}$ is the linear hyperplane defining $\Hcal_i$. By connectedness, the image of $f$ is one of the two connected components of $\P(V)\setminus \{\P(H_1/H_1\cap H_2),\P(H_1/H_1\cap H_2)\}$. We identify this component projectively with the interval $I$. Since $f$ is a projective mapping, see Example \ref{exemple_app_projective}, Proposition~\ref{Proposition_applications_projectives} implies that 
$$d_I(f(x_k),f(x_{k+l}))\leq \delta_\Omega(x_k,x_{k+l}),$$
so $(f(x_k))$ is a Cauchy sequence of $(I,d_I)$. Since $(I,d_I)$ is complete, the sequence $(f(x_k))$ converges in $I$. But $x_k\to x\in \Hcal_2$, so $f(x_k)\to\P(H_2/H_1\cap H_2)\in\partial I$, a contradiction.

\smallskip

\emph{$\bullet$ Case 2: Every supporting Möbius hyperplane of $\Omega$ contains $x$.} We claim that $C(x)$ is disjoint from $\Omega$. Assume by contradiction that $\Omega$ intersects $C(x)$. Let $\Delta$ be a photon containing $x$ and intersecting $\Omega$. Since $\Omega$ is totally conformally lightlike incomplete, the lightlike line $\Delta\setminus\{x\}$ cannot be fully contained in $\Omega$, so we can find a point $y\in \Delta\cap\partial\Omega$. Now $\Omega$ is conformally convex, so it admits a supporting Möbius hyperplane $\Hcal$ at $y$. By assumption we have $x\in \Hcal$, so $\Delta\subset\Hcal$ by Lemma~\ref{lemme_intersection_photon_hyperplane}. This contradicts $\Omega\cap\Delta\neq\emptyset$. Therefore $\Omega\cap C(x)=\emptyset$. 

Let us fix a stereographic projection $\R^{p,q}\simeq \Ein^{p,q}\setminus C(x)$, so that $\Omega$ is identified with a subset of $\R^{p,q}$. Let $\Hcal$ be a supporting Möbius hyperplane of $\Omega$. Since $\Hcal$ contains $x$, its trace in $\R^{p,q}$ is an affine hyperplane of $\R^{p,q}$ by Lemma~\ref{lemme_geometry_mobius_support_hyperplane}. In particular, the domain $\Omega$ is a connected component of the complement of affine hyperplanes in $\R^{p,q}$, so $\Omega$ is a convex subset of $\R^{p,q}$. We can find a splitting $\R^{p,q}=E\oplus F$ such that $\Omega= E\oplus C$ for some convex $C\subset F$ containing no affine line. Since $\Omega$ does not contain a lightlike line, the subspace $E$ is definite. In particular, we can always assume that the splitting $\R^{p,q}=E\oplus F$ is orthogonal. Let $\pi:\R^{p,q}\to F$ be the orthogonal projection onto $F$. Then $\pi$ sends an affine lightlike line in $\R^{p,q}$ to an affine line in $F$, hence it is a projective mapping in the sense of Definition \ref{def_app_proj}. Therefore 
the restriction $\pi_\Omega=\pi\vert_\Omega:\Omega\to C$ is projective. By Proposition~\ref{Proposition_applications_projectives}, one has 
$$\delta_C^\kob(\pi_\Omega(x_k),\pi_\Omega(x_{k+l}))\leq \delta_\Omega(x_k,x_{k+l}),$$
for all $k,l\geq 0$. In particular the sequence $(\pi_\Omega(x_k))$ is a Cauchy sequence with respect to $\delta_C^\kob$. Now $C$ is convex and contains no line, so $(C,\delta_C)$ is a complete metric space by Proposition~\ref{Proposition_complletude_delta_kob}. Hence $(\pi_\Omega(x_k))$ converges to some limit $\ell\in C$. Let $\hat{\ell}\in\pi_\Omega^{-1}(\{\ell\})$ and $\varepsilon>0$ such that $B(\hat{\ell},\varepsilon)$ has compact closure in $\Omega$. For $k\geq 0$, let $y_k\in\R^{p,q}$ be the intersection of $x_k+E$ with $F$. Then $y_k\to \hat{\ell}$ as $k\to \infty$, so up to extraction, we may assume that $y_k\in B(\hat{\ell},\varepsilon/2)$ for all $k\geq 0$. Up to extracting a subsequence, we can also assume that $\delta(x_k,x_l)\leq \varepsilon/2$ for all $k,l\geq 0$. 
Now, the distance $\delta_\Omega$ is invariant by translations in the direction $E$, since these translations are conformal maps that preserve $\Omega$. Hence, if we let $u=x_0-y_0\in E$, we obtain
\begin{align*}
    \delta_\Omega(x_k-u,\hat{\ell})&\leq\delta_\Omega(x_k-u,x_0-u)+\delta_\Omega(x_0-u,\ell)\\
    &=\delta_\Omega(x_k,x_0)+\delta_\Omega(y_0,\hat{\ell})\\
    &\leq \varepsilon/2+\varepsilon/2=\epsilon.
\end{align*}
Since $B(\hat{\ell},\varepsilon)$ has compact closure, a subsequence of $(x_k-u)$ converges in $\Omega$, so $(x_k)$ has a convergent subsequence. Hence $(x_k)$ converges and $(\Omega,\delta_\Omega)$ is complete.

Finally, we show that a maximal $\gamma\in\PPL(I,M)$ is isometric. By definition we have 
$$\delta_\Omega(\gamma(s),\gamma(t))\leq d_I(s,t).$$
We claim that equality holds in the previous inequality. Let $a_1,a_2\in\partial\Omega$ be the endpoints of $\gamma$, and let $\Hcal_1$ and $\Hcal_2$ be supporting Möbius hyperplanes of $\Omega$ at $a_1$ and $a_2$, respectively. Let $f:\Omega\to \P((H_1\oplus H_2)/H_1\cap H_2)$ be the same quotient map as in \emph{Case 1} above. As before, identify $f(\Omega)$ projectively with $I$. Then $f(\gamma(I))=f(\Omega)=I$, so 
$$\delta_\Omega(\gamma(s),\gamma(t))\geq d_I(f(\gamma(t)),f(\gamma(s)))=d_I(s,t).$$
Hence $\delta_\Omega(\gamma(s),\gamma(t))= d_I(s,t)$ for all $s,t\in I$. 
\end{proof}

\begin{rmk}
    To prove Theorem~\ref{thm_conformement_convex}, one may be tempted to argue as follows: a conformally convex domain $\Omega\subset\Ein^{p,q}\subset\P(\R^{p+1,q+1})$ is contained as a closed subset of its convex envelope $\Conv(\Omega)$ in real projective space. By construction one has $\delta^\kob_{\Conv(\Omega)}\leq \delta_\Omega$ on $\Omega$. If $\Conv(\Omega)$ contains no line, then $\delta^\kob_{\Conv(\Omega)}$ is complete by Proposition \ref{Proposition_complletude_delta_kob}, since $\Conv(\Omega)$ is convex by construction, so $\delta_\Omega$ is complete. However, $\Conv(\Omega)$ may contain a projective line even if $\Omega$ doesn't. This is the case, for instance, for the Einstein-de Sitter half-space.
\end{rmk}

\subsection{Complete non-conformally convex domains}
We end this section with examples of spacetimes that show that Proposition \ref{Proposition_complletude_delta_kob} has no simple analog with conformally convex domains, as there exist domains $\Omega\subset \Ein^{1,n}$ having a complete Markowitz distance that are not conformally convex.

\begin{prop}
    \label{prop_domain_non_conf_conv_but_complete}
    Let $\Omega\subset \R^{1,n}$ denote the union of the future of two transverse spacelike hyperplanes of $\R^{1,n}$. Then $(\Omega,\delta_\Omega)$ is complete Markowitz hyperbolic, but $\Omega$ is not conformally convex.
\end{prop}

\begin{proof}
    Let $v_1,v_2\in \R^{1,n}$ be two distinct timelike vectors, and let $H_1$ and $H_2$ be their orthogonal, respectively. For $i\in \{1,2\}$, we let $U_i=I^+(H_i)$ and we let $\Omega=U_1\cup U_2$. The domain $\Omega$ is not conformally convex, as points in $H_1\cap H_2$ do not admit any supporting Möbius hyperplanes. We claim that $\delta_\Omega$ is a complete metric. For $i\in\{1,2\}$, we let $g_i$ denote the Wick rotation of $\b$ in the direction $\R v_i$ (defined as in the proof a Proposition~\ref{prop_equivalence_d_Omega_ell}).
    For $x\in \R^{1,n}$, let $t_i(x)$ denote the Lorentz distance from $x$ to $H_i$. From the formula given in Example~\ref{exemple_formula_F_M}, we have 
    $$F_{U_i}(v)=\frac{\sqrt{g_i(v,v)}}{\sqrt{2}\,t_i(x)},$$
    for every point $x\in U_i$ and lightlike vector $v\in T_xU_i$. Now, a maximal lightlike geodesic of $\Omega$ is either contained in $U_1$ or in $U_2$, hence the infinitesimal functional of $\Omega$ is given by
    $$F_{\Omega}(v)=\frac{1}{\sqrt{2}}\min\left(\frac{\sqrt{g_1(v,v)}}{t_1(x)},\frac{\sqrt{g_2(v,v)}}{t_2(x)}\right),$$
    for every point $x\in \Omega$ and lightlike vector $v\in T_x\Omega$. Let $\alpha\geq 1$ and let $g$ be an inner product on $\R^{1,n}$ such that $g\leq g_i\leq \alpha g$ for $i=1,2$. Then, for $x\in \Omega$ and $v$ lightlike, one has 
    $$F_{\Omega}(v)\geq\frac{\sqrt{g(v,v)}}{\sqrt{2}\max\{t_1(x),t_2(x)\}}\geq \frac{\sqrt{g(v,v)}}{\sqrt{2} \alpha d_g(x,\partial\Omega)},$$
    where $d_g$ is the distance induced by $g$ on $\Omega$. Now, the metric $g_{\text{q.h.}}=\frac{g}{ d_g(\cdot,\partial\Omega)^2}$ is the \emph{quasihyperbolic metric} of $(\Omega,g)$ and it has a complete Riemannian distance $d_{\text{q.h.}}$, see for instance \cite[Prop. 2.8]{Uniformizing_gromov_hyp_space}. By Theorem~\ref{thm_lien_dmark_Fmark}, one obtains $\delta_\Omega\geq d_{\text{q.h.}}/\sqrt{2}\alpha$, so $\delta_\Omega$ is complete.
\end{proof}

\section{Globally hyperbolic C-maximal spacetimes}
\label{section_GHM}

\subsection{Cauchy maximality} A conformal spacetime $M$ is said to be \emph{globally hyperbolic} if it is \emph{causal}, meaning that it does not contain non-constant closed causal curves, and if its causal diamonds $J(x,y)$ are compact for every $x\leq y$.
Globally hyperbolic spacetimes are characterized by the fact that they admit \emph{Cauchy hypersurfaces}, that is, topological hypersurfaces that intersect every inextendible causal curve exactly once, see \cite{MinguzziSanchez2008}. Given two globally hyperbolic spacetimes $M$ and $N$, a \emph{Cauchy embedding} from $M$ to $N$ is a conformal map $f:M\to N$  that sends a Cauchy hypersurface of $M$ to a Cauchy hypersurface of $N$. 
\begin{definition}
    \label{definition_C_maximality}
    A globally hyperbolic spacetime $M$ is $C$-\emph{maximal} if every Cauchy embedding $f:M\to N$ is surjective.
\end{definition}
This notion was introduced by Rossi \cite{Clara} as a conformal analog of maximal spacetimes in Lorentzian geometry \cite{Choquet-Bruhat_Geroch}, and was later studied by Smaï, see \cite{SmaiThese,smai2023enveloping,smai2025futures}. Our goal is to prove the following theorem.

\begin{thm}
    \label{thm_GHM}
    Let $(M,[g])$ be a conformally flat, globally hyperbolic, $C$-maximal spacetime. 
    Then $M$ is Markowitz hyperbolic if and only if $M$ is totally conformally lightlike incomplete.
\end{thm}

We deduce the following corollary.

\begin{cor}
    \label{corollary_properness_GHMC}
    Let $M$ be a conformally flat, globally hyperbolic, $C$-maximal spacetime. If $M$ is totally conformally lightlike incomplete, then $\Conf(M)$ acts properly on $M$.\qed
\end{cor}

\begin{ex}
    \label{exemple_causalement_convexes_maximaux}
    Let $\Omega\subset \Eintilde$ be a causally convex domain. Then $\Omega$ is globally hyperbolic. If $\Omega$ is $C$-maximal, then there exists a closed subset $\Lambda\subset\Eintilde$ such that $\Omega=\cap_{x\in\Lambda}E(x)$, see \cite{smai2023enveloping}. Therefore the domain $\Omega$ can be identified with a conformally convex domain of $\Ein^{1,n}$ (in fact, one can identify $\Omega$ with a dually convex domain of $\Ein^{1,n}$, see \cite[Prop. 6.3]{Cha_Gal}). Hence, for causally convex $C$-maximal domains of $\Eintilde$, one can obtain the conclusion of Theorem~\ref{thm_GHM} with Theorem~\ref{thm_conformement_convex}. Also, we know that for this class of globally hyperbolic spacetimes, the distance $\delta_\Omega$ is complete and lightlike geodesics are geodesic by Theorem~\ref{thm_conformement_convex}. 
\end{ex}

Theorem~\ref{thm_GHM} is a direct consequence of Proposition \ref{prop_TCLI} and Proposition \ref{prop_cara_complete_hyperbolic} as well as the following proposition, whose proof is given below in Section \ref{section_preuve_prop_GHMC}.

\begin{prop}
    \label{prop_F_CO_GHM}
    Let $(M,[g])$ be a conformally flat, globally hyperbolic, $C$-maximal spacetime. Then $F_M$ is continuous.
\end{prop}

\begin{rmk}
    We do not know whether Proposition~\ref{prop_F_CO_GHM} can be strengthened to show that $\delta_M$ is a complete metric when $F_M$ is positive. Such a result would establish that the converse implication in Proposition~\ref{prop_cara_complete_hyperbolic}(2) holds in the context of GHMC spacetimes. Example~\ref{exemple_causalement_convexes_maximaux} shows that this improvement can indeed be obtained in many cases.
\end{rmk}

\subsection{Causal geometry of $\Eintilde$} When $n\geq 2$, the universal cover of $\Ein^{1,n}$ is conformally diffeomorphic to $\R\times\S^n$. We introduce some notations and give basic facts on the causal structure of $\Eintilde$, which we identify with $\R\times\S^n$, see also Figure~\ref{Figure_Cone_dans_Eintilde}:
\begin{itemize}
    \item the causal relations are given as follows: for all $(s,x),(t,y)\in\R\times\S^n$, one has
    $$(s,x)\leq (t,y)\Longleftrightarrow d(x,y)\leq  t-s,$$
    where $d$ is the round metric on $\S^n$;
    \item the automorphism group of the universal cover $\Eintilde\to\Ein^{1,n}$ is generated by the conformal diffeomorphism $\sigma$, given by $\sigma(t,x)= (t+\pi,-x)$;     
    \item the set of points of $\Eintilde$ that are not causally related to $p=(s,x)$ is denoted 
    $$E(p)=\{(t,y)\in\Eintilde\,\vert\,\vert t-s\vert<d(x,y)\}.$$
    This domain if conformally equivalent to $\R^{1,n}$ and any identification $E(p)\simeq\R^{1,n}$ is called a \emph{stereographic projection} at $p$; 
    \item the lightcone of a point $p=(s,x)\in \Eintilde$ is given by 
    $$C(p)=\{(t,y)\in\Eintilde\,\vert\,\vert t-s\vert=d(x,y)\mod 2\pi\Z\}.$$
    The complement of $C(p)$ in $\Eintilde$ has countably many components, given by $\{ E(\sigma^k(p))\}_{k\in\Z}$;
    \item we write $M^+(p)=E(\sigma(p))$ and $M^-(p)=E(\sigma^{-1}(p))$.
\end{itemize}

\begin{figure}
    \centering    

     \begin{tikzpicture}[scale=1.7,xscale=1.3]
            \begin{scope}[thick, lightblue,yshift=-14]
            \begin{scope}[rotate=-26.56]
            \draw[dashed] (1.12,0) arc (0:180:1.12 and 0.15);
          \draw (-1.12,0) arc (180:360:1.12 and 0.15);
        \end{scope}
        \end{scope}

        \begin{scope}[thick, lightblue,yshift=14]
            \begin{scope}[rotate=+26.56]
            \draw[thick,dashed] (1.12,0) arc (0:180:1.12 and 0.15);
          \draw (-1.12,0) arc (180:360:1.12 and 0.15);
        \end{scope}
        \end{scope}

        \begin{scope}[thick, lightblue]
            \draw plot[domain=0.7:1,smooth,variable=\ang]
   ({\ang}, {0.5*cos(90*\ang)+1});
            \draw plot[domain=0.7:1,smooth,variable=\ang]
   ({\ang}, {-0.5*cos(90*\ang)-1});
        \end{scope}

        \begin{scope}[thick, lightblue]
            \draw[densely dashed] plot[domain=0.85:1,smooth, variable=\ang]
   ({\ang}, {0.8*cos(90*\ang)+1});
            \draw[densely dashed] plot[domain=0.85:1,smooth, variable=\ang]
   ({\ang}, {-0.8*cos(90*\ang)-1});
        \end{scope}

        \draw[dashed] (1,1) arc (0:180:1 and 0.15);
          \draw (-1,1) arc (180:360:1 and 0.15);
          \draw[dashed] (1,-1) arc (0:180:1 and 0.15);
          \draw (-1,-1) arc (180:360:1 and 0.15);
          \draw (1,-1.2) -- (1,1.2);
          \draw[<-] (-1,1.2) -- (-1,-1.2);

        \fill[fill=black] (1,1) circle (1pt);
        \fill[fill=black] (-1,0) circle (1.2pt);
        \fill[fill=black] (1,-1) circle (1pt);
        \node at (1.2,1) {$\sigma(p)$};
        \node at (1.25,-1) {$\sigma^{-1}(p)$};
        \node[lightblue] at (0.7,0.4) {$C(p)$};
        \node at (-1.1,-0.05) {$p$};
        \node at (-1.15,1.1) {$\R$};
        \node at (0,-1.3) {$\S^{n}$};
        \node at (0,0) {$E(p)$};
        \node at (-0.6,-0.6) {$M^-(p)$};
        \node at (-0.6,0.6) {$M^+(p)$};
        \end{tikzpicture}

    \caption{The lightcone of a point $p$ in $\Eintilde$.}
    \label{Figure_Cone_dans_Eintilde}
\end{figure}

\subsection{Proof of Proposition~\ref{prop_F_CO_GHM}}\label{section_preuve_prop_GHMC}
We begin with a standard lemma.
\begin{lem}\label{lemme_assiettes}
    Let $X$ and $Y$ be two metric spaces and let $f:X\to Y$ be a local homeomorphism. Assume that there exists two open subsets $U,V\subset X$ such that $X=U\cup V$, and such that $f\vert _U$ and $f\vert_V$ are injective. If $f(U)\cap f(V)$ is a nonempty connected open subset, then $f$ is injective.
\end{lem}

\begin{proof}
    Let us show that $f(U\cap V)$ is closed in $f(U)\cap f(V)$. Let $(x_k)$ be a sequence of $U\cap V$ such that $f(x_k)$ converges to some $y\in f(U)\cap f(V)$. Let $x\in U$ such that $y=f(x)$. Since $f$ is a local homeomorphism, one has $x_k\to x$ as $k\to \infty$. Similarly, if $x^\prime\in V$ is such that $f(x^\prime)=y$, one has $x_k\to x^\prime$, so $x=x^\prime\in U\cap V$. Therefore $y\in f(U\cap V)$ and $f(U\cap V)$ is closed in $f(U)\cap f(V)$. By connectedness, we obtain $f(U\cap V)= f(U)\cap f(V)$. It follows that $f$ is injective on $X=U\cup V$.
\end{proof}

We will need the following theorem of Smaï, which is stated here with a weaker conclusion than its original formulation in \cite{smai2025futures}.

\begin{thm}[{\cite[Thm. 1.3]{smai2025futures}}]
    \label{thm_de_Rym}
    Let $M$ be a simply connected, conformally flat, globally hyperbolic, $C$-maximal spacetime different from $\Eintilde$. For every $x\in M$, the restriction of $\dev:M\to \Eintilde$ to $I^+(x)$ is injective and has image contained in $M^+(\dev(x))$. In a stereographic projection $M^+(\dev(x))\simeq\R^{1,n}$, this image is convex, past complete and disjoint from a spacelike hyperplane of $\R^{1,n}$. 
\end{thm}

The same statement holds if we replace ``$+$'' by ``$-$'' and ``past complete'' by ``future complete''.

\begin{proof}[Proof of Proposition~\ref{prop_F_CO_GHM}]
    Since $F_M$ is invariant under covering, see Proposition~\ref{prop_naturalité_F_M}, we can assume that $M$ is simply connected. Let $\dev:M\to \Eintilde$ be a developing map for $M$. Let $a,b\in M$ be such that $a\ll b$, and let $D=I(a,b)$ and $U=I^+(a)\cup I^-(b)$. We will show that $F_M$ is continuous on $D$, which will conclude the proof. Let us first show that $F_M$ and $F_U$ coincide on $D$. Since $U\subset M$, the inequality $F_M\leq F_U$ holds on $U$ by Proposition~\ref{proposition_naturalité}. Let $x\in D$ and $v\in T_xM$ be a lightlike vector. Let $\gamma\in \PPL(I,M)$ be a maximal lightlike geodesic such that $\gamma(0)=x$ and $\gamma^\prime(0)=\lambda v$ for some constant $\lambda>0$. Then $\gamma\vert_{[0,1)}$ is a future directed causal curve, so $\gamma([0,1))\subset J^+(x)\subset U$. Similarly, one has $\gamma((-1,0])\subset J^-(x)\subset U$, so $\gamma$ takes values in $U$. In particular 
    $$F_U(v)\leq \sqrt{\rho_I\left(\lambda^{-1}\partial_t,\lambda^{-1}\partial_t\right)}=\frac{2}{\lambda}=F_M(v),$$
    so $F_U$ and $F_M$ coincide on $D$. Therefore, we only need to show that $F_U$ is continuous on $D$.     
    
    Let $p=\dev(a)$ and $q=\dev(b)$. From Theorem~\ref{thm_de_Rym}, the restriction of the developing map to $I^+(a)$ (resp. $I^-(b)$) is injective. Since $\dev(I^+(a))$ and $\dev(I^-(b))$ are causally convex, one has
    $$I\left(p,q\right)\subset \dev(I^+(a))\cap \dev(I^-(b)) \subset I^+(p)\cap I^-\left(y\right) \subset I\left(p,q\right).$$
    Therefore $\dev(I^+(a))\cap \dev(I^-(b))= I\left(p,q\right)$ is a diamond of $\Eintilde$, so $\dev(I^+(a))\cap \dev(I^-(b))$ is connected. Since the restrictions of $\dev$ to $I^+(a)$ and $I^-(b)$ are injective, the restriction of $\dev$ to $U$ is injective by Lemma~\ref{lemme_assiettes}. Let $V=\dev(U)$. Since $F_U=\dev^* F_V$, we only need to show that $F_V$ is continuous on $D^\prime=\dev(D)=I\left(p,q\right)$. 
    Identify $\Eintilde$ with the product $\R\times \S^n$ and let $g=dt^2+d_{\S^n}$. For a future lightlike vector $v\in T_xD^\prime$ , let $x_+(v)$ and $x_-(v)$ denote the endpoints in $\Eintilde$ of the maximal lightlike geodesic generated by $v$ and contained in $V$, and let $\sqrt{2}L_v$ denote the $g$-length of that lightlike segment. Then from Proposition~\ref{Prop_exemple_calcul_F_M}, one has 
    $$F_V(v)=\left(\frac{1+\tan(d_g(x,x_-(v))-L_v/2)^2}{\tan(L_v/2)+\tan(d_g(x,x_-(v))-L_v/2)}
        +\frac{1+\tan(d_g(x,x_+(v))-L_v/2)^2}{\tan(L_v/2)+\tan(d_g(x,x_+(v))-L_v/2)}\right)\frac{\Vert v\Vert_g}{\sqrt{2}},$$
    if $L_v<\pi$ and $F_\Omega(v)=0$ otherwise. In order to show that $F_V$ is continuous, it is sufficient to show that $x_-(v)$ and $x_+(v)$ depend continuously on $v$.
    
     Identify $M^+(p)$ conformally with $\R^{1,n}$ and let $\Omega=\dev(I^+(a))\subset \R^{1,n}$. By Theorem~\ref{thm_de_Rym}, the domain $\Omega$ is convex, past complete and disjoint from at least one spacelike hyperplane $H\subset\R^{1,n}$. 
    Let $(v_k)$ be a sequence of lightlike vectors based at $x_k\in D^\prime$ and converging to $v\in T_x D^\prime$. Since $(x_k)$ converges, we can always find $z\in \Omega$ such that $z\leq x_k$ for all $k\geq 0$. Then one has $z\leq x_+(v_k)$ for all $k\geq 0$, so $x_+(v_k)$ belongs to the compact set $J^+(z)\cap J^-(H)$ since $\Omega$ is contained in the past of $H$. 
    Let $\bar x$ be an accumulation point for $x_+(v_k)$. By convexity of $\Omega$, one has $[x,\bar x)\subset\Omega$. Since $x_+(v_k)\in \partial\Omega$ for all $k\geq 0$, the point $\bar x$ belongs to $\partial \Omega$, so $\bar x=x_+(v)$ by definition. Hence $x_+(v_k)\to x_+(v)$ as $k\to +\infty$, so $v\mapsto x_+(v)$ is continuous. Similarly, the map $v\mapsto x_-(v)$ is continuous.
\end{proof}

\section{Application to conformal Lorentzian geometry}
\label{section_application}
The aim of this section is to prove Theorem~\ref{thm_application}. The proof is based on Proposition \ref{prop_compacité_app_conformes} and two lemmas on the Markowitz pseudodistance that are given below in Section \ref{Section_deux_lemmes_sur_delta_M}. We will also need the following lemma on $\Eintilde$.

\begin{lem}
    \label{lemme_géométrie_eintilde}
    Let $p\in \Eintilde$ and $\Omega\subset E(p)$ such that in the stereographic projection of $p$, the domain $\Omega$ is an Einstein-de Sitter half-space. Let  $\gamma\in\PPL(\R, \Eintilde)$ be non-constant such that $\gamma(\R)\subset\overline{\Omega}$. Then the endpoints of $\gamma$ are $p$ and $\sigma(p)$.
\end{lem}

\begin{proof}
    Since $\gamma$ is parametrized on $\R$, the endpoints of $\gamma$ are at distance $\sqrt{2}\pi$ from each other, see Example~\ref{exemple_para_proj}. Up to a conformal transformation, we can always assume that $p=(0,a)\in\R\times\S^n$ and that $\Omega$ takes the expression 
    $\Omega=\{(t,x)\,\vert\, 0<t<d(x,a)\}$.
    See also Figure~\ref{Figure_Einstein_de_Sitter}. If $q=(t,x)$ is the future endpoint of $\gamma$, then the past endpoint of $\gamma$ is 
    $\sigma^{-1}(q)=(t-\pi,-x)\in \overline{\Omega},$
    so $0\leq t-\pi\leq 0$. Hence $t=\pi$ and $q=\sigma(p)$. 
\end{proof}

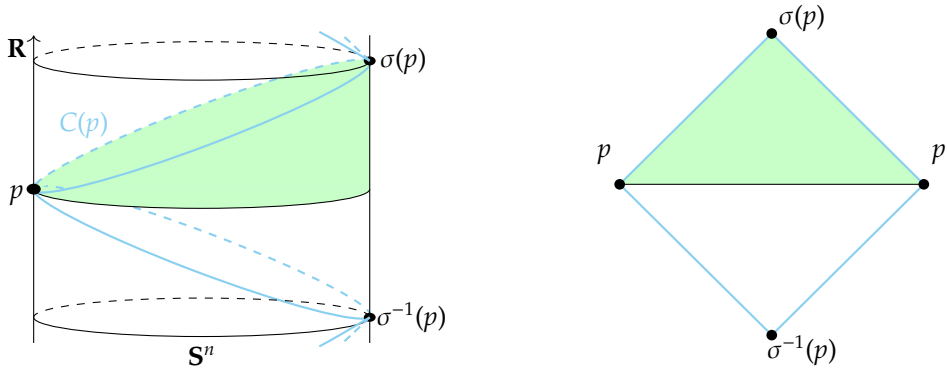
\begin{figure}
    \centering    
    \begin{tabular}{cccccc}
     \begin{tikzpicture}[scale=1.7,xscale=1.3]

        \fill[bovert] (-1,0) -- (1,0) -- (1,1);

                \begin{scope}[thick, lightblue,yshift=-14]
            \begin{scope}[rotate=-26.56]
            \draw[dashed] (1.12,0) arc (0:180:1.12 and 0.15);
          \draw (-1.12,0) arc (180:360:1.12 and 0.15);
        \end{scope}
        \end{scope}

        \draw[fill=bovert] (-1,0) arc (180:360:1 and 0.15);
        
        \begin{scope}[thick, lightblue,yshift=14]
            \begin{scope}[rotate=+26.56]
            \draw[thick,dashed,fill=bovert] (1.12,0) arc (0:180:1.12 and 0.15);
          \draw (-1.12,0) arc (180:360:1.12 and 0.15);
        \end{scope}
        \end{scope}
        \draw[dashed] (1,1) arc (0:180:1 and 0.15);
          \draw (-1,1) arc (180:360:1 and 0.15);
          \draw[dashed] (1,-1) arc (0:180:1 and 0.15);
          \draw (-1,-1) arc (180:360:1 and 0.15);
          \draw (1,-1.2) -- (1,1.2);
          \draw[<-] (-1,1.2) -- (-1,-1.2);

        \fill[fill=black] (1,1) circle (1pt);
        \fill[fill=black] (-1,0) circle (1.2pt);
        \fill[fill=black] (1,-1) circle (1pt);
        \node at (1.2,1) {$\sigma(p)$};
        \node at (1.25,-1) {$\sigma^{-1}(p)$};
        \node[lightblue] at (-0.7,0.5) {$C(p)$};
        \node at (-1.1,-0.05) {$p$};
        \node at (-1.1,1.1) {$\R$};
        \node at (0,-1.3) {$\S^{n}$};
                \begin{scope}[thick, lightblue]
            \draw plot[domain=0.7:1,smooth,variable=\ang]
   ({\ang}, {0.5*cos(90*\ang)+1});
            \draw plot[domain=0.7:1,smooth,variable=\ang]
   ({\ang}, {-0.5*cos(90*\ang)-1});
        \end{scope}

        \begin{scope}[thick, lightblue]
            \draw[densely dashed] plot[domain=0.85:1,smooth, variable=\ang]
   ({\ang}, {0.8*cos(90*\ang)+1});
            \draw[densely dashed] plot[domain=0.85:1,smooth, variable=\ang]
   ({\ang}, {-0.8*cos(90*\ang)-1});
        \end{scope}
        \end{tikzpicture}

            &&&&&
        \begin{tikzpicture}[scale=2]       
    \fill[bovert] (-1,0) -- (1,0) -- (0,1);
        \draw[lightblue,line width = 0.8pt] (-1,0) -- (0,-1) -- (1,0) -- (0,1) -- (-1,0);
        \draw (-1,0) -- (1,0);
      \fill (1,0) circle (1pt);
      \fill (-1,0) circle (1pt);
       \fill (0,1) circle (1pt);
      \fill (0,-1) circle (1pt);
      \draw node at (0.2,1.1) {$\sigma(p)$};
      \draw node at (0.2,-1.1) {$\sigma^{-1}(p)$};
      \draw node at (-1.1,0.2) {$p$};            
      \draw node at (1.1,0.2) {$p$};

        \end{tikzpicture}
        \end{tabular}

    \caption{An Einstein-de Sitter half-space in $\Eintilde$.}
    \label{Figure_Einstein_de_Sitter}
\end{figure}

\subsection{Two lemmas on the Markowitz pseudodistance}\label{Section_deux_lemmes_sur_delta_M}
Choose an orthogonal splitting of $\R^{1,n}=\R\oplus\R^n$ such that the metric $\b$ can be written $\b=-dt^2+d\sigma^2$. For $\varepsilon>0$, we let $\b_\varepsilon$ denote the metric
$$\b_\varepsilon= -(1+\varepsilon)^2dt^2+dx^2.$$
The metric $\b_\varepsilon$ is a Lorentzian metric on $\R^{1,n}$ whose lightcone strictly contains that of the metric $\b$. We write $I_\varepsilon^\pm(x)$ and $J_\varepsilon^\pm(x)$ for the chronal/causal past or future of a point $x\in \R^{1,n}$ with respect to $\b_\varepsilon$.
\begin{lem}
    \label{lemme_ouvert_mark_hyp_eps}
    Let $\varepsilon>0$ and $x\in \R^{1,n}$. Then  $\Omega=\R^{1,n}\setminus J_\varepsilon^-(x)$ is Markowitz hyperbolic. 
\end{lem}

\begin{proof}
    Assume that $x=0$ for simplicity. We denote $\Sigma=\partial J_\varepsilon^-(x)$. This is a topological hypersurface that can be parametrized as the graph of the map $f:\R^n\to \R$ given by
    $f(y)=\frac{-1}{1+\varepsilon}\Vert y\Vert$ for all $y\in \R^n$.
    Let $\gamma$ be an inextensible causal curve of $\R^{1,n}$, which can be written as $\gamma(t)=(t,\alpha(t))$ where $t\in \R$ and $\Vert\alpha^\prime(t)\Vert\leq 1$, up to reparametrization. The curve $\gamma$ intersects $\Sigma$ at a time $t\in \R$ if and only if $t$ is a fixed point of the function $f\circ \alpha$.
    Now $f\circ \alpha$ is a contracting mapping, hence has a unique fixed point $t\in\R$. Thus $\gamma$ intersects $\Sigma$ at a unique point, that is, $\Sigma$ is a Cauchy hypersurface of $\R^{1,n}$. 
    For $y\in \Omega$ and $v\in T_y\Omega$ lightlike, let $p_v$ denote the intersection point of the lightlike line generated by $v$ and $\Sigma$. 
    Let $\Vert\cdot\Vert$ be a norm on $\R^{1,n}$. Then from Example~\ref{exemple_formula_F_M}, one has
    $$F_\Omega(v)=\frac{\Vert v\Vert}{\Vert y-p_v\Vert}.$$
    Since $v\mapsto p_v$ is continuous, the infinitesimal functional $F_\Omega$ is continuous and positive, so $\Omega$ is Markowitz hyperbolic by Proposition \ref{prop_cara_complete_hyperbolic}. 
\end{proof}

\begin{lem}
    \label{lemme_classique}
    Let $\Omega$ be a quasi-homogeneous domain of a connected Markowitz hyperbolic conformal manifold $M$. If every $g\in\Conf(\Omega)$ extends to a conformal transformation of $M$, then $\Omega=M.$
\end{lem}

\begin{proof}
    By connectedness, we only need to show that $\Omega$ is closed in $M$, or equivalently that $\partial\Omega=\emptyset$. Assume by contradiction that there exists a point $x\in \partial \Omega$ and let $(x_k)$ be a sequence of $\Omega$ converging to $x$. By quasi-homogeneity, there exists a sequence $(g_k)\in\Conf(\Omega)$ and a sequence of points $(y_k)$ converging to some $y\in \Omega$ such that $g_k(y_k)=x_k$. Since $x_k\to x$, one has $\delta_M(x_k,\partial\Omega)\to 0$, but since $\partial \Omega$ is $\Conf(\Omega)$ invariant, one has 
    $$\delta_M(x_k,\partial\Omega)=\delta_M(g_k(y_k),g_k(\partial\Omega))=\delta_M(y_k,\partial\Omega)\to\delta_M(y,\partial\Omega)>0,$$
    a contradiction.
 \end{proof}

\subsection{Proof of Theorem~\ref{thm_application}}
Let $H\subset\R^{1,n}$ be a spacelike hyperplane and let   $\Omega\subset I^+(H)$. Assume that $\Conf(\Omega)$ acts cocompactly on $\Omega$. We let $G=\Conf^+(\Omega)$ denote the group of time orientation preserving conformal diffeomorphisms of $\Omega$. Since $G$ has index at most $2$ in $\Conf(\Omega)$, it acts on $\Omega$ with compact quotient. We fix $p\in \Eintilde$ and $\st:\R^{1,n}\to E(p)$ a stereographic projection at $p$. Let $\widetilde\Omega=\st(\Omega)$. All causal sets are to be understood in $\R^{1,n}$. We will use the notation $\widetilde{I^+}(x)$ to denote the future of a point $x\in\Eintilde$, and similarly for the past.
    
    \smallskip
    
    \emph{\textbf{Step 1:} For every $x\in \Omega$, there exists $a\in \partial\Omega\cap I^-(x)$ such that $J_\varepsilon^-(a)\cap\overline{\Omega}=\{a\}$ for some $\varepsilon>0$.} Let $x\in \Omega$. Then the set $\Kcal=J^-(x)\cap \overline{\Omega}$ is compact, since it is contained in the bounded set $J^-(x)\cap J^+(H)$. The compactness of $\Kcal$ implies that we can find a point $a\in \partial\Omega\cap I^-(x)$ such that 
    $$r=\vert\b(a-x,a-x)\vert=\max_{y\in\Kcal}\vert\b(y-x,y-x)\vert.$$
    Let $S=\{y\in I^-(x)\,\vert\,\b(y-x,y-x)=-r\}$. For $\varepsilon>0$ small enough, one has $J_\varepsilon^-(a)\cap J^+(H) \subset J^-(S)$, see Figure~\ref{figure_point_a}. Hence, by decreasing $\varepsilon>0$ again if necessary, one has $J_\varepsilon^-(a)\cap\overline{\Omega}=\{a\}$.
        \begin{figure}
    \centering
         \begin{tikzpicture}

    \def\a{2}
    \def\xplus{4.7}
    \def\xmoins{4}
    \fill[bovert, opacity=1]  (-\xmoins,\a)
    -- (-\xmoins,2)
    -- (-2.5,0.9) 
    -- (0,0)  
    -- (1,-0.2) 
    -- (4,1.3)
    -- (\xplus,2)
    -- (\xplus,\a) ;
    \draw  (-\xmoins,2)
    -- (-2.5,0.9) 
    -- (0,0)  
    -- (1,-0.2) 
    -- (4,1.3)
    -- (\xplus,2) ;
    \node at (4,1) {$\Omega$};

    \def\xun{-0.3}
    \def\xdeux{1.5}
    \def\para{3.5}
    \draw[thick] (\xun-\para,\xdeux-\para) -- (\xun,\xdeux) -- (\xun+\para,\xdeux-\para);
    \draw[fill=blue,opacity=0.1] (\xun-\para,\xdeux-\para) -- (\xun,\xdeux) -- (\xun+\para,\xdeux-\para) -- cycle;
    \node at (\xun+\para+0.5,\xdeux-\para+0.2) {$J^-(x)$};
    \fill (\xun,\xdeux) circle (2pt) node [above left] {$x$};

    \draw[->, >=Stealth] (1.6,-1.3) to [out=30, in=-130] (3.1,-0.65) node [above right] {$J^-_\varepsilon(a)$};

    \draw (-4,-1) -- (4,-1);
    \node at (-3.8,-0.8) {$H$};
    \node at (-2.8,-1.65) {$S$};

    \def\r{2.2}
    \draw [thick,samples=200,variable=\x,domain=-3.17:3.17] plot ({\xun+\x}, {\xdeux-sqrt(\r+\x*\x)});


    \fill (0,0) circle (2pt) node [above right] {$a$};
    
    \draw[thick] (-\a,-\a) -- (0,0) -- (\a,-\a);
    \fill[pattern=north east lines] (-\a,-\a) -- (0,0) -- (\a,-\a);

    \def\B{0.8}
    \draw[thick,densely dashed] (-\a/\B,-\a) -- (0,0) -- (\a/\B,-\a);

\end{tikzpicture}
       \caption{Construction of the point $a\in \partial\Omega$ in step 1 of the proof of Theorem~\ref{thm_application}.}
    \label{figure_point_a}
\end{figure}
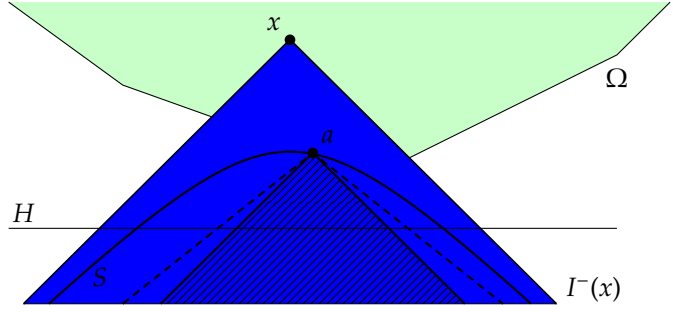
    
    \smallskip
    
    \emph{\textbf{Step 2:} the domain $\Omega$ is a causally convex subset of $\R^{1,n}$.} Since the stereographic projection is a conformal diffeomorphism onto its image, it is sufficient to prove that $\widetilde\Omega$ is a causally convex domain of $\Eintilde$. Let $U$ denote the interior of the intersection of all causally convex domains of $\Eintilde$ containing $\widetilde\Omega$. Since causal convexity is stable under infinite intersection, the domain $U$ is causally convex and contains $\widetilde\Omega$. Also, since $E=\st(I^+(H))$ is causally convex, it contains $U$ by definition. Now $E$ is conformally equivalent to an Einstein-de Sitter half-space, so it is Markowitz hyperbolic, see Section~\ref{section_misner}. We deduce that $U$ is Markowitz hyperbolic. Now, by construction, the domain $U$ is $\Conf(\widetilde\Omega)$-invariant, hence by Liouville's theorem, one has 
    $$\Conf(\widetilde\Omega)\subset\Conf(U).$$
    From Lemma~\ref{lemme_classique}, we deduce that $\widetilde\Omega=U$, that is $\widetilde\Omega$ is causally convex. 
    
    \smallskip
    
    \emph{\textbf{Step 3:} the domain $\Omega$ is a past, that is, there exists $b\in \partial\widetilde\Omega\subset\Eintilde$ such that $\widetilde\Omega\subset \widetilde{I^-}(b)$.} Let $x\in \Omega$ and let $a\in \partial\Omega\cap I^-(x)$ be such that $J_\varepsilon^-(a)\cap\overline{\Omega}=\{a\}$ for some $\varepsilon>0$. If we denote $U=\R^{1,n}\setminus J_\varepsilon^-(a)$, then $\Omega\subset U$ by construction. For simplicity, we fix $a=0$. Let $f_k:\Omega\to U$ be the sequence of conformal maps defined for all $y\in \Omega$ and $k\geq 0$ by 
    $$f_k(y)=2^ky.$$
    Let $x_k=2^{-k}x$ for all $k\geq 0$. Since $\Omega$ is causally convex, one has $x_k\in I(a,x)\subset\Omega$ for all $k\geq 0$. Since $\Omega$ is quasi-homogeneous, there exists a sequence $(g_k)\in\Conf(\Omega)$ such that $(y_k)=(g_k(x_k))$ stays in a compact subset of $\Omega$. Let $h_k=f_k\circ g_k^{-1}:\Omega\to U$. Then, for all $k\geq 1$, one has 
    $h_k(y_k)=f_k(x_k)=x$ and
    $$I(a,2x)\subset I(a,f_k(x))\subset f_k(I(a,x_k))\subset f_k(\Omega)= h_k(\Omega).$$
    By Lemma~\ref{lemme_ouvert_mark_hyp_eps}, the domain $U$ is Markowitz hyperbolic, and by Liouville's theorem, the maps $h_k$ are injective for all $k\geq 0$. Therefore we can apply Proposition~\ref{prop_compacité_app_conformes} to deduce that a subsequence of $(h_k)$ converges to a conformal map $h:\Omega\to U$. 
    Now, for all $k\geq 0$, one has 
    $$I(a,2^kx) = f_k(I(a,x))\subset f_k(\Omega)=h_k(\Omega).$$
    Since $I(a,2^kx)$ is an increasing sequence of open subsets, we deduce that 
    $$I^+(a)=\cup_kI(a,2^kx)\subset h(\Omega).$$
    In particular, the point $\sigma(p)$ is the endpoint in $\Eintilde$ of any future timelike half-line. If we still denote $h$ for the extension of $h$ to $\Eintilde$, one has $\sigma(p)\in \partial h(\widetilde\Omega)$ and $h(\widetilde\Omega)\subset \widetilde{I^-}(\sigma(p))$. Hence, one has if we let $b=h^{-1}(\sigma(p))$, then $b\in \partial\widetilde\Omega$ and $\widetilde\Omega\subset \widetilde{I^-}(b)$.
    
    \smallskip
    
    \emph{\textbf{Step 4:} in the stereographic projection of $q=\sigma^{-1}(b)$, the domain $\widetilde\Omega$ is future complete.} First, we explain why the stereographic projection of $q$ contains $\widetilde\Omega$ and why, in that stereographic projection, the domain $\widetilde\Omega$ is still disjoint from at least one spacelike hyperplane. Since $\widetilde\Omega\subset E(p)\subset \widetilde{I^-}(\sigma(p))$, one has $b\leq \sigma(p)$. Hence $q\leq p$ and $\sigma^{-1}(q)\leq \sigma^{-1}(p)$, so 
    $$\widetilde\Omega\subset \widetilde{I^-}(b)\cap \widetilde{I^+}(\sigma^{-1}(q))=E(q).$$
    Let $\Sigma=\overline{\st(H)}$. Then $q\leq p\in \Sigma$, hence $q\in \widetilde{I^-}(\Sigma)$. In particular, we can always translate $\Sigma$ in the past to another Möbius sphere $\Sigma^\prime$ such that $q\in\Sigma^\prime$ and $\Sigma^\prime\subset I^-(\Sigma)$. Thus, in the stereographic projection of $q$, the domain $\widetilde\Omega$ is contained in the future of the spacelike hyperplane $H^\prime$ defined by $\Sigma^\prime$ in that chart. In order to keep notations simple, we will still write $\Omega$ for the image of $\widetilde\Omega$ in the chart of $q$ and we will write $H$ for some spacelike hyperplane in that chart disjoint from $\Omega$. 
    
    From Step 3, there is some $x\in \Omega$ such that $I^+(x)\subset \Omega$. Let $y\in \Omega$ and fix a future timelike vector $v\in \R^{1,n}$. Since $x+kv\in I^+(x)$, causal convexity implies that $I(y,x+kv)\subset \Omega$ for all $k\geq 0$. Now 
    $$I^+(y)=\cup_{k\geq 0} I(y,x+kv),$$ 
    so $I^+(y)\subset \Omega$. Hence $\Omega$ is a future complete domain of $\R^{1,n}$.

    \smallskip
    
    \emph{\textbf{Step 5:} in the stereographic projection of $q$, the domain $\widetilde\Omega$ is convex.}
    Note that $b$ is the unique point of $\Eintilde$ that both belongs to $\partial\widetilde\Omega$ and contains $\widetilde\Omega$ in its past. Indeed, if $c$ is another such point, then $b\in \partial\widetilde\Omega\subset J^-(c)$ and similarly $c\in J^-(b)$. Hence $b\leq c\leq b$, so $c=b$. Hence $b$ is fixed by the group $\Conf^+(\widetilde\Omega)$, so $q=\sigma^{-1}(b)$ is also $\Conf^+(\widetilde\Omega)$-invariant. Therefore 
    $$ G=\Conf(\Omega)\subset\Conf(\R^{1,n})=\Sim(\R^{1,n}).$$
    In particular, the group $G$ acts by affine transformations on $\R^{1,n}$. Let $U=\Conv(\Omega)$ be the convex envelope of $\Omega$ in $\R^{1,n}$. Since $\Omega\subset I^+(H)$ and since $I^+(H)$ is convex in $\R^{1,n}$, one has $U\subset I^+(H)$, so $U$ is Markowitz hyperbolic. Since $G$ is a group of affinity, the convex envelope $U$ is $G$-invariant. Since $U$ is connected, it follows from Lemma~\ref{lemme_classique} that $\Omega=U$, so $\Omega$ is convex. 

    At this point, the domain $\Omega$ is convex, future complete and disjoint from a spacelike hyperplane. These domains carry a natural projection map $r:\Omega\to \partial\Omega$ defined as follows. For every $x\in \Omega$, the point $r(x)$ is the unique point of $J^-(x)\cap\overline{\Omega}$ maximizing the time separation from $x$. Such a point exists and is unique, and the projection map $r:\Omega\to \partial\Omega$ is $C^1$, see \cite[Prop. 4.3]{Bonsante}. The image $S=r(\Omega)$ is called the \emph{singularity in the past} of $\Omega$. For $x\in \Omega$, we write $H_x$ for the spacelike hyperplane through $r(x)$ and orthogonal to $x-r(x)$, which is a supporting hyperplane of $\Omega$ at $x$.
    
    \smallskip

    \emph{\textbf{Step 6:} in the stereographic projection of $q$, the domain $\widetilde\Omega$ is a cone at (at least) one point of the singularity in the past.} Let $x\in \Omega$. We define $f_k:\Omega\to I^+(H_x)$ by 
    $$f_k(y)=2^k(y-r(x))+r(x).$$
    Then $f_k(y_k)=x$ where $y_k=2^{-k}(x-r(x))+r(x)$. As before, we can find a sequence $(z_k)$ converging to some $z\in \Omega$ and we can find $(g_k)\in \Conf^+(\Omega)$ such that $g_k(z_k)=y_k$. The sequence $h_k=f_k\circ g_k$ satisfies $h_k(z_k)=x$ and $x\in I^+(r(x))\subset h_k(\Omega)$ for all $k\geq 0$. By Proposition~\ref{prop_compacité_app_conformes}, there exists a subsequence of $(h_k)$ converging to a conformal map $h:\Omega\to I^+(H_x)$. Since $f_k$ and $g_k$ are both affine maps for all $k\geq 0$, the limit $h$ is also affine. Now, by construction, the images $h_k(\Omega)=f_k(\Omega)$ are a sequence of increasing convex sets, and the union $\cup_{k\geq 0}h_k(\Omega)=h(\Omega)$ is a convex cone at $r(x)$. Hence $\Omega$ is a cone at $s=r(h^{-1}(x))$.
    
    \smallskip

    \emph{\textbf{Step 7:} in the stereographic projection of $q$, the domain $\widetilde\Omega$ is a HB-domain.} We first treat the case where $\Omega$ contains no line. We show that the singularity in the past consists of a single point. Assume that there exists a point $s^\prime=r(z)\in S$ distinct from $s$. As before we define $f_k:\Omega\to I^+(H_z)$ by $f_k(y)=2^k(y-s^\prime)+s^\prime.$ The sequence $(f_k)$ may be reparametrized into a sequence $(h_k)$ that converges to a similarity $h:\Omega\to I^+(H_z)$ whose image equals $h(\Omega)=\cup_kf_k(\Omega)$. By construction, one has 
    $$\cup_kf_k(s+\R_{\geq 0}(s^\prime-s))=\Span(s,s^\prime),$$
    hence $\partial h(\Omega)$ contains a line, and so does $\partial\Omega$, a contradiction. Therefore $S=\{s\}$ and $\Omega=I^+(S)=I^+(s)$ is a HB-domain of index $0$.
    
    We now turn to the general case. Since $\Omega$ is convex and contains no lightlike line, one can find an orthogonal decomposition $\R^{1,n}=\R^{1,n-\ell}+\R^\ell$ such that $\Omega=\Omega^\prime + \R^\ell$, where $\Omega^\prime\subset\R^{1,n-\ell}$ is a convex domain containing no line. Since $\Omega$ is future complete and disjoint from a spacelike hyperplane, the domain $\Omega^\prime$ is convex, future complete and disjoint from a spacelike hyperplane of $\R^{1,n-\ell}$. Also since $G$ preserves the splitting $\R^{1,n}=\R^{1,n-\ell}+\R^\ell$, the induced action of $G$ on $\Omega^\prime$ is by means of similarities of $\R^{1,n-\ell}$. This action is quasi-homogeneous because the image of a compact $\Kcal\subset \Omega$ intersecting every $G$-orbit under the orthogonal projection $\R^{1,n}\to \R^{1,n-\ell}$ is a compact of $\Omega^\prime$ intersecting every $G$-orbit. By the preceding case, the domain $\Omega^\prime$ consists of the future of a single point $s$. Now $\Omega=I^+(F_\ell)$ where $F_\ell =\R^\ell+s$, so $\Omega$ is a HB-domain. 
    
    \smallskip

    \emph{\textbf{Last step:} from conformal equivalence to equality}. Note that at this point, we have only shown that in some stereographic projection, the domain $\Omega$ is a HB-domain, because we replaced $\Omega$ with its image in the stereographic projection of $q$ in step 4. We now show that, when $\ell\geq 1$, then $\Omega$ is actually equal to $\Omega_\ell$, up to composition by a similarity.
    We identify $\R^{1,n}$ with $E(p)$ for some $p\in \Eintilde$. 
    Let $f\in \Conf(\Eintilde)$ such that $f(\Omega_\ell)=\Omega$. If $\ell\geq 1$, then there is a Lorentzian 2-plane $P\subset\R^{1,n}$ intersecting $F_\ell$ along a spacelike line. The boundary of $P\cap\Omega_\ell$ in $\Eintilde$ contains two lightlike geodesics joining $p$ to $\sigma(p)$ that are projectively parametrized on $\R$. Let $\Delta$ be the image of one of them. Then $f(\Delta)$ is contained in $\overline\Omega$, hence in $\partial I^+(H)\subset\Eintilde$. By Lemma~\ref{lemme_géométrie_eintilde}, the segment $f(\Delta)$ joins the two conjugate points $p$ and $\sigma(p)$. In particular, 
    $$f(\partial \Delta)=\partial(f(\Delta))=\{p,\sigma(p)\}=\partial\Delta,$$ so $\{p,\sigma(p)\}$ is invariant under $f$. Since both $\Omega_\ell$ and $\Omega$ are contained in $E(p)$ but disjoint from $E(\sigma(p))$, we must have $f(p)=p$, so $f$ is a similarity of $\R^{1,n}$. Hence $\Omega=f(\Omega_\ell)$ is a HB-domain. 

    When $\ell=0$, the domain $\Omega_\ell$ is conformally equivalent to a diamond in $\R^{1,n}$. If $D=I(u,v)\subset\R^{1,n}$ is such a diamond and $f:D\to \Omega$ is a conformal diffeomorphism, then $f(v)\in I^+(f(u))\cap I^-(\sigma(p))$. Up to a similarity, there are three different configurations for $f(v)$: either $f(v)\in \R^{1,n}$, or $f(v)=\sigma(p)$, or $f(v)\in J(p,\sigma(p))$ with $f(v)\not \in\{p, \sigma(p)\}$. This leads to three models of domains that are conformally equivalent to $\Omega_0$ that are disjoint from a spacelike hyperplane, see Figure~\ref{Figure_Modèles_affines_des_diamants}.  \qed

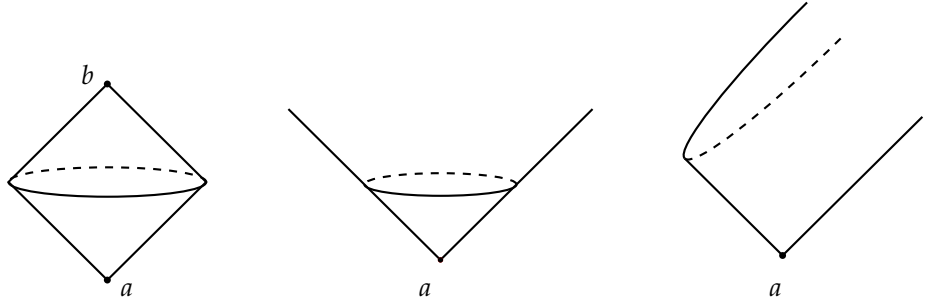
\begin{figure}
    \centering
    $$\begin{array}{ccccccc}
        \begin{tikzpicture}[scale=1.3,rotate=-90]       
        \begin{scope}[rotate = 90,line width = 0.8pt]
          \draw (-1,0) arc (180:360:1 and 0.15);
        \end{scope}
        \begin{scope}[rotate = 90,line width = 0.8pt]
          \draw[dashed] (1,0) arc (0:180:1 and 0.15);
        \end{scope}
        \draw[line width = 0.8pt] (-1,0) -- (0,-1) -- (1,0) -- (0,1) -- (-1,0);
      \fill (1,0) circle (1pt);
      \fill (-1,0) circle (1pt);
      \draw node at (1.1,0.2) {$a$};
      \draw node at (-1.1,-0.2) {$b$};

        \end{tikzpicture}
        & & &

            \begin{tikzpicture}[scale=1,rotate=90]
            \draw[line width = 0.8pt] (2,2) -- (0,0) -- (2,-2); 
         \fill[red,line width = 0.8pt] (0,0) circle (1pt);
          \begin{scope}[rotate = -90]
              \draw[dashed,line width = 0.8pt] (1,1) arc (0:180:1 and 0.15);
          \draw[line width = 0.8pt] (-1,1) arc (180:360:1 and 0.15);
          \end{scope}
          \draw node at (-0.4,0.2) {$a$};
          \fill (0,0) circle (1pt);
          
        \end{tikzpicture}
        & & & 
            \begin{tikzpicture}[scale=1.3,rotate=135]       
        \begin{scope}[rotate = 90,line width = 0.8pt]
          \draw (1,0) arc (0:-90:2 and 0.25);
        \end{scope}
        \begin{scope}[rotate = 90,line width = 0.8pt]
          \draw[dashed] (1,0) arc (0:90:2 and 0.25);
        \end{scope}
        \draw[line width = 0.8pt] (0,1) -- (-1.4,1) -- (-1.4,-1);
      \fill (-1.4,1) circle (1pt);
      \draw node at (-1.6,1.3) {$a$};

        \end{tikzpicture}

    \end{array}$$
    \caption{The three models for a diamond in a stereographic projection that are disjoint from a spacelike hyperplane. From left to right, these domains are: a chronal diamond of $\R^{1,n}$, the future of a point $a\in \R^{1,n}$ and the intersection of the future of a point $a\in \R^{1,n}$ with the past of a degenerate hyperplane.}
    \label{Figure_Modèles_affines_des_diamants}
\end{figure}

\begin{proof}[Proof of Corollary~\ref{corollaire_de_lapplication}]
    Let $\dev:\widetilde M\to \Ein^{1,n}$ and $\hol:\pi_1(M)\to \PO(2,n)$ be a developing map and a holonomy homomorphism for $M$, respectively. Let $\Kcal\subset \widetilde M$ be a compact subset intersecting every $\pi_1(M)$-orbit. Since the developing map is $\pi_1(M)$-equivariant, the compact $\Ccal=\dev(\Kcal)\subset \Omega=\dev(M)$ intersects every $\Gamma=\hol(\pi_1(M))$-orbit of points of $\Omega$. Hence $\Omega$ is quasi-homogeneous. 

    Assume that the image of $\dev$ is contained in an Einstein-de Sitter half-space $H$. By Theorem~\ref{thm_application}, the domain $\Omega$ is conformally equivalent to a HB-domain of some index $\ell\geq 0$. In particular, the domain $\Omega$ carries a conformally invariant Riemannian metric $g_\ell^+$, see Section \ref{section_misner}. The pullback metric $g=\dev^*g_\ell^+$ is a $\pi_1(M)$-invariant Riemannian metric on $\smash{\widetilde M}$. In particular, the metric $g$ quotients to a  Riemannian metric on $M$, which is complete by the Hopf-Rinow theorem since $M$ is compact. Therefore $g$ is a complete Riemannian metric on $\smash{\widetilde M}$ and $\dev:\widetilde M\to \Omega$ is a local isometry. It follows that the developing map is a covering map onto its image, hence it is a diffeomorphism since $\Omega\simeq\Omega_\ell$ is contractible.

    Choose an orthogonal splitting $\R^{1,n}=\R^\ell+\R^{1,n-\ell}$ so that $\Omega_\ell=\R^\ell+I^+(0)$. Let $X$ be the timelike vector field on $\Omega_\ell$ defined by 
    $X(u,v)=v$
    in these coordinates. Then $X$ generates a flow $(\phi_t)$ given by 
    $\phi_t(u,v)=(u,e^tv).$
    In particular, if $\dvol_{\R^{1,n}}$ denotes the volume form of the flat metric on $\R^{1,n}$, then 
    $$(\phi_t)^*\dvol_{\R^{1,n}}=(e^t)^{n-\ell+1}\dvol_{\R^{1,n}}.$$
    Now, the volume form of the metric $g_\ell^+=\frac{1}{\Vert v\Vert^2}\b$ is given by $\dvol_{g_\ell^+}=\frac{1}{\Vert v\Vert^{n+1}}\dvol_{\R^{1,n}}$. We deduce that 
    $$(\phi_t)^*\dvol_{g_\ell^+} =\frac{(e^t)^{n-\ell+1}}{(e^t)^{n+1}\Vert v\Vert^{n+1}}\dvol_{\R^{1,n}}
        =e^{-t\ell}\dvol_{g_\ell^+}.$$
    Now the flow $\phi_t$ centralizes $\Conf(\Omega_\ell)$, hence it pulls back to a flow $(\tilde\phi_t)$ centralizing $\pi_1(M)$. In particular, we obtain a flow $(\phi_t)$ on the closed manifold $M$. Now $M$ has finite volume with respect to $g=\dev^*g_\ell^+$, and for all $t\in\R$, one has
    $$\text{Vol}(M)=\int_M\dvol_{g}=\int_M(\phi_t)^*\dvol_{g}=e^{-t\ell}\text{Vol}(M).$$
    Hence $\ell=0$. Therefore, the only possibility is that $\smash{\widetilde M}$ is conformally diffeomorphic to $I^+(0)\subset \R^{1,n}$. Now $I^+(0)$ is conformally equivalent to $(\R\times\H^n,-dt^2+\ghyp)$ and under that identification, the conformal group of $I^+(0)$ is equal to $G=\Iso(\R)\times\Iso(\H^n)$. Now $M\simeq (\R\times\H^n)/\Gamma$ for some cocompact lattice $\Gamma<G$. The lattice $\Gamma$ virtually splits as a product $\Z\times \Gamma^\prime$, where $\Gamma^\prime<\Iso(\H^n)$ is a cocompact lattice (see for instance \cite{morris2015}), so $M$ is finitely covered by a product $\S^1\times (\H^n/\Gamma^\prime)$. 
\end{proof}

\subsection{Declaration} This research was funded in part by the Luxembourg National Research
Fund (FNR), grant reference O24/18936913/RiGA.  For the purpose of open access, and
in fulfilment of the obligations arising from the grant agreement, the author has applied
a Creative Commons Attribution 4.0 International (CC BY 4.0) license to any Author
Accepted Manuscript version arising from this submission.

\bibliographystyle{alpha}
\bibliography{biblioplus}

\end{document}